\documentclass[reqno,11pt]{amsart}
 
\usepackage{amssymb,latexsym,amsmath,amsthm}
\usepackage[left=1.25in,top=1in,right=1.25in,bottom=1in]{geometry} 
\usepackage{fancyhdr, lastpage} 
\usepackage{leftidx}
\usepackage{color}
\usepackage{graphicx}
\usepackage{mathrsfs}
\usepackage{comment}
\usepackage[mathscr]{euscript}   
\usepackage[toc,page]{appendix}
\usepackage[linkcolor=blue,colorlinks=true]{hyperref}

\begin{document}

\sloppy
\renewcommand{\theequation}{\arabic{section}.\arabic{equation}}
\thinmuskip = 0.5\thinmuskip
\medmuskip = 0.5\medmuskip
\thickmuskip = 0.5\thickmuskip
\arraycolsep = 0.3\arraycolsep

\newtheorem{theorem}{Theorem}[section]
\newtheorem{lemma}[theorem]{Lemma}
\newtheorem{corollary}[theorem]{Corollary}
\newtheorem{prop}[theorem]{Proposition}
\newtheorem{definition}[theorem]{Definition}
\newtheorem{remark}[theorem]{Remark}
\renewcommand{\thetheorem}{\arabic{section}.\arabic{theorem}}
\newcommand{\prf}{\noindent{\bf Proof.}\ }
\def\prfe{\hspace*{\fill} $\Box$

\smallskip \noindent}

\def\be{\begin{equation}}
\def\ee{\end{equation}}
\def\bea{\begin{eqnarray}}
\def\eea{\end{eqnarray}}
\def\beas{\begin{eqnarray*}}
\def\eeas{\end{eqnarray*}}

\newcommand{\loc}{\operatorname{loc}}
\newcommand{\R}{\mathbb R}
\newcommand{\N}{\mathbb N}
\newcommand{\T}{\mathbb T}
\newcommand{\K}{\mathbb S}
\newcommand{\leftexp}[2]{{\vphantom{#2}}^{#1}{#2}}
\newcommand{\eqdef}{\overset{\mbox{\tiny{def}}}{=}}
\newcommand{\snorm}[1]{\|#1\|}

\newcommand{\todo}[1]{\vspace{5 mm}\par \noindent
\marginpar{\textsc{ \hspace{.2 in}   \textcolor{red}{ To Fix}}} \framebox{\begin{minipage}[c]{0.95
\textwidth} \tt #1 \end{minipage}}\vspace{5 mm}\par}

\def\k{\kappa}
\def\l{\lambda}
\def\g{\partial}
\def\E{\mathcal{ E}}
\def\D{\mathcal{D}}
\def\t{\bar{\partial}} 
\def\div{\mbox{div}}
\def\curl{\mbox{curl}}
\def\open#1{\setbox0=\hbox{$#1$}
\baselineskip = 0pt
\vbox{\hbox{\hspace*{0.4 wd0}\tiny $\circ$}\hbox{$#1$}}
\baselineskip = 11pt\!}
\def\pA{\leftexp{+}{\!\!\!\!\!\!A}}
\def\mA{\leftexp{-}{\!\!\!\!\!\!A}}
\def\A{\leftexp{\kappa}{\!\!\!\!\!\!A}}
\def\a{\leftexp{\kappa}{\!\!\!a}}
\def\kA{\leftexp{\kappa}{\!\!\!\tilde{A}}}
\def\w{\leftexp{\kappa}{\!\!\!w}}
\def\kw{\leftexp{\kappa}{\!\!\!\tilde{w}}}
\def\qn{\open{q}}
\def\vp{(v^+)}
\def\vm{(v^-)}
\def\newe{{\scriptstyle \mathcal{E} }}
\def\newd{{\scriptstyle \mathcal{D} }}
\def\eee{{{\scriptscriptstyle  \mathcal{E} }}}
\def\ppp{{{\scriptscriptstyle  \partial }}}
\def\ppp{{{\scriptscriptstyle  \Gamma}}}
\def\aabb{{{\scriptscriptstyle  a,b }}}
\def\ddd{{\scriptscriptstyle \mathcal{D} }}
\def\newb{{\scriptstyle \mathcal{B} }}
\def\qok{{{}^\k\hspace{-0.02in}q_0}}
\def\rok{{{}^\k\hspace{-0.02in}r_0}}
\def\muok{{{}^\k\hspace{-0.02in}\mu_0}}
\def\abark{{{}^\k\hspace{-0.05in}\bar{A}}}
\def\ak{{{}^\k\hspace{-0.05in}A}}
\def\ako{{{}^\k\hspace{-0.05in}A_0}}
\def\ao{{A_0}}
\def\psibark{{{}^\k\hspace{-0.02in}\bar{\Psi}}}
\def\psik{{{}^\k\hspace{-0.02in}\Psi}}
\def\psiko{{{}^\k\hspace{-0.02in}\Psi_0}}
\def\Qko{{{}^\k\hspace{-0.02in}Q_0}}
\def\Uko{{{}^\k\hspace{-0.02in}U_0}}
\def\Jbark{{\bar{J}_\k}}
\def\labark{{{}^\k\hspace{-0.02in}\bar{a}}}
\def\qs{{{}^s\hspace{-0.02in}q}}

\title{Local well-posedness and Global stability of the Two-Phase Stefan problem}

\author{M. Had{\v z}i\'c, G. Navarro, S. Shkoller}

\begin{abstract}
The two-phase Stefan problem describes the temperature distribution in a homogeneous medium undergoing a phase transition such as ice melting 
to water. This is accomplished by solving the heat equation on a time-dependent domain, composed of two regions separated by an a priori unknown moving boundary which is transported by the difference (or jump) of the normal derivatives of the temperature in each phase. 
We establish local-in-time well-posedness and a global-in-time stability result for arbitrary sufficiently smooth domains and small initial temperatures. 
To this end, we develop a higher-order energy with natural weights adapted to the problem and combine it with Hopf-type inequalities. This extends the previous work by Had\v{z}i\'{c} \& Shkoller \cite{mHsS2013,mHsS2015} on the one-phase Stefan problem to the setting of two-phase problems, and 
simplifies the proof significantly.  
\end{abstract}

\date{July 1, 2016}
\maketitle

\tableofcontents

\section{Introduction to the problem}

\subsection{Problem formulation and the reference domain}\label{se:form}
We consider the local and global well-posedness and  interface regularity of solutions to the classical {\em two-phase Stefan problem}, describing the evolving interface, separating  a freezing liquid and a melting solid.
The temperature of the liquid-solid phase $p^{\pm}(t,x)$ 
and the {\it a priori unknown}  {\it moving interface} $\Gamma(t)$ must satisfy the following system of equations:
\begin{subequations}
\label{E:stefan}
\begin{alignat}{2}
p^{\pm}_t-\Delta p^{\pm}&=0&&\ \text{ in } \ \Omega^\pm(t)\,\label{eq:heat}\\
[\g_np]^{\pm}&=-V_{\Gamma(t)}&& \ \text{ on } \ \Gamma(t)\,\label{eq:neumann}\\
p^+=p^-&=0&& \ \text{ on } \ \Gamma(t)\,\label{eq:dirichlet}\\
p^{\pm}(0,\cdot)=p^{\pm}_0\,, & \ \Gamma(0)=\Gamma_0&&\,,\label{eq:initial}
\end{alignat}
\end{subequations}
\begin{figure}[htbp]\label{FIGURE}
\begin{center}
\includegraphics[scale = 0.6]{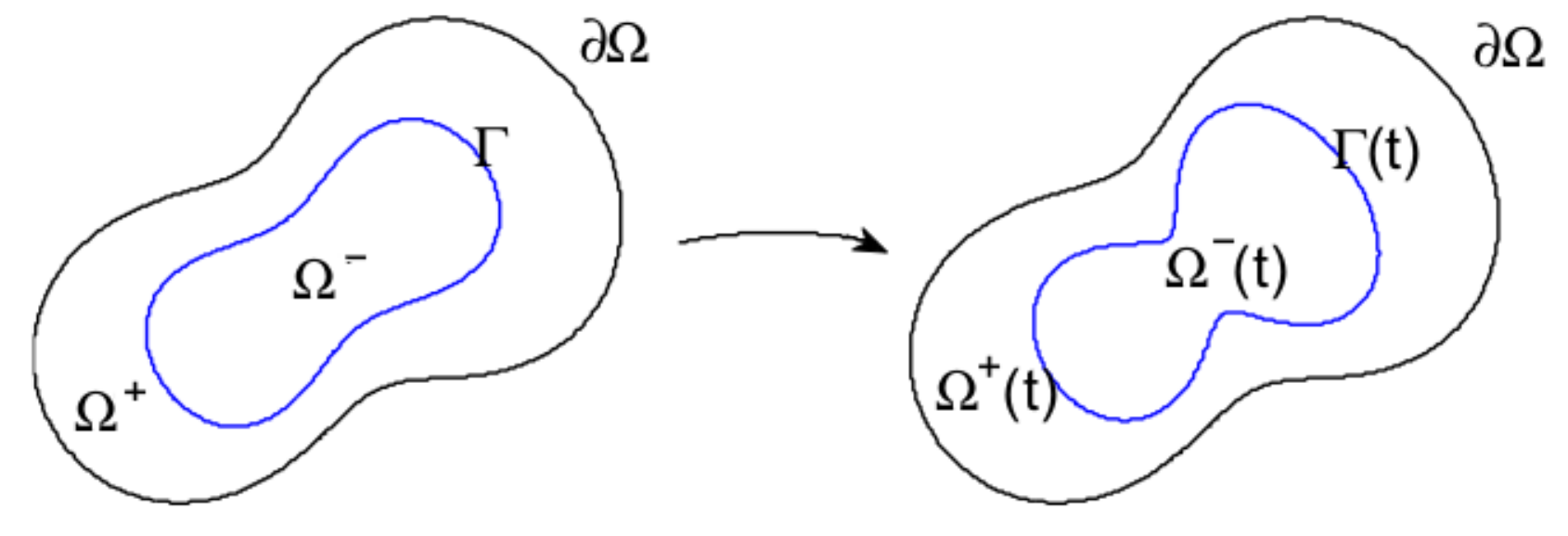}
\end{center} \caption{{\footnotesize The two-phase Stefan problem.  Displayed on the left side of the figure are the reference domains $\Omega^\pm$ and reference
interface $\Gamma$.  The time-dependent domains $\Omega^\pm(t)$ and the moving interface $\Gamma(t)$ are shown on the right side
of the figure.  The domain $\Omega^-(t)$ denotes the solid phase, while the domain $\Omega^+(t)$ denotes the liquid phase.}}\label{fig1}
\end{figure}
where for each time $t \in [0,T]$,  $\Omega^+(t)$ and $\Omega^-(t)$ denote  two evolving open  and bounded domains as shown
in Figure \ref{fig1}, and $\Gamma(t)$ denotes the moving  interface
 separating $\Omega^+(t)$ and $\Omega^-(t)$, so that  $\Gamma(t)=\overline{\Omega^-(t)} \cap \overline{\Omega^+(t)}$.  
 
\begin{definition}[The domains $\Omega$ and $\Omega^\pm(t)$]\label{def-domains}
For $d \ge 2$, we denote by $\Omega \subset \mathbb{R}  ^d$, a
 fixed, open, and
 bounded set such that
$$
 \Omega = \overline{\Omega^-(t) } \cup\Omega^+(t) \,,
$$
as shown Figure \ref{fig1}.  We assume that the fixed boundary $\partial \Omega$ is $C^ \infty $.
\end{definition}

The equations \eqref{eq:heat} model temperature diffusion in the bulk $\Omega^\pm(t)$,  
while the interface jump condition \eqref{eq:neumann} states that the jump in temperature gradients evolves the interface; 
that is, $[\g_np]^{\pm}:= \g_np^+-\g_np^-=\nabla (p^+-p^-)\cdot n$ on $\Gamma(t)$, where
 $n( \cdot ,t)$ denotes the outward  unit normal on $\Gamma(t)$ (pointing into  $\Omega^+(t)$), and 
$V_{\Gamma(t)}$ denotes the speed or  normal velocity of the interface $\Gamma(t)$.
Note that the freezing of the liquid and the melting of the solid occur at a constant temperature $p=0$ as seen from the 
Dirichlet boundary condition~\eqref{eq:dirichlet}.   
Initial conditions are prescribed in (\ref{eq:initial}):  the initial interface $\Gamma_0$
and the initial temperature functions  $p_0^-$ and $p_0^+$ are specified.    

Herein, we shall, for simplicity,  consider the two-dimensional Stefan problem $d=2$, although all of our methods easily extend in a straightforward manner to the case that $d \ge 3$.
No convexity assumptions  are made on the initial interface
 $\Gamma_0$, but we shall assume that $\Gamma_0$ is diffeomorphic to the unit circle $ \mathbb{S}  ^1$.

\begin{remark} 
Surface tension effects can be included as well  by replacing \eqref{eq:dirichlet} with 
\be\label{eq:dirichletgamma}
p^\pm=\gamma H_{\Gamma(t)} \ \text{ on } \ \Gamma(t),
\ee
where $\gamma\ge0$ is the surface tension parameter and $H_{\Gamma(t)}$ 
is the mean curvature of $\Gamma(t)$.  Herein, we shall study the case that $\gamma=0$.
We shall also consider the two-dimensional problem $d=2$, although all of our results extend in a straightforward manner to the case that $d \ge 3$.\end{remark}

\subsection{Specifying a smooth reference  interface $\Gamma$ and reference domains $\Omega^\pm$}\label{SS:smoothreference}   
In order to describe our initial interface $\Gamma_0$, we employ an
$H^6$-class
 parameterization  $z_0 : \mathbb{S}  ^1 \to \Gamma_0$,  where $ \mathbb{S}  ^1$ is identified with the period $[0, 2 \pi]$.  
To construct a smooth reference interface, we consider a $C^ \infty $ nearby interface $\Gamma_ \sigma $ which is constructed 
by smoothing $z_0( \theta)$, $\theta \in \mathbb{S}  ^1$ using a standard mollification approach.
 For $ \sigma >0$ taken sufficiently small, we define the  convolution operator  $ \Lambda _ \sigma $ as follows: 
\begin{equation}\label{E:Lambdadefinition}
\Lambda _ \sigma  z_0 (\theta)= \int_{ \mathbb{R}} \rho_\sigma  (\theta -\vartheta) z_0(\vartheta) d\vartheta \,,
\end{equation}

where $\rho_\sigma(\theta) = \sigma ^{-1} \rho( \theta/ \sigma  )$, and $\rho$ is the standard mollifier on $ \R$, given by
\begin{equation*}
\rho(x) = \left\lbrace \begin{array}{cc} Ce^{\frac{-1}{1-|x|^2}} & |x|\leq 1\,,\\
0 & |x|> 1\,. \end{array}\right.
\end{equation*}
\begin{definition}[$C^ \infty$ reference interface and domains] For $ \sigma >0$ taken sufficiently small and fixed, we
set
\begin{align*}
z^\sigma_0(\theta) &= \Lambda_\sigma z_0(\theta),
\end{align*}
and we define the nearby $C^ \infty$  curve 
$$\Gamma := z_\sigma ( \mathbb{S}  ^1) \,. $$
The curve $\Gamma$ is the reference interface, and define the reference domain $\Omega^- $ to be the open set enclosed 
by $\Gamma $, and we set 
$\Omega^-  = \Omega - \overline{\Omega^+}$.
\end{definition} 

The initial interface $\Gamma_0$ is in the normal bundle of $\Gamma$;  hence there exists a signed height
function  $h_0\in H^6(\Gamma_ \sigma )$ such that
$$
z_0(\theta) = z^\sigma_0 (\theta) + h_0(z_\sigma(\theta))N_\sigma(z^\sigma_0(\theta)),\ z^\sigma_0(\theta) \in \Gamma_\sigma,
$$ 
where $N_\sigma(z^\sigma_0(\theta))$ is the  unit normal vector to $\Gamma $ (pointing into  $\Omega^+$) at the point 
$z^ \sigma _0(\theta)$.
It follows that the initial height function $h_0$ has amplitude of order $\sigma$, and that $ h_0 \to 0$  as $\sigma\rightarrow 0$.

As time evolves, if the interface
$\Gamma(t)$ stays in the normal bundle of $ \Gamma_ \sigma $, then for each time $t$, we can define the corresponding
signed height function $h(t,z_\sigma)$ as follows:
\begin{equation}
\Gamma(t) := \{ y\ \vert\ y= z^\sigma_0 + h(t,z^\sigma_0)N_\sigma(z^\sigma _0),\ z^\sigma_0 \in\Gamma_\sigma\},
\end{equation} 
with the initial condition 
$$
h(0,z_\sigma) = h_0(z_\sigma).
$$

We note, that while it is not essential, it is convenient to use the $ C^ \infty $ curve $\Gamma_ \sigma $ as the reference interface.  This allows
us to use the normal bundle of $ \Gamma_ \sigma $ with a $C^ \infty $ unit normal vector field.  If we had instead worked with the initial
interface $\Gamma_0$ as
the reference interface, we would have been forced to use a different (from the normal) transverse vector field to define the height function due to 
the limited regularity of $\Gamma_0$ and the fact that the regularity of the normal  would have a  one derivative loss.

For notational clarity, we shall henceforth drop the explicit  dependence on $\sigma$ in our parametrization, and write $z_0(\theta)$ for
 $ z^\sigma_0 (\theta)$.


\subsection{Notation}\label{ss:Notation}
We denote the identity map $x\mapsto x$ by $e$ and the identity $(2\times2)$-matrix $(\delta_{ij})_{i,j=1,2}$ by $\text{Id}.$
 A constant C is a generic constant and may change from line to line, and we write $X\lesssim Y$ to denote $X\leq CY$. 
Similarly, we use the notation $P(\cdot)$ to denote a generic polynomial of the form $P(x)=Cx^p$, with $p\geq 1$, and the constants C and $p$ may also change from line to line.

We use $\nabla = (\partial_{x_1}, \partial_{x_2})$ to denote the gradient operator.
For $i=1,2$,  we shall abbreviate partial differentiation
of a function $f$ as   $f,_k = \frac{\partial f}{\partial x^k}$, and for time-differentiation we let $F_t:=\g_t F$. 
We shall use the Einstein summation convention, where repeated indices are summed from $1$ to $2$.
Furthermore, given a function $F(t,x)$, we shall often write $F(t)$ instead of $F(t,\cdot)$ and $F(0)$ instead of $F(0,x)$.

To deal with lower-order terms in energy estimates, we shall use the abbreviation $\text{l.o.t.}$ for spacetime integrals
in which the  integrand has sufficiently few derivatives so as to be bounded by a simple application of H\"{o}lder's  inequality; this is made precise in\eqref{E:lotdefinition}.\\

It will be useful for some estimates to set the following notation for functions evaluated at time $t=0$:
\begin{equation*}
\Psi_0^\pm:= \Psi^\pm(0),\ \psiko^\pm := \psik^\pm(0),\ \ao^\pm:=A^\pm(0) = [\nabla\Psi_0^\pm]^{-1},\ \ako^\pm := \ak^\pm(0) = [\nabla \psik_0^\pm]^{-1},
\end{equation*}	
and we can define as well the respective differential operators 
\begin{align*}
\Delta_{\Psi_0^\pm} f &:= {\ao^\pm}^i_j (\ao^k_j f,_k),_i,\\
\Delta_{\psiko^\pm} f &:= {\ako^\pm}^i_j ({\ako^\pm}^k_j f,_k),_i,\\
(\nabla_{\Psi_0^\pm} f)^j &:= {\ao^\pm}^i_j f,_i,\\
(\nabla_{\psiko^\pm} f)^j &:= {\ako^\pm}^i_j f,_i,\\
\end{align*}
which are the generalizations of the Laplacian and the gradient respectively for a scalar function $f$ over the regions $\Omega^\pm$. For a vector field $F$, we defined the matrices,
\begin{align*}
[\nabla_{\Psi_0^\pm} F]^i_j &:= {\ao^\pm}^k_j F^i,_k,\\
[\nabla_{\psiko^\pm} F]^i_j &:= {\ako^\pm}^k_j F^i,_k.
\end{align*}

\subsection{Sobolev norms} For any $s\geq 0$ and given functions $F^\pm:\Omega^\pm\to\R,\ \varphi:\Gamma\to \R$, we denote the norms in the standard Sobolev spaces $H^s(\Omega^\pm),\ H^s(\Gamma)$ as,
\begin{equation*}
\|F^\pm\|_{s} := \|F^\pm\|_{H^s(\Omega^\pm)},\ \ |\varphi|_s := \|\varphi\|_{H^s(\Gamma)},
\end{equation*}
where $\|F^\pm\|_s$ is either $\|F^+\|_{H^s(\Omega^+)}$ or $\|F^-\|_{H^s(\Omega^-)}$ depending on which domain we are considering.

If $F:[0,T]\times\Omega\to\R,\ \varphi:[0,T]\times\Gamma\to\R$ are given time-dependent functions, then
\begin{align*}
\|F\|_{L^2_tH^s} &:= \left(\int_0^t \|f(s)\|^2_{H^s(\Omega)}ds\right)^{1/2},\\
\|F\|_{L^\infty_tH^s}&:= \sup\limits_{0\leq s\leq t}\ \|F(s)\|_{H^s(\Omega)},\\
|\varphi|_{L^2_tH^s} &:= \left(\int_0^t |\varphi(s)|^2_{H^s(\Gamma)}ds\right)^{1/2},\\
|\varphi|_{L^\infty_tH^s} &:= \sup\limits_{0\leq s\leq t}\ |\varphi(s)|_{H^s(\Gamma)}.
\end{align*}

For given weight functions $W^\pm:\Omega^\pm\to\R$ such that $W^\pm >0$, we define the {\bf weighted $L^2$ norm} as
\begin{align}
\|F^\pm\|_{L^{2,W^\pm}}^2& := \int_{\Omega^\pm}|F^\pm(s,x)|^2W^\pm\,dx.
\end{align}

\subsection{Tangential derivatives} \label{SS:TANGENTIAL}
For a given $0<\nu \ll1$, we define a smooth cut-off function $\mu: \bar{\Omega} \rightarrow \mathbb{R}_+$ satisfying
\begin{equation}\label{eq:mudefinition}
\mu(x) \equiv 1\ \text{ if } \text{dist}(x,\Gamma\cup\g\Omega) \leq \nu \text{ and } \mu(x)\equiv 0\ \text{ if }\ \text{dist}(x,\Gamma\cup\g\Omega) \geq 2\nu.
\end{equation}
This will allow us to localize the analysis to a neighborhood of the interface, wherein we define \emph{tangential derivative} as $\t f =  \nabla f \cdot \tau$, where $\tau$ is the smooth extension of the tangent vector to $\Gamma$ into that neighborhood. In the case of functions defined solely on $\Gamma$, we define the tangential derivative naturally as $\t g = \frac{1}{\|z_0'\|}\frac{d}{d\theta} g(z_0(\theta))$, where $z_0(\theta)$ is the parametrization of $\Gamma$ described in Section \ref{SS:smoothreference}. 

\subsection{Steady states}\label{S:SS}
Let $\bar\Gamma$ be any given closed $C^1$-curve separating $\Omega$   into two connected components $\Omega^+$ and $\Omega^-.$ Then the triple
$(u^+,u^-,\Gamma)\equiv(0,0,\bar\Gamma)$ constitutes a steady state solution to~\eqref{E:stefan}. The space of steady states is therefore infinite-dimensional 
and NOT parametrized by finitely many parameters. 
The main goal of this article is to understand the nonlinear stability of these steady states.


\subsection{Pulling-back to the reference domains $\Omega^\pm$} \label{SS:FIXEDDOMAIN}
To develop a well-posedness theory for (\ref{E:stefan}), we pull-back the equations to the reference domains $\Omega^\pm$.  It is convenient
to construct harmonic diffeomorphisms. Hence, for each $t \in [0,T]$, 
 we define the 
diffeomorphisms $\Psi^\pm( \cdot , t): \Omega^\pm \to \Omega^\pm(t)$ as the solution of
\begin{subequations}\label{E:definition_Psi}
\begin{align}
\Delta \Psi^\pm &= 0\ \text{ in } \Omega^\pm,\\
\Psi^\pm(t,x) &= x + h(t,x)N(x)\ \text{ on } \Gamma,\label{E:definition_Psi_boundary}\\
\Psi^+ &= \text{e}\ \text{ on } \g\Omega,  \label{E:FIXEDBOUNDARY}
\end{align}
\end{subequations}
where $e$ is the identity map on $\g\Omega$. 
Elliptic estimates show that for $k \ge 1$,
\begin{equation}\label{E:Psiboundarycontrol}
\|\Psi^\pm-e\|_{6.5} \leq C(|h|_{6}) \,.
\end{equation}  
When $|h|_6 \le  \epsilon \ll 1$, the inverse function theorem, together with the Sobolev embedding
theorem, show that $\Psi^\pm( \cdot , t): \Omega^\pm \to \Omega^\pm(t)$ are $H^{6.5}$-class diffeomorphisms. Of course, we could have used any Sobolev space $H^s$, in place of $H^{6.5}$, but  our analysis will make use of the latter.


We next introduce our physical variables set on the fixed reference domains $\Omega^\pm$.  We set
\begin{subequations}
\begin{alignat}{2}
q^\pm &:= p^\pm(t,\cdot) \circ \Psi^\pm,\\
v^\pm &:= \nabla p^\pm(t,\cdot)\circ \Psi^\pm,\\
A^\pm &:= (\nabla\Psi^\pm)^{-1},\label{E:A_definition}\\
w^\pm &:= \Psi_t^\pm.
\end{alignat}
\end{subequations}
In the parlance of fluid dynamics, the mappings $\Psi^\pm$ are often called Arbitrary Lagrangian Eulerian (ALE) coordinates.
The Laplace operator in ALE coordinates is given by
$$\Delta_{\Psi^\pm}:= A^\pm{}^i_j\g_i(A^\pm{}^k_j \g_k)\,.$$

Therefore, on the fixed reference domains $\Omega^\pm$,  the Stefan problem~\eqref{E:stefan} has the following form:
\begin{subequations}
\label{E:ALE} 
\begin{alignat}{2}
q^\pm_t - \Delta_{\Psi^\pm} q^\pm &= - v^\pm\cdot w^\pm\ & \text{ in }& \Omega^\pm,\label{E:ALEheat}\\
v^\pm + A^\pm{}^T \nabla q^\pm &= 0\  &\text{ in } &\Omega^\pm,\label{E:ALEv}\\
q^\pm(t,x) &= 0\ &\text{ on }& \Gamma,\label{E:ALEdirichlet}\\
h_t  &= \left[v\cdot \tilde{n}\right]^+_- &\text{ on }& \Gamma, \label{E:ALEneumann}\\
v^+\cdot {\bf N}^+ &= 0\ &\text{ on } & \g\Omega,\label{E:topneumann}\\ 
q^\pm(0,x) &= q_0^\pm(x) &\text{ on }& \{t=0\}\times \Omega^\pm,\label{E:ALEinitial}\\
h &= h_0\ &\text{ on }& \{t=0\}\times \Gamma.
\end{alignat}
\end{subequations} 
The motion of the interace $\Gamma(t)$ is given by
equation \eqref{E:ALEneumann}, which is an equivalent form of  \eqref{eq:neumann}, since the speed  $V_\Gamma(t) $ of $\Gamma(t)$
is equal to $ \Psi_t\cdot n = h_t N\cdot n$, where $n$ is the outward normal vector to $\Omega^-(t)$ to be defined below, and $\tilde{n}:=\frac{n}{n\cdot N}$. 
Observe that the matrices $A^\pm$ depend on $\Psi^\pm$, and the  $\Psi^\pm$ are extensions of $e + hN$  obtained from (\ref{E:definition_Psi}).

Notice that using the description of the reference interface $\Gamma$ as a curve $z(\theta)$, the moving interface $\Gamma(t)$ is 
described as $y(\theta) = z(\theta) + h(t,z(\theta))N(z(\theta))$, and so the normal vector $n$ is given by
\begin{equation}\label{E:moving_normal}
n(t,y(\theta)) = \frac{-\t h \tau + (1+H(\theta)h(t,z(\theta)))N}{\sqrt{(\t h)^2 + (1+H(\theta)h)^2}}
\end{equation}
where $H(\theta) := \frac{z_2'z_1''-z_1'z_2''}{\|z'\|^3}$ is the signed curvature of $\Gamma$ at the point $z(\theta)$ and $\tau$ is defined in Section~\ref{ss:Notation}.

\subsection{Higher-order Norm used for our Analysis}\label{S:norms}
\subsubsection{Local well-posedness theory}
We will develop the local-in-time well-posedness  theory with respect to the following norm:
\begin{equation}
\mathcal{S} (t) := \underbrace{\newe^+(t) + \newe^-(t) +\newe^\ppp_{\loc}(t) }_{L^ \infty\text{-in-time-control}} 
\ + \  \underbrace{\int_0^t \left(\newd^+(s) + \newd^-(s) + \newd^\ppp_{\loc}(s)\right)\,ds}_{L^ 2\text{-in-time-control}}\,, \label{S-local}
\end{equation}
where 
\begin{align*}
\newe^\pm(t) &:= \sum\limits_{l=0}^3 \|\g_t^l q^\pm\|_{L^\infty_t H^{6-2l}(\Omega^\pm)}^2 + \|\t^{5-2l}\g_t^lv^\pm\|^2_{L^\infty_t L^2(\Omega^\pm)}\,, \\
\newd^\pm(t) &:= \sum\limits_{l=0}^3 \|\g_t^l q^\pm(t)\|_{H^{6.5-2l}(\Omega^\pm)}^2 + \|\t^{6-2l}\g_t^l v^\pm(t)\|_{L^2(\Omega^\pm)}^2 \,, \\
\newe^\ppp_{\loc}(t) &:= \sum\limits_{l=0}^3 \sup\limits_{0\leq s\leq t} |\g_t ^ l h(s)|^2_{H^{6-2l}(\Gamma)}\,, \\
\newd^\ppp_{\loc}(t) &:=  \sum\limits_{l=0}^2 |\g_t^{l+1} h|_{H^{5-2l}(\Gamma)}^2 \,.
\end{align*}

\subsubsection{Global well-posedness and nonlinear stability theory}
We will develop the global-in-time well-posedness and our nonlinear stability theory with respect to the following norm:
\begin{equation}
S(t) := \underbrace{\newe^+(t) + \newe^-(t) + \newe^\ppp(t)}_{L^ \infty\text{-in-time-control}} 
\ + \  \underbrace{\int_0^t \left(\newd^+(s) + \newd^-(s) + \newd^\ppp(s)\right)\,ds}_{L^ 2\text{-in-time-control}} \  
+ \ \underbrace{E_\beta^+(t) +E_\beta^-(t)}_{\text{exponential-decay-in-time}}\,, \label{S-global}
\end{equation}
where 
\begin{align*}
\newe^\pm(t) &:= \sum\limits_{l=0}^3 \|\g_t^l q^\pm\|_{L^\infty_t H^{6-2l}(\Omega^\pm)}^2 + \|\t^{5-2l}\g_t^lv^\pm\|^2_{L^\infty_t L^2(\Omega^\pm)}\,, \\
\newd^\pm(t) &:= \sum\limits_{l=0}^3 \|\g_t^l q^\pm(t)\|_{H^{6.5-2l}(\Omega^\pm)}^2 + \|\t^{6-2l}\g_t^l v^\pm(t)\|_{L^2(\Omega^\pm)}^2 \,, \\
\newe^\ppp(t) &:= \sum\limits_{l=0}^3 \sup\limits_{0\leq s\leq t} e^{(-\lambda_1 + \eta)s}|\g_t ^ l h(s)|^2_{H^{6-2l}(\Gamma)}\,, \\
\newd^\ppp(t) &:=  e^{(-\lambda_1 + \eta)t}\sum\limits_{l=0}^2 |\g_t^{l+1} h|_{H^{5-2l}(\Gamma)}^2 \,,\\
E_{\beta}^\pm(t) &:= e^{\beta^\pm t} \sum\limits_{l=0}^2\|\g_t^lq^\pm(t)\|^2_{H^{4-2l}(\Omega^\pm)}\,,
\end{align*}
where 
$\lambda_1 = \min\{\lambda_1^+,\lambda_1^-\}$ is the smaller of the two first eigenvalues $\lambda_1^\pm$ of the \emph{Dirichlet-Laplacian} on the reference domains $\Omega^\pm$, 
 $\eta>0$ is a small constant relative to $\lambda_1$ to be fixed later, and
 \be\label{E:BETAPMDEFINITION}
\beta^\pm = 2\lambda_1^\pm -\eta \,.
\ee 
%

\begin{remark}
The definition of our higher-order energy function requires the definition of the terms $\g_t^l q$ and $\g_t^l h$ at time $t=0$. These are computed using the time-differentiated version of the equation \eqref{E:ALEheat} at time $t=0$. For example,
\begin{align}
q_1^\pm := q_{t}^\pm(0) &= \Delta_{\Psi^\pm(0)}q_0^\pm + A^\pm(0)^\top \nabla q^\pm_0\cdot \Psi^\pm_t(0).\label{E:qt_atzero} 
\end{align}
The other time derivatives follow the same procedure. The terms $\g_t^l h\big|_{t=0}$  follow similarly
 by means of taking time derivatives of the evolution equation~\eqref{E:ALEneumann} and restricting it at time $t=0$. We define the functions,
\begin{subequations}\label{E:h_and_derivativesatzero}
\begin{align}
g_1 &:= h_t(0) = [v(0)\cdot\tilde{n}(0)]^+_-\label{E:definition_h_t0}\\
g_2 &:= h_{tt}(0) = [v_t(0)\cdot\tilde{n}(0)]^+_- + [v(0)\cdot \tilde{n}_t(0)]^+_-\\
g_3 &:= h_{ttt}(0) = [v_{tt}(0)\cdot \tilde{n}(0)]^+_- + 2[v_t(0)\cdot \tilde{n}_t(0)]^+_- + [v(0)\cdot \tilde{n}_{tt}(0)]^+_-,
\end{align}
\end{subequations}
where $v(0),v_t(0),v_{tt}(0)$ are computed from equation \eqref{E:ALEv} by taking time derivatives and restricting to $t=0$. Notice that they depend only on $h_0$ and $q_0^\pm$. The derivatives of $\tilde{n}(0)$ are given by,  
\begin{align*}
\tilde{n}(0) &:= (N-\tau \frac{\t h_0}{(1+H h_0)}),\\
\tilde{n}_t(0) &:= \tau\left(\frac{\t g_1}{(1+Hh_0)} - \frac{\t h_0 H g_1}{(1+Hh_0)^2}\right),\\
\tilde{n}_{tt}(0) &:= \tau\left(\frac{\t g_2}{(1+Hh_0)} + \frac{\t g_1 H g_1}{(1+Hh_0)^2} - \frac{(\t g_1 H g_1  + \t h_0 H g_2)}{(1+Hh_0)^2} + \frac{2\t h_0 H^2 (g_1)^2}{(1+Hh_0)^3}\right).
\end{align*}
\end{remark}

\subsection{Compatibility conditions}\label{SS:compatibility}

Since we study the twice time-differentiated problem, in order to ensure the continuity of the solution in time, we have to impose certain compatibility conditions. Since $q^\pm_t\vert_\Gamma = q^\pm_{tt}\vert_\Gamma \equiv 0$, by restricting time derivatives of equation \eqref{E:qt_atzero} to $\Gamma$ at time $t=0$, we obtain that $q_0^\pm$ must satisfy, 

\begin{subequations}\label{E:compatibility}
\begin{align}
\Delta_{\Psi_0^\pm}q_0^\pm &= \nabla_{\Psi_0} q_0^\pm\cdot N g_1 \text{ on } \Gamma,\label{compatibility1}\\
-\Delta^2_{\Psi_0^\pm}q_0^\pm &= -\ao^i_l\Psi_t(0)^l,_s\ao^s_j(\ao^k_j q_0,_k),_i - \ao^i_j(\ao^k_l\Psi_t(0)^l,_s\ao^s_j q_0,_k),_i\nonumber \\
&\qquad -\ao^i_k\Psi_t(0)^k,_l\ao^l_j q_0,_i N^j g_1 + \nabla_{\Psi_0}(\Delta_{\Psi_0}q_0+\nabla_{\Psi_0}q_0\cdot\Psi_t(0))\cdot N g_1\nonumber\\
&\qquad + \nabla_{\Psi_0}q_0\cdot N g_2 + \Delta_{\Psi_0}(\nabla_{\Psi_0}q_0\cdot\Psi_t(0)), \label{compatibility2}
\end{align}
and over $\g\Omega$, 
\begin{align}\label{E:compatibility_gOmega}
\nabla_{\Psi_0}q_0^+\cdot {\bf N}^+ &= 0,\\
\nabla_{\Psi_0} (\Delta_{\Psi_0}q_0^+)\cdot {\bf N}^+ &= -\nabla_{\Psi_0}(\nabla_{\Psi_0}q_0^+\cdot\Psi_t(0))\cdot {\bf N}^+ + \ao^i_l\Psi_t^l(0),_s \ao^s_j q_0^+,_i {{\bf N}^+}^j. \label{E:COMP_GOMEGA}
\end{align}
\end{subequations}

\begin{remark}
Notice that, when the initial data $h_0\equiv 0$, equation \eqref{compatibility1} takes the more familiar form 
\begin{equation}
\Delta q_0^\pm = \g_N q^\pm_0 [\g_N q_0]^+_- \text{ on } \Gamma, \label{E:compatibility1_identity}
\end{equation}
and the functions $q_0^\pm = \pm\ dist(x,\Gamma)$, satisfy \eqref{E:compatibility1_identity}.  
\end{remark}

\begin{remark} Unlike the analysis of  \cite{mHsS2013}, herein,  the matrix $A^\pm(0)\neq \operatorname{Id}$, but it is nevertheless a very small perturbation of the identity matrix. We have the following estimate:
\begin{equation*}
\|A^\pm(0)- \operatorname{Id} \|_{s} \lesssim |h_0|_{s-0.5} \lesssim \sigma,
\end{equation*}
which follows from the boundary condition \eqref{E:definition_Psi_boundary} restricted at time $t=0$. Recall that $h_0$ is defined on the
 smooth reference curve $\Gamma$, and the graph of $h_0$ defines the initial interface $\Gamma_0$ in the normal bundle over $\Gamma$. 
\end{remark}

\subsection{Non-degeneracy or  Rayleigh-Taylor stability condition}
To ensure the local-in-time well-posedness of the Stefan problem, we impose the well-known nondegeneracy condition that has been used by
Meirmanov \cite{amM1992extra},    Pr{\"u}ss, Saal, \& Simonett \cite{jPjSgS2007}, Had\v{z}i\'{c} \& Shkoller \cite{mHsS2013,mHsS2015} and other
authors; specifically, we make the following requirement on the initial temperature function: 
\begin{equation}\label{E:localTaylorsign}
\g_Nq_0^\pm \geq \delta >0 \text{, uniformly on } \Gamma,
\end{equation}
for some constant $\delta>0.$ Condition~\eqref{E:localTaylorsign} is the classical Rayleigh-Taylor sign condition that naturally appears in the context of many free boundary  and moving interface problems in fluid dynamics as a stability condition for well-posedness~\cite{CoShfreesurfEuler}, wherein the function $q_0$ is the initial pressure function rather than temperature function.
   For the one-phase water waves equations with an interface that does not self-intersect, 
(\ref{E:localTaylorsign}) was shown to always hold in \cite{Wu1999} using the Hopf lemma.  For the incompressible Euler equations with
vorticity, it is essential to chose an initial velocity profile that provides a pressure function satisfying (\ref{E:localTaylorsign}), for if 
(\ref{E:localTaylorsign}) holds, then the free boundary incompressible Euler equations are well-posed (see, for example, 
\cite{CoShfreesurfEuler} and the references therein), while if (\ref{E:localTaylorsign}) does not hold, then the problem is ill-posed, as shown in
 \cite{Eb1987}.   In the setting of compressible flows with the so-called physical vacuum boundary, condition (\ref{E:localTaylorsign}) is
equivalent to the sound speed of the gas vanishing as the square-root of the distance function to the vacuum boundary, and is also required for
well-posedness as shown in \cite{CoSh2011, CoSh2012a}.   This condition also appears in both the Hele-Shaw and Muskat
problems (see, for example, \cite{ChCoSh2012a} and 
\cite{ChGrSh2015}).   In all of these problems, the natural control of the second-fundamental form
of the moving interface $\Gamma(t)$ can be obtained in a somewhat similar fashion (at least for short-time), and we will discuss this further in Section~\ref{S:METHODOLOGY}.

To ensure the global-in-time nonlinear stability of the steady state solutions of the Stefan problem~\eqref{E:stefan} described in Section~\ref{S:SS}, we shall demand natural sign assumptions on the initial temperatures in the liquid
and the solid phases,  respectively:
\be\label{E:SIGNASSUMPTION}
q_0^+>0 \  \text{ in } \ \Omega^+ \ \ \text{ and } \ q_0^-<0 \  \text{ in } \ \Omega^-.
\ee
In addition to this, given some universal constant $C^*>0,$ we consider initial temperature distributions $q_0^\pm$ satisfying
\begin{equation}\label{E:globalTaylorsign}
k(q_0^\pm):=\frac{\inf_{x\in\Gamma}\g_N q^\pm_0(x)}{\int_{\Omega^\pm} q_0^\pm \varphi_1^\pm dx} \geq C^*,
\end{equation}
where $\varphi_1^\pm\ge 0$ denotes the first eigenfunction of the {\em Dirichlet-Laplacian} on $\Omega^\pm.$ 
The quantity $k(q_0)$ is dimensionless and it is invariant under scaling, i.e. $k(\varepsilon q_0)=k(q_0),$ $\varepsilon\neq 0.$
We also denote
\begin{equation}
c_1^\pm : = \int_{\Omega^\pm} q_0^\pm \varphi_1^\pm dx,\label{E:c_1}
\end{equation}
i.e.  $c_1^\pm$
is the projection of $q_0^\pm$ onto the first eigenfunction of the Dirichlet-Laplacian.
Observe that~\eqref{E:globalTaylorsign} implies~\eqref{E:localTaylorsign}.
Notice that under the sign assumptions~\eqref{E:SIGNASSUMPTION} the parabolic Hopf Lemma implies that $\g_Nq^\pm(t) > 0$ for some period of time $0<t<T$;
however, as $t\rightarrow 0^+$, there is no uniformity on this lower bound.
Condition~\eqref{E:globalTaylorsign} is designed to ensure a uniform lower bound on $\g_N q^\pm(t)$ as $t\rightarrow 0$ in a certain quantified manner, involving the quantity $c_1^\pm.$
This will be crucial in obtaining a sharp lower bound for $\g_N q^\pm(t),$ $t>0,$ 
which is used in the proof of the global-in-time stability. 

\section{Main results}
Our first result is a local well-posedness theorem in Sobolev spaces.
\begin{theorem}[Local well-posedness]\label{T:main} With $\Omega$ and $\Omega^\pm$ as given in Definition \ref{def-domains},  and with
$(q_0^\pm,h_0)$ satisfying the initial data compatibility conditions \eqref{compatibility1}--\eqref{E:COMP_GOMEGA}, the Rayleigh-Taylor sign condition~\eqref{E:localTaylorsign} and 
\[
\mathcal{S} (0) < \infty\,,
\]
where $ \mathcal{S} (t)$ is defined in (\ref{S-local})\,,
there exists a time $T>0$ and a universal constant $C>0$, such that there exists a unique solution to the two-phase Stefan problem~\eqref{E:ALE}, the map $t\to S(t)$ is continuous on $[0,T]$, and the solution verifies the following estimate:
$$
\mathcal{S} (t) \leq C\mathcal{S}(0) \ \ \text{ for all } \ \ t\in[0,T] \,.
$$
\end{theorem}


Having established short-time existence for arbitrarily  large data, we next consider the nonlinear stability of equilibria.
To do so, we introduce the following dimensionless quantity:
\begin{equation}\label{ssK}
K(q_0^\pm): = \frac{\|q_0^\pm\|_4}{\|q_0^\pm\|_0},
\end{equation} 
which is invariant under the rescaling $q_0^\pm \mapsto \varepsilon q_0^\pm.$ Note that $K>1$, since, in the standard definition of the norm in $H^4(\Omega^\pm)$, the $L^2(\Omega^\pm)$ norm is contained in the sum.
\begin{theorem}[Global existence, nonlinear stability, and decay]
\label{T:main_global}
For $K>1$, suppose that the initial data $(q^\pm_0,h_0)$ satisfies the conditions \eqref{compatibility1}--\eqref{E:COMP_GOMEGA}, 
\eqref{E:SIGNASSUMPTION},  and \eqref{E:globalTaylorsign}, 
as well as the condition
\[
\max\{K(q_0^+),\,K(q_0^-)\} \le K \,,
\]
where $K (q_0^\pm)$ is defined in (\ref{ssK}).
Then,
there exists an $\epsilon_0 >0$ and a monotonically increasing function $F:(1,\infty)\rightarrow\R_+$ which is independent of $\epsilon_0$ and $K$,
such that if
\begin{equation}\label{initialdatawithF}
S(0) < \frac{\epsilon_0^2}{F(K)} \,,
\end{equation} 
then there exist a unique global-in-time solution $(q^\pm,h)$ to problem \eqref{E:ALE} satisfying 
\begin{equation}
S(t) < C\epsilon_0^2,\ \ t\in [0,\infty),
\end{equation}
for some universal constant $C>0$. Moreover, the temperature $q^\pm(t)\rightarrow 0$ as $t\rightarrow \infty$ with 
a decay rate
\begin{equation}
\|q^\pm(t)\|^2_{H^4(\Omega^\pm)}\leq Ce^{-\beta^\pm t},
\end{equation}
where $\beta^\pm = 2\lambda^\pm_1-\eta$ is defined in~\eqref{E:BETAPMDEFINITION}.
The moving boundary $\Gamma(t)$ converges asymptotically to some nearby time-independent hypersurface $\bar{\Gamma}$,
and 
\begin{equation}
\sup\limits_{0\leq t<\infty}|h(t,\cdot)-h_0(\cdot)|_{4.5} \leq C\sqrt{\epsilon_0}.
\end{equation}
\end{theorem}

\begin{remark}
Other existing global stability results for the Stefan problem contain an effective heat source most commonly introduced through the presence of nontrivial Dirichlet boundary conditions~\cite{amM1992extra, fQjlV2000}. Such stability questions are simpler than our problem, as the presence of a heat source makes the family of possible steady states finite-dimensional. This allows to a priori guess a possible asymptotic attractor for the nonlinear dynamics. In our case, due to an abundance of possible steady states, small perturbations converge to some nearby element of the set of steady states. 
Characterization of such a nearby asymptotic state in terms of initial data is a difficult problem. 
\end{remark}

\subsection{A brief history of prior results}
The Stefan problem was introduced by Jo\v zef Stefan in 1889 as a model for the melting of  ice caps \cite{jS1889c, jS1891},
and is now considered  a prototype  free boundary problem in the area of nonlinear partial differential equations;
a historical account of the analysis of related free boundary is given  in \cite{aFfR1988, dKgS1980} for results prior to the 1980s. An account of
more recent results 
is provided in~\cite{amM1992extra,liR1971,aFfR1988,aV1996}; see also the introduction to~\cite{mHsS2013}.

 Weak solutions to the classical Stefan problem were shown to exist in~\cite{slK1961,aF1968, qaLvaSnnU1968}, for both the one-phase and the two-phase problem.
In the one-phase case, the problem lends itself to a variational approach that was successfully used in~\cite{aFdK1975} to study the existence and regularity of solutions.
Important regularity results were established in~\cite{lC1977,dKlN1977,dKlN1978,lC1978,lCaF1979}. The continuity of the the temperature function for the weak solutions of the two-phase classical Stefan problem
in any dimension was proved in~\cite{lCcE1983}. Another notion of a generalized solution for the classical Stefan problem, called the {\em viscosity} solution, was introduced and studied in the 
seminal works~\cite{iAlCsS1996.0, iAlCsS1996, iAlCsS1998, iAlCsS2003}, while the existence proof and further regularity results can be found in~\cite{iK2003,iKnP2011,sCiK2010, sCiK2012}. 
An overview of various regularity results for  viscosity solutions prior to 2005 can be found in the monograph~\cite{lCsS2005}.

Short-time existence of classical solutions of the one-phase problem was established in~\cite{eiH1981} under sufficient regularity assumptions and higher-order compatibility conditions. 
In~\cite{FrSo} the authors prove local existence for the one-phase classical Stefan problem in higher dimension.
In the two-phase case, local existence and uniqueness of classical solutions was proven in~\cite{amM1992extra}. Neither of these  papers, however, established the full well-posedness in the sense of Hadamard, as 
the constructed solutions experience a potential derivative loss.  
Under mild regularity assumptions on the initial data and a more general domain,
the local-in-time solutions were shown to exist in~\cite{jPjSgS2007}, proving additionally the space-time analyticity of the solutions. Smoothness of the free boundary and the temperature was also shown 
in~\cite{hK1998}.

 Using initial domains $\mathbb T^{d-1}\times (0,1)$ and $\mathbb T^{d-1 } \times (-1,0)$ and temperature profiles that allows for only one steady state solution to the two-phase Stefan problem, global-in-time stability  was established in \cite{amM1992extra}, by 
imposing Dirichlet  boundary conditions on the two fixed  boundaries: $\mathbb T^{d-1}\times\{x_d=1\}$ and $\mathbb T^{d-1}\times\{x_d=-1\}.$ 
In such a setting, the
solution to the nonlinear problem, can be treated as a small perturbation of the known linear solution by contrast to our problem.


A similar strategy is taken in~\cite{fQjlV2000}, where a global-in-time description of the dynamics for the one-phase classical Stefan problem is given. 
Therein, the authors study the {\em exterior} problem in the presence of a nontrivial heat source, modeled again through the imposition of an appropriate Dirichlet boundary condition.
The corresponding free boundary expands to infinity and the asymptotic rate is given. That method relies on the availability of a non-trivial background Hele-Shaw solution.

Global existence of  classical solutions of the one-phase problem was  proved in \cite{pDkL2004}
for log-concave initial temperatures, and hence for convex initial domains.

In the presence of surface tension, the families of steady states to the Stefan problem are parametrized by finitely many parameters and therefore the problem does not exhibit the same type 
of difficulty as the Stefan problem in the absence of surface tension. 
The global-in-time nonlinear stability of flat steady states was established in~\cite{mHyG2010}. In the more complicated case of steady spheres, the nonlinear stability was first proved in~\cite{mH2012}, and by a different method in
\cite{jPgSrZ2013}.


In the absence of surface tension, 
the nonlinear stability of nearly spherical steady state solutions to the one-phase Stefan problem was proved in~\cite{mHsS2013} and the authors generalized that result to allow for  arbitrary (bounded)  initial domains in~\cite{mHsS2015}.
Due to the  infinite-dimensional space of steady states,  the nonlinear stability theory cannot be viewed as a perturbation of a given linear profile;
thus, a  novel hybrid methodology was developed in \cite{mHsS2013,mHsS2015},  which combined  energy methods with pointwise maximum principle techniques to establish exponential-in-time lower-bounds on the Rayleigh-Taylor stability condition.   
Maximum principles together with energy estimates were also used in~\cite{aCdCfG2011} for the analysis of the related Muskat problem.


\subsection{Methodology and outline of the paper}\label{S:METHODOLOGY}

Our first main result is Theorem~\ref{T:main} proving the local well-posedness for the two-phase classical Stefan problem. Our methodology extends the hybrid method developed in~\cite{mHsS2011,mHsS2013} for the one-phase problem in a fundamental way.

Following the energy method of \cite{CoShfreesurfEuler} for the incompressible Euler equations,
tangential and temporal energy estimates on the problem
~\eqref{E:ALE} 
lead to control of the interface regularity via an  integral of the type
$
\int_\Gamma   \mathcal{F} (x,t) |\t^k h|^2\,dx$ for some function $ \mathcal{F} (x,t)>0$.   
Unlike the one-phase problem wherein $\Gamma(t)$ moves with speed $v \cdot n$ and so $\mathcal{F} (x,t) = -\g_Nq$ , in the 
 two-phase setting,  $\Gamma(t)$ moves with the jump  of $v \cdot n$, and hence weight functions must be introduced into the energy method
 to obtain the function $ \mathcal{F} (x,t)$.  Specifically, since on $\Gamma$, 
\begin{equation}\label{ssW}
\g_Nq^+ \neq \g_Nq^- \,,
\end{equation} 
we introduce the weight functions $W^\pm$~\eqref{E:W} in the interior of the two phases $\Omega^\pm$ designed to resolve the mismatch in
(\ref{ssW}), and to allow us to form a common factor in the difference of the two boundary integrals arising from integration-by-parts in both phases
$\Omega^+$ and $\Omega^-$.
With our weighted energy method, we obtain control on the boundary integral
\be\label{E:BOUNDARYENERGY}
\newe^\ppp = e^{-(\l_1+\eta)t}\int_\Gamma  |\t^6 h|^2\,dx,
\ee
where $\l_1$ is the smaller of the two first eigenvalues of the Dirichlet-Laplacian on the domains $\Omega^\pm.$
In order to use this energy control to prove the existence of solutions, we will regularize in the tangential directions with the convolution-by-layers smoothing operator introduced in~\cite{CoShfreesurfEuler}, 
and study the regularized problem. We obtain that for a short time $T$, depending on the smoothing parameter $\k$, there exists a solution. The aforementioned energy control will give us a uniform bound, which will guarantee that the time of existence does not vanish as the smoothing parameter goes to zero. Taking a limit as $\k\to 0$ leads to a local-in-time solution to~\eqref{E:ALE}.

In order to prove global-in-time stability of a given steady state, we need to contend with the exponentially decaying weight present in~\eqref{E:BOUNDARYENERGY}.
Its presence suggests that a bound on \eqref{E:BOUNDARYENERGY} implies that $|\t^6 h|_0$ can grow exponentially fast.
A related  issue is also present in the one-phase case~\cite{mHsS2013} and our general strategy is similar;
whenever we have to bound the top order norms of $h$ we do that at the cost of an exponentially growing factor, since 
\be\label{E:BADESTIMATE}
|\t^6h|_0 \le e^{(\l_1+\eta)t/2} (\newe^\ppp)^{1/2}.
\ee
On the other hand, we {\em do expect} that the temperature $q$ will decay exponentially fast to the equilibrium, as it solves a heat equation. Therefore,
each time we use~\eqref{E:BADESTIMATE} we have to make sure that it comes coupled with a lower order derivative of $q$ which decays sufficiently fast to counter balance
a possible growth coming from~\eqref{E:BADESTIMATE}.
While this strategy works for most of the error terms, there are certain {\em energy critical} error terms with no room left to obtain 
the desired exponentially decaying factor in the error terms.

In~\cite{mHsS2013} this critical term took the form
\begin{equation}\label{E:difficult_term}
\int_\Gamma \partial_N q_t \, |\t^6 h|^2,
\end{equation}
which could not be treated as an error term since the expected decay rate of $\g_Nq_t$ is exactly the same as the decay rate of $\g_Nq.$ 
To resolve this issue, the authors proved that after a sufficiently long time interval,  the term~\eqref{E:difficult_term} is sign-definite, with a favorable sign.
This required a complicated usage of comparison principles and a decomposition of the temperature into the eigenfunctions of the {\em Dirichlet-Laplacian}. 

In our current treatment,  we  circumvent this difficulty 
through the introduction of the weights $W^\pm$ in the definition of the natural energy $\mathcal E.$ As a consequence, the corresponding ``critical" term takes the form
\begin{equation}\label{E:not_so_difficult_term}
\int_\Gamma \partial_t \left(e^{-(\l_1+\eta)t}\right)|\t^6 h|^2 = -(\l_1+\eta)\int_\Gamma  e^{-(\l_1+\eta)t}|\t^6 h|^2 < 0.
\end{equation}
The simplification in our analysis caused by the estimate~\eqref{E:not_so_difficult_term} is very substantial, but it does come with a small price. 
Terms $\partial_Nq^\pm_t$ are implicitly hidden in the terms $\partial_tW^\pm$ 
which appear inside some of the interior error  terms involving integration over $\Omega^\pm.$ 
However, the dissipative effects are stronger inside $\Omega^\pm$ and we combine 
norm interpolation and energy estimates to overcome a potentially exponential growth in our estimates caused by $\g_tW^\pm.$ 
This simplifies the proof significantly, as we no longer need to \emph{wait} until the dynamics settle into a regime dominated by the first eigenfunction of the \emph{Dirichlet-Laplacian} as in Section 4.3 of \cite{mHsS2013}.

To get quantitative lower bounds on the weights $W^\pm,$ we must
obtain sharp quantitative lower bounds for the quantities $\g_Nq^\pm.$
We implement a bootstrap scheme, where we first assume such bounds and use them to prove important energy-norm equivalence Lemmas in Section~\ref{S:bootstrap}.
Just like in~\cite{mHsS2015}, to show that the lower bound is dynamically preserved, we make a very sharp use of the Pucci operators and comparison principles as explained in Section~\ref{S:lowerbound}.
Finally, using  standard continuity arguments and the improvement of the bootstrap bounds, we present the proof of Theorem~\ref{T:main_global} in Section~\ref{S:finalproof}.

To summarize, a novel aspect 
of our methodology is the  introduction  of the weight functions  $W^\pm(t,x)$ with very specific decay properties. One of its key features is that it measures the boundary energy contribution 
in terms of a higher-order Sobolev norm weighted by an explicit exponential
$e^{(-\lambda_1 + \eta)t}$. 
This  simplifies the global-in-time analysis with respect to~\cite{mHsS2013}, and provides a tool for studying similar multi-phase problems in the absence of surface tension.
Equally importantly, using the weighted higher-order energy, we are able to show that  the top-order norms $\newe^\pm$ also decay-in-time. 
The top-order terms decay at a slower rate than predicted by the 
 linear theory, a consequence of the degeneracy caused by the nonlinear and mixed parabolic-hyperbolic character of the equations.

Finally, the perturbation, given by  $h$,  from our initial geometry does not decay, but rather it converges, as $t \to \infty $,  to some
 nontrivial $h_\infty$ which is very small in a suitable Sobolev norm.

\subsection{Future work}

The well-posedness framework introduced in this work is well-suited for the investigation of various singular limits that commonly arise in the study of free boundary problems.  We intend to establish that  solutions to the
 one-phase Stefan problem are, in fact,  limits of solutions of the  two-phase Stefan problem in the limit as the ratio of the diffusion coefficients
converges to zero.   A second important singular limit amenable to our approach is the problem of the vanishing surface tension limit.  Our
energy method naturally extends to the  surface tension problem, by simply adding new top-order energy terms, weighted by the surface tension coefficient.
 We intend to examine the possibility of a splash singularity for the one-phase Stefan problem as in \cite{CoSh2014} and to 
investigate if a splash singularity can occur  for the two-phase Stefan problem following the methodology of \cite{CoSh2016}.


\subsection{Outline of the paper}

Section \ref{C:local} is devoted to the proof of the local well-posedness Theorem~\ref{T:main}. In Section~\ref{S:regularization} we regularize the Stefan problem. 
In Section \ref{s:energydef} we define the energy functionals with the new weights $W^\pm$. In Sections~\ref{S:ENERGYEQUIV} and~\ref{S:derivation} we establish the short-time relationship between the natural energy and the norms
and derive the energy identities respectively. In Section \ref{S:energyestimates} we prove the energy estimates and in Section~\ref{S:PROOF1} we finally finish the proof of the local existence Theorem. 
In Section~\ref{S:bootstrap} we re-introduce the hypotheses for the global stability Theorem, and the bootstrap assumptions. 
 In Sections~\ref{S:WGLOBAL} and~\ref{S:APRIORIH} we obtain global estimates for the weights $W^\pm,$ energy-norm equivalence, and some a priori estimates for the height function $h$. Sections \ref{S:lowerbound}, \ref{S:lower_order} and \ref{S:improvedenergy} are dedicated to the proof of the dynamic improvement of our the bootstrap assumptions, and in Section \ref{S:finalproof} we present the proof of the global stability Theorem. Appendix \ref{S:apendix_local} briefly presents some useful bounds for the change of variables $\Psi^\pm$.     

\section{Local well-posedness: proof of Theorem \ref{T:main}}\label{C:local}

We begin by constructing a sequence of approximate, so called $ \kappa $-problems which
retain the  nonlinear structure of the original two-phase Stefan problem. The small number $ \kappa >0$ is the radius of
convolution, and our $ \kappa$-problems
 \eqref{E:ALEregularized} are founded
upon the smoothing of the evolving interface $\Gamma(t)$; in particular, in
Section \ref{S:regularization}, we regularize the height function $h$ using a symmetric horizontal convolution operator $ \Lambda _\kappa$,  and otherwise keep the structure of the equations
the same.
In Section \ref{S:linearproblem}
we establish an existence theorem for our sequence of  $ \kappa $-approximations  \eqref{E:ALEregularized} by the contraction mapping principle.
The  time of existence $T_\kappa $, a priori, may shrink to zero as $\kappa \to 0$, but in Section \ref{S:energyestimates}, we establish
$ \kappa$-independent estimates, which allow us to prove that $T_ \kappa$ is, in fact,  independent of $\kappa$. 
Passing to the limit as $\k\rightarrow 0$, we shall obtain solutions to the Stefan problem \eqref{E:ALE}.

\subsection{Sequence of approximate $ \kappa$-problems}\label{S:regularization}

For a given parameter $\k>0$ and a height function $h$,  we define its regularization by

\begin{equation}
h^\k := \Lambda_\k\Lambda_\k h, \label{E:h_kappa}
\end{equation}
where $\Lambda_\k$ is the smoothing operator defined in \eqref{E:Lambdadefinition}.
We introduce the \emph{regularized coordinate transformations} $\psik^\pm$ as the solutions to
\begin{subequations}\label{E:Psi_regularized}
\begin{alignat}{2}
\Delta\psik^{\pm} &= 0\ &&\text{ in } \Omega^{\pm},\\
\psik^{\pm}(t,x) &= x + h^{\k}(t,x) N\  &&\text{ on } \Gamma,\\
\psik^+ &= e\  &&\text{ on } \g\Omega,
\end{alignat} 
\end{subequations}
where we recall that $\text{e}$ is the identity map on $\partial\Omega$.
Similarly as for~\eqref{E:definition_Psi}, notice that equation \eqref{E:Psiboundarycontrol} for this regularized problem is,
\begin{equation}\label{E:Psiboundarycontrolkappa}
\|\psik^\pm - e\|_{6.5} \leq C|h^\k|_{6}.
\end{equation}
Therefore, the smallness of $|h^\k-h_0|_6$ for short time together with the choice of $|h_0|_6\leq C\sigma\ll 1$ gives us that $\psik^\pm$ are in fact $H^{6.5}$-class diffeomorphisms.   
As in \eqref{E:A_definition}, we define now, ${}^\k \hspace{-.06in} A^\pm :=  ( 
\nabla \psik^\pm)^{-1}$, and let $w^\pm_\k := \g_t \psik^\pm$. We introduce our sequence of approximations to the Stefan problem as
the following {\bf $\k$-problem}:
\begin{subequations}\label{E:ALEregularized}
\begin{align}
q_t^\pm - \Delta_{\psik^\pm} q^\pm &= -v^\pm\cdot w^\pm_\k + \alpha^\pm \text{ in } \Omega^\pm,\label{E:ALEregularized_heat}\\
v^\pm + {}^\k \hspace{-.06in} A^\pm{}^\top  \nabla q^\pm &= 0 \text{ in } \Omega^\pm,\label{E:ALEregularized_v}\\
q^\pm &=  \mp\k^2 ((v^\pm)^i\ {}^\k \hspace{-.06in} A^j_i N^j) \pm \k^2\beta^\pm(t,x) \text{ on } \Gamma,\label{E:ALEregularized_dirichlet}\\
 h_t &= [v\cdot\tilde{n}^\k]^+_-\ \text{ on } \Gamma, \label{E:HKAPPAEVOLUTION}\\
v^+\cdot {\bf N}^+ &= \gamma \text{ on } \g\Omega,\label{E:ALEregularized_neumann_top}\\
q^\pm\vert_{t=0} &= \Qko^\pm \text{ on } \{t=0\}\times\Omega^\pm,\label{E:ALEregularized_initial} \\
h\vert_{t=0} &=h_0^\k \ \text{ on } \{t=0\}\times \Gamma.
\end{align}
\end{subequations}
where $\bf N^+$ is the exterior normal vector to the fixed boundary $\g\Omega$, $\tilde{n}^\k:= \frac{n^\k}{n^\k\cdot N}$, $n^\k$ is the normal vector to the regularized interface $\Gamma_\k(t)$, given by
$$
n^\k(t,y(\theta)) = \frac{-\t h^\k \tau + (1+H(\theta)h^\k(t,z(\theta)))N}{\sqrt{(\t h^\k)^2 + (1+H(\theta)h^\k)^2}},
$$
and $\Qko^\pm$ is the initial data defined carefully in Section \ref{qzerokappa}.
The introduction of this special initial data, and the functions $\alpha^\pm,\ \beta^\pm$ and $\gamma$ into equations \eqref{E:ALEregularized_heat},\eqref{E:ALEregularized_dirichlet}, \eqref{E:ALEregularized_neumann_top} respectively, has the purpose of \emph{cancelling} the new compatibility conditions that arise in the $\k$-problem due to the smoothing. The central idea is that equations \eqref{E:ALEregularized} and their time derivatives, when restricted to time $t=0$, will produce new terms from the $\k$-dependent coefficients $\ako^\pm$, that the functions $\alpha^\pm(0),\ \beta^\pm(0)$, and $\gamma(0)$ will cancel and replace with the analogous terms of the non-regularized problem \eqref{E:ALE}, which in turn corresponds to the original compatibility conditions satisfied by $q_0^\pm$ \eqref{E:compatibility}.\\ 

For $s\geq 0$, let $E^\pm:H^s(\Omega^\pm)\rightarrow H^s(\mathbb{R}^2)$ be the Sobolev extension operator of $\Omega^\pm$. Then we define the function $ \alpha^\pm (t,x)$ over $\Omega^\pm$ as,
\begin{equation}\label{E:alpha}
\alpha^\pm(t,x) = \alpha_0^\pm(x) + \int_0^t r^\pm(s,x) ds,
\end{equation}
where $\alpha_0^\pm$ is given by
\begin{equation*}
\alpha_0^\pm := - \nabla_{\psik_0}\Qko^\pm\cdot \psibark_t^\pm(0) + \nabla_{\Psi_0}q_0^\pm\cdot \Psi_t^\pm(0),
\end{equation*}
and $r^\pm := \bar{r}^\pm\vert_{\Omega^\pm}$ is the restriction to $\Omega^\pm$ of the solution to the parabolic problem, 
\begin{subequations}
\begin{align*}
\bar{r}^\pm_t + \Delta^2 \bar{r}^\pm &=0 \text{ in } \mathbb{R}^2,\\
\bar{r}^\pm(t=0) &= E^\pm(\alpha_1^\pm) \text{ on } \mathbb{R}^2\times\{ t=0\},\label{E:r_0}
\end{align*}
where $\alpha_1^\pm$ is defined in $\Omega^\pm$ as, 
\begin{align*}
\alpha_1^\pm &:= -B_1^\pm(\Qko,h^\k_0) + B_2^\pm(q_0, h_0),
\end{align*}
with,
\begin{align}
B_1(\Qko,h^\k_0) &:= \Delta_{\psik_0}\left(\nabla_{\psik_0}\Qko\cdot\psik_t(0) +\alpha_0\right) - \ak_0{}^i_l\psik_t{}^l,_s(0) \ak_0{}^s_j(\ak_0{}^k_j\Qko,k),_i\nonumber\\
&\quad - \ak_0{}^i_j(\ak_0{}^k_l\psik_t{}^l,_s\ak_0{}^s_j \Qko,_k),_i - \ak_0{}^i_l\psik_t^l,_s\ak_0{}^s_j\Qko,_i\psik_t(0)^j\nonumber\\
&\quad + \nabla_{\psik_0}(\Delta_{\psik_0}\Qko + \nabla_{\psik_0}\Qko\cdot\psik_t(0) + \alpha_0)\cdot\psik_t(0) + \nabla_{\psik_0}\Qko\cdot\psik_{tt}(0),\\
B_2(q_0,h_0) &:= -\ao^i_l\Psi_t(0)^l,_s\ao^s_j(\ao^k_j q_0,_k),_i - \ao^i_j(\ao^k_l\Psi_t(0)^l,_s\ao^s_j q_0,_k),_i\nonumber \\
&\qquad -\ao^i_k\Psi_t(0)^k,_l\ao^l_j q_0,_i \Psi_t(0)^j + \nabla_{\Psi_0}(\Delta_{\Psi_0}q_0+\nabla_{\Psi_0}q_0\cdot\Psi_t(0))\cdot \Psi_t(0)\nonumber\\
&\qquad + \nabla_{\Psi_0}q_0\cdot \Psi_{tt}(0) + \Delta_{\Psi_0}(\nabla_{\Psi_0}q_0\cdot\Psi_t(0)).
\end{align}
Notice that, since $q_0$ satisfies the compatibility conditions \eqref{compatibility2}, when restricted to $\Gamma$, $B_1$ and $B_2$ can be writen simply as, 
\begin{align*}
B_1(\Qko,h_0^\k) &= \g_t(\Delta_{\psik}q + \nabla_{\psik}q\cdot\psik_t)\vert_{t=0} - \Delta_{\psik_0}^2\Qko,\\
B_2(q_0,h_0) &= \Delta_{\Psi_0}^2 q_0,
\end{align*}
where the value for $q_t(0)$ can be obtained from restricting equation \eqref{E:ALEregularized_heat} to $t=0$ and using that $q(0)=\Qko$. 
\end{subequations}

\begin{remark}
Since, in the right-hand side of equation \eqref{E:ALEregularized_heat} we have the term $\alpha^\pm(t,x)$, the regularity of $\g_t^l q^\pm$ for $l=0,..,3$, depends, among other things, on the regularity of this term and its time derivative. Specifically, we will need to bound $\alpha^\pm$ in $H^5(\Omega^\pm)$ (see Section \ref{SS:alpha_gamma_errors}), in order to obtain the desired regularity of $q^\pm$, but at the same time we require that $\g_t\alpha^\pm (0) = \alpha_1^\pm$ on $\Gamma$, which does not have enough derivatives. The introduction of $r^\pm$ solves this problem since, from the standard parabolic regularity theory, 
the solution $r^\pm(t,x)$ is in $L^2((0,T);H^5(\Omega^\pm))$ since the initial datum $\bar{r}^\pm(t=0)$ belongs to $H^3(\Omega)$ by our regularity assumptions on $(q_0,h_0)$.
\end{remark}

Now let us define the function $\gamma$ on $\g\Omega$ as: 
\begin{equation}\label{E:gamma}
\gamma := \mathcal{G}^0\cdot{\bf N}^+ + \int_0^t \mathcal{G}^1(s)\cdot{\bf N}^+ds + \int_0^t\int_0^s \mathcal{G}^2(\tau)\cdot{\bf N}^+d\tau ds \text{ in } \Omega^+,
\end{equation}
where $\mathcal{G}^i := \tilde{\mathcal{G}}^i\vert_{\g\Omega}$ for $i=0,1,2$, is the restriction to $\g\Omega$ of the solution to the following parabolic problem:
\begin{align}
\tilde{\mathcal{G}}^i_t - \Delta^{2i+1} \tilde{\mathcal{G}}^i &= 0\ \text{ in } \mathbb{R}^2,\label{E:GEQUATION1}\\
\tilde{\mathcal{G}}^i(t=0) &= E^+(\gamma_i)\ \text{ on } \mathbb{R}^2\times\{t=0\}, \label{E:GEQUATION2}
\end{align} 
with $\gamma_i$ defined in $\Omega^+$ as,
\begin{align*}
\gamma_0 &:= -\nabla_{\psik_0^+}\Qko^+ + \nabla_{\Psi_0^+}q_0^+, \\
\gamma_1 &= -[\ako^+\nabla\psik^+_t(0)\ako^+]^\top\nabla\Qko^+ + \nabla_{\psik_0^+}(\nabla_{\Psi_0^+}q^+_0\cdot \Psi_t^+(0)) + \nabla_{\Psi_0^+}(\Delta_{\Psi_0^+}q_0^+),\\
\gamma_2 &= \g_t^2(\nabla_{\psik^+}q^+)\vert_{t=0} - \nabla_{\psik_0^+}(\Delta_{\psik_0^+}^2\Qko^+) + \nabla_{\Psi_0^+}(\Delta^2_{\Psi_0^+}q_0^+).
\end{align*} 
As a consequence of~\eqref{E:GEQUATION1}--\eqref{E:GEQUATION2} and the Sobolev regularity of $\tilde{\mathcal{G}}^i(t=0)$, $i=0,1,2$, standard parabolic regularity theory gives the bound
\begin{align}\label{E:GBOUND}
\|\tilde{\mathcal G}^0\|_{L^\infty_tH^{5.5}}+\sum_{i=0}^2 \|\tilde{\mathcal G}^i\|_{L^2_tH^6} \lesssim 1.
\end{align}
Estimate~\eqref{E:GBOUND} plays a crucial role in the nonlinear estimates in Section~\ref{SS:alpha_gamma_errors}.

The function $\beta^\pm(t,x)$ is defined on $\Gamma$ as
\begin{equation}\label{E:beta}
\beta^\pm(t,x) := \sum\limits_{k=0}^3 \frac{t^k}{k!}\g_t^k((v^\pm)^i\ {}^\k \hspace{-.06in} A^j_i N^j)\vert_{t=0}. 
\end{equation} 
Note that $\beta$ is a cubic polynomial in $t$ with space dependent coefficients.

\begin{remark}
The functions $\beta^\pm$ serves a similar purpose on the boundary $\Gamma$ as $\alpha$ does in the interior, and they are used to avoid new compatibility conditions that may appear from the boundary regularization \eqref{E:ALEregularized_dirichlet}. This regularization is needed to overcome a technical difficulty in the higest-in-time energy estimates when we have a term of the form $$\int_\Gamma (\psik^\pm{}_{ttt}\cdot n)( v_{ttt}^\pm\cdot n) d\sigma,$$ because the trace of $v^\pm_{ttt}\cdot n$ is not necessarily well defined. 
\end{remark}

\subsubsection{Definition of the smooth initial data $\Qko^\pm$}\label{qzerokappa}
We now construct a smooth version of the initial data $q_0^\pm$ that will satisfy the compatibility conditions for the $\k$-problem \eqref{E:ALEregularized}. We solve the tri-Laplacian, 
$$\Delta^3_{\psik_0^\pm}\Qko^\pm = \Delta_{\Psi_0^\pm}^3 (\eta_\k\ast E(q_0^\pm)) \text{ in } \Omega^\pm,$$
with specific boundary data designed to satisfy the compatibility conditions. We proceed by solving the equivalent system of elliptic equations: 
\begin{subequations}\label{E:kappa_initialdata_system}
\begin{align}
\Delta_{\psik_0^\pm}\Qko^\pm &= {}^\k R_0^\pm \text{ in } \Omega^\pm,\\
\Qko^\pm &= 0 \text{ on } \Gamma,\\
\nabla_{\psik_0^+}\Qko^+\cdot {\bf N}^+ &= 0 \text{ on } \g\Omega,
\end{align}

\begin{align}
\Delta_{\psik_0^\pm}{}^\k R_0^\pm &= {}^\k U_0^\pm \text{ in } \Omega^\pm,\\
{}^\k R_0^\pm &= -(\nabla_{\Psi_0^\pm}q_0^\pm\cdot N) g_1 \text{ on } \Gamma,\label{E:nokappaboundary1}\\
\nabla_{\psik_0^+}{}^\k R_0^+\cdot {\bf N}^+ &= \nabla_{\Psi_0^+}(\Delta_{\Psi_0^+}q_0^+)\cdot {\bf N}^+ \text{ on } \g\Omega,\label{E:nonhighorder1}
\end{align}
\begin{align}
\Delta_{\psik_0^\pm} {}^\k U_0^\pm &= \Delta_{\psik_0^\pm}^3(\eta_\k\ast E(q_0^\pm)) \text{ in } \Omega^\pm,\label{E:uokLaplace}\\
{}^\k U_0^\pm &= B_2(q_0^\pm,h_0) \text{ on } \Gamma,\label{E:nokappaboundary2}\\
\nabla_{\psik_0^+}{}^\k U_0^+\cdot {\bf N}^+ &= \nabla_{\Psi_0^+}(\Delta_{\Psi_0^+}^2 q_0^+)\cdot {\bf N}^+ \text{ on } \g\Omega.\label{E:nonhighorder2}
\end{align}
\end{subequations}
Recall that $g_1$ is defined in~\eqref{E:definition_h_t0}.
Notice that the system is decoupled, and therefore existence of solutions follows directly. 
This choice of initial data, and the fact that $q_0^\pm$ satisfy \eqref{E:compatibility}, shows that the compatibility conditions for equation \eqref{E:ALEregularized} are automatically satisfied. Moreover, as $\k\rightarrow 0$, 
\begin{equation*}
\Qko^\pm \rightharpoonup q_0^\pm\text{ in } H^6(\Omega^\pm).
\end{equation*}  

\begin{remark}
We actually have strong convergence of $\Qko^\pm$ to $q_0^\pm$ in $H^6(\Omega^\pm)$. The argument is simple, but cumbersome, as it involves elliptic estimates from all three equations \eqref{E:kappa_initialdata_system}. Consider for example equations \eqref{E:uokLaplace} with boundary condition \eqref{E:nokappaboundary2}, on the region $\Omega^-$ (we will omit the index ``$-$''). In order to estimate the difference between $\Uko - \Delta^2_{\psik_0}q_0$ we analyze the elliptic problem: 
\begin{align*}
\Delta_{\psik_0}(\Uko - \Delta_{\Psi_0}^2 q_0) &= \Delta_{\psik_0}\left(\Delta_{\psik_0}^2(\eta_\k\ast E(q_0)) - \Delta_{\Psi_0}^2 q_0\right)\ \text{ in } \Omega,\\
\Uko - \Delta_{\Psi_0}^2 q_0 &= 0\ \text{ on } \Gamma.
\end{align*}
Let us define $G := \Uko - \Delta_{\Psi_0}^2 q_0$; Then, the system can be rewritten as 
\begin{align*}
\Delta_{\Psi_0}G &=  \Delta_{\psik_0}\left(\Delta_{\psik_0}^2(\eta_\k\ast E(q_0)) - \Delta_{\Psi_0}^2 q_0\right) + (\Delta_{\Psi_0} - \Delta_{\psik_0})G \text{ in } \Omega,\\
G &= 0 \text{ on } \Gamma,
\end{align*} 
and so, by elliptic estimates we have the bound: 
\begin{align}
\|G\|_2 &\leq \|\Delta_{\psik_0}(\Delta_{\psik_0}^2(\eta_\k\ast E(q_0))- \Delta_{\Psi_0}^2 q_0)\|_0 + \|A_0A_0 - \ako\ako\|_{L^\infty}\|G\|_2 + \mathcal{O}_\k,\label{E:G_bound_kappa}
\end{align}
where we have gathered all the lower-order terms coming from the product rule in $\mathcal{O}_\k$. Since we have strong convergence of $h^\k_0\rightarrow h_0$ in $H^6(\Gamma)$, we conclude that $\psik_0\rightarrow \Psi_0$ in $H^{6.5}(\Omega)$, and therefore $\ako\rightarrow A_0$ in $H^{5.5}(\Omega)$. Combining this fact together with the strong convergence of $\eta_\k\ast E(q_0)$ to $q_0$ in $H^6(\Omega)$, the right-hand side of equation \eqref{E:G_bound_kappa} goes to zero as $\k\rightarrow 0$, and therefore, $G\rightarrow 0$ in $H^2(\Omega)$, or equivalently: $\Uko \rightarrow \Delta_{\Psi_0}^2 q_0$ in $H^2(\Omega)$. This same procedure applied to the other equations of the system \eqref{E:kappa_initialdata_system} gives us the necessary estimates to prove the strong convergence $\Qko^\pm \rightarrow q_0^\pm$ in $H^6(\Omega^\pm)$.

\end{remark}		

\subsection{Existence theorem for the $\k$-problem}\label{S:linearproblem}
In this section, we use the contraction mapping theorem to find solutions to the $\k$-problem \eqref{E:ALEregularized}.
We introduce the following normed space of functions:

\begin{align}\label{E:contraction_space}
X_M^\k &= \left\{ \g_t^l h^\k\in C([0,T_\k];H^{6-2l}(\Gamma)), \g_t^{s+1} h^\k \in L^2([0,T_\kappa];H^{5-2s}(\Gamma)):\ 0\leq l\leq 3, 0\leq s\leq 2,\right.\nonumber \\
&\qquad \sum\limits_{l=0}^3 |\g_t ^ l h^\k(s)|^2_{L^\infty_t H^{6-2l}(\Gamma)} + \sum\limits_{l=0}^2|\g_t^{l+1} h^\k|_{L^2_t H^{5-2l}(\Gamma)}^2 \leq M,\nonumber\\
&\qquad  \left.h^\k(0,x) = h_0^\k(x),\ h_t^\k(0,x) = g_1^\k(x),\ h_{tt}^\k(0,x) = g_2^\k(x),\ h_{ttt}^\k(0) = g_3^\k(x)\right\},
\end{align}
with $M>0$ a function of the initial data to be determined, and $g_i^\k := \Lambda_\k\Lambda_\k g_i,\ i=1,2,3,$ the smooth versions of \eqref{E:h_and_derivativesatzero}.

\begin{theorem}[Solutions to the $\kappa$-problem]\label{T:smooth_solutions}
For any fixed $\kappa > 0$ there exist a time $T_\kappa>0$, such that there exists a unique olution $(q, h^\k)$ to the {\bf non-linear} $\kappa$-problem~\eqref{E:ALEregularized} on the time-interval $[0,T_\k]$, and $\mathcal{S}(t) = \mathcal{S}(q^\pm, h^\k) \leq M$.
\end{theorem}

\begin{proof}

We will separate the proof into three steps. 

\subsubsection{Step 1: The linear problem.}
Assuming that a function $\bar{h} \in X_M$ is given, consider the regularized version of $\bar{h}$:
\begin{align}
\bar{h}^\k := \Lambda_\k\Lambda_\k \bar{h}.\label{E:bar_h}
\end{align}
Again, we define the \emph{regularized coordinate transformations} $\psibark^\pm$, as the solutions to

\begin{subequations}\label{E:Psi_regularized_linear}
\begin{alignat}{2}
\Delta \psibark^{\pm} &= 0\ &&\text{ in } \Omega^{\pm},\\
\psibark^{\pm}(t,\cdot) &= x + \bar{h}^{\k}(t,\cdot) N(\cdot)\  &&\text{ on } \Gamma,\\
\psibark^+ &= e\  &&\text{ on } \g\Omega.
\end{alignat} 
\end{subequations}

We define  $\abark^\pm :=  (\nabla \psibark^\pm)^{-1}$, $\Jbark := \text{det}\nabla\psibark$, the cofactor matrix $\labark:= \Jbark\ \abark$, and ${}^\k\hspace{-0.02in}\bar{w}^\pm := \psibark_t^\pm$. Recall that equation \eqref{E:Psiboundarycontrol} holds for $\psibark^\pm$ with $\bar{h}^\k$ in the right-hand side. This gives us that the transformations $\psibark^\pm$ are in $C^\infty(\Omega^\pm)$. 
We then define the following {\it linearization} of the $\k$-problem:

\begin{subequations}\label{E:ALEregularized_linear}
\begin{align}
q_t^\pm - \Delta_{\psibark^\pm} q^\pm &= -v^\pm\cdot {}^\k\hspace{-0.02in}\bar{w}^\pm + \alpha^\pm \text{ in } \Omega^\pm,\label{E:ALEregularized_linear_heat}\\
v^\pm + \abark^\pm{}^\top \nabla q^\pm &= 0 \text{ in } \Omega^\pm,\label{E:ALEregularized_linear_v}\\
q^\pm &= \mp\k^2((v^\pm)^i\ \abark^j_iN^j) \pm \k^2\beta^\pm(t,x) \text{ on } \Gamma,\label{E:ALEregularized_linear_dirichlet}\\
v^+\cdot {\bf N}^+ &= \gamma \text{ on } \g\Omega,\label{E:ALEregularized_neumann_top}\\
q^\pm\vert_{t=0} &= \Qko^\pm \text{ on } \Omega^\pm\times\{t=0\},
\end{align}
\end{subequations}
where $\alpha^\pm,\ \beta^\pm$ and $\gamma$ are the functions of the initial data defined in \eqref{E:alpha},\eqref{E:beta} and \eqref{E:gamma} respectively.
Since $\bar{h}^\k(t)$ is prescribed, the linear system of equations \eqref{E:ALEregularized_linear} decouples into two linear heat equations on $\Omega^\pm$ respectively, with $C^\infty(\Omega^\pm)$ coefficients. The initial data for the linear smooth problem is $\Qko^\pm\in H^6(\Omega^\pm)$, which was designed in \eqref{E:kappa_initialdata_system}, along with the terms $\alpha^\pm,\ \beta^\pm,\ \gamma$, to recover the original two-phase compatibility conditions from the decoupled two phases as $\k\rightarrow 0$.

\subsubsection{Step 2: Higher regularity for the linear problem.} \label{S:higherregularitylinear}
We want to prove that there exists solutions to \eqref{E:ALEregularized_linear} such that $\g_t^l q^\pm\ \in C([0,T_\k]; H^{6-2l}(\Omega^\pm))\cap L^2([0,T_\k];H^{7-2l}(\Omega^\pm))$, for all $l=0,\dots,3$.
We proceed as in \cite{MR2208319}. Since the two phases are decoupled we have the weak formulation of the two different problems on the regions $\Omega^\pm$ separately. 
For a given function $f\in L^2\left([0,T_\kappa],\mathbb R\right)$ consider the discrete time difference $\delta_t^s f(t) = (\mathfrak Ef(t+s)-\mathfrak Ef(t))s^{-1}$, where $\mathfrak E$ is a Sobolev extension operator to the positive real line $[0,\infty)$.  We define then $q^\pm$ to be a weak solution to equation \eqref{E:ALEregularized_linear}, if for all $\phi^\pm\in H^1(\Omega^\pm)$, the following equations hold pointwise in time for all $t\in [0,T_\k]$: 
\begin{subequations}\label{E:weak_time_l}
\begin{align}
&\langle\g_t^l (\Jbark\, q_t),\phi^-\rangle +  \left(\g_t^l(\labark^i_j (\abark^k_j q,_k), \phi^-,_i\right)_{L^2(\Omega^-)} + \int_\Gamma \g_t^l(\Jbark(\k^{-2}q^- + \beta^-))\phi^- d\sigma  \nonumber\\
& = \left(\g_t^l(\labark^i_j q,_i \delta_t^s(\psibark^j)), \phi^-\right)_{L^2(\Omega^-)} + \left(\g_t^l(\Jbark\,\alpha^-),\phi^-\right)_{L^2(\Omega^-)}, \label{E:weak_time_l_minus}\\
&\langle\g_t^l (\Jbark\, q_t),\phi^+\rangle +  \left(\g_t^l(\labark^i_j (\abark^k_j q,_k), \phi^+,_i\right)_{L^2(\Omega^+)} + \int_\Gamma \g_t^l(\Jbark(\k^{-2}q^+ + \beta^+))\phi^+ d\sigma  \nonumber\\
& =  \left(\g_t^l(\labark^i_j q,_i \delta_t^s(\psibark^j)), \phi^+\right)_{L^2(\Omega^+)} + \int_{\g\Omega}(\g_t^l(\Jbark\,\gamma)) \phi^+ d\sigma\nonumber\\
&\qquad + \left(\g_t^l(\Jbark\,\alpha^+),\phi^+\right)_{L^2(\Omega^+)},\ \text{ for } l=0,\dots,3, \label{E:weak_time_l_plus}
\end{align}
\end{subequations}
with initial conditions given by $q^\pm(0) = \Qko^\pm$, and for $l=1,2,3$,
$\g_t^l q^\pm(0) := \g_t^{l-1}(\Delta_{\psik}q + \nabla_{\psik}q\cdot\psik_t + \alpha)\vert_{t=0}$. Notice that the solution $q^\pm$ depends on the parameter $s$, but we will omit its dependence for simplicity of notation, and only at the end of the proof we will take limit as $s\rightarrow 0$. The use of the difference quotient $\delta_t^s$ in equations \eqref{E:weak_time_l_minus} and \eqref{E:weak_time_l_plus}, is necessary in order to study the third time-differentiated problem; This is due to the fact $\psibark_{tttt}$ is not well defined when $\bar{h}$ is given in $X_M^\k$.\\

In what follows, we will omit the upper index $\pm$ for simplicity of notation, but we will keep in mind that we must perform the analogous techniques in the now decoupled regions $\Omega^\pm$. Existence of solutions $q\in C([0,T_\k]; L^2(\Omega))\cap L^2([0,T_\k]; H^1(\Omega))$, follow from a Galerkin approximation scheme for parabolic equations, i.e. we consider a solution of the form, 
\begin{equation}\label{E:galerkin_sol}
q^m(t,x) = \sum\limits_{n=1}^m c_n(t) \varphi_n,
\end{equation}
where $\{\varphi_n\}_{n\in\mathbb{N}}$ is a basis of $H^1(\Omega)$ that, for simplicity, we will choose so that,
\begin{equation*}
\left(\Jbark(0) \varphi_n,\varphi_n\right)_{L^2(\Omega)} = 1, \text{ and }\left(\Jbark(0) \varphi_n,\varphi_s\right)_{L^2(\Omega)} = 0,\ \forall s\neq n,
\end{equation*}
and the coefficients $c_n(t)$ satisfy the system of fourth-order differential equations: 
\begin{equation}\label{E:Galerkin_system}
\sum\limits_{n=1}^m \left\{ e^4_{sn}c_n^{(4)}(t) + e^3_{sn}c_n^{(3)}(t) + e^2_{sn}c_n^{(2)}(t) + e^1_{sn}c_n^{(1)}(t) + e^0_{sn}c_n(t)  \right\}= \mathcal{A}_{s} +\mathcal{B}_{s},\ \ \forall s=1,\dots,m, 
\end{equation}
with initial data given by,
\begin{subequations}\label{E:Galerkin_initial_data}
\begin{align}
q^m(0) &= \Qko^m = \sum\limits_{n=1}^m\left(\Qko,\varphi_n\right)_{L^2(\Omega)}\varphi_n,\\
\g_t^kq^m(0) &= {}^\k\hspace{-.02in}Q^m_k = \sum\limits_{n=1}^m {}^\k\hspace{-.02in}Q^m_k(n) \varphi_n  \text{ for } k=1,2,3,
\end{align}
where the coefficients ${}^\k\hspace{-.02in}Q^m_k(n)$ are given by	,
\begin{align*}
{}^\k\hspace{-.02in}Q^m_k(n) &:= \left[\left([\g_t^{k-1},\Jbark]q^m_t,\varphi_n\right)_{L^2} -  \left(\g_t^{k-1}(\labark^i_j (\abark^k_j q^m,_k), \varphi_n,_i\right)_{L^2(\Omega)} - \int_\Gamma \g_t^{k-1}(\Jbark(\k^{-2}q^m + \beta))\varphi_n d\sigma \right.  \nonumber\\
& \left.\qquad + \left(\g_t^{k-1}(\labark^i_j q^m,_i \delta_t^s(\psibark^j)), \varphi_n\right)_{L^2(\Omega)} + \left(\g_t^{k-1}(\Jbark\,\alpha^-),\varphi_n\right)_{L^2(\Omega)}\right]_{t=0}.
\end{align*}
\end{subequations}
The coefficients $e^i_{sn}$ for $i=1,\dots,4$ of the system \eqref{E:Galerkin_system}, are given by,
\begin{align*}
e^4_{sn} &:= \left( \Jbark \varphi_n,\varphi_s\right)_{L^2(\Omega)},\\
e^3_{sn} &:= 3\left(\g_t(\Jbark) \varphi_n,\varphi_s\right)_{L^2(\Omega)} + \left(\labark^i_j\abark^k_j \varphi_n,_k , \varphi_s,_i\right)_{L^2(\Omega)}\\
&\qquad + \int_\Gamma \Jbark\k^{-2}\varphi_n\varphi_sd\sigma - \left(\labark^i_j\varphi_n,_i\delta_t^s(\psibark^j), \varphi_l\right)_{L^2(\Omega)},\\
e^2_{sn} &:= 3\left(\g_t^2(\Jbark) \varphi_n,\varphi_s\right)_{L^2(\Omega)} + 3\left(\g_t(\labark^i_j\abark^k_j) \varphi_n,_k , \varphi_s,_i\right)_{L^2(\Omega)}\\
&\qquad + 3\int_\Gamma \g_t(\Jbark)\k^{-2}\varphi_n\varphi_sd\sigma - 3\left(\g_t(\labark^i_j\delta_t^s(\psibark^j))\varphi_n,_i, \varphi_l\right)_{L^2(\Omega)},\\
e^1_{sn} &:= 3\left(\g_t^3(\Jbark) \varphi_n,\varphi_s\right)_{L^2(\Omega)} + 3\left(\g_t^2(\labark^i_j\abark^k_j) \varphi_n,_k , \varphi_s,_i\right)_{L^2(\Omega)}\\
&\qquad + 3\int_\Gamma \g_t^2(\Jbark)\k^{-2}\varphi_n\varphi_sd\sigma - 3\left(\g_t^2(\labark^i_j\delta_t^s(\psibark^j))\varphi_n,_i, \varphi_l\right)_{L^2(\Omega)},\\
e^0_{sn} &:= \left(\g_t^3(\labark^i_j\abark^k_j) \varphi_n,_k , \varphi_s,_i\right)_{L^2(\Omega)} + \int_\Gamma \g_t^3(\Jbark)\k^{-2}\varphi_n\varphi_sd\sigma\\
&\qquad - \left(\g_t^3(\labark^i_j\delta_t^s(\psibark^j))\varphi_n,_i, \varphi_l\right)_{L^2(\Omega)},
\end{align*}
and,
\begin{equation*}
\mathcal{A}_{s} := \left(\g_t^3(\Jbark\alpha),\varphi_s\right)_{L^2(\Omega)},\ \mathcal{B}_s := - \int_\Gamma \g_t^3(\Jbark\beta)\varphi_s d\sigma.
\end{equation*}
\begin{remark}
Notice that we are considering as a generic model the weak formulation in the domain $\Omega^-$, but an analogous process works for $\Omega^+$, with the inclusion of the integral term on the boundary $\g\Omega$: $\mathcal{C}_s = \int_{\g\Omega}\g_t^3(\Jbark \gamma)\varphi_sd\sigma$. 
\end{remark}
The fundamental theorem of ODEs provides us then with a solution $q^m$ of the form \eqref{E:galerkin_sol}, that satisfies the system \eqref{E:Galerkin_system} in the time interval $[0,T_\k^m]$, which a priori depends on the parameter $m$. Moreover, from linearity, $q^m$ satisfies, 
\begin{subequations}\label{E:weak_time_3_Galerkin}
\begin{align}
&\left(\g_t^3 (\Jbark\, q^m_t),\phi^-\right)_{L^2} +  \left(\g_t^3(\labark^i_j (\abark^k_j q^m,_k), \phi^-,_i\right)_{L^2(\Omega^-)} + \int_\Gamma \g_t^3(\Jbark(\k^{-2}q^m + \beta^-))\phi^- d\sigma  \nonumber\\
& = \left(\g_t^3(\labark^i_j q^m,_i \delta_t^s(\psibark^j)), \phi^-\right)_{L^2(\Omega^-)} + \left(\g_t^3(\Jbark\,\alpha^-),\phi^-\right)_{L^2(\Omega^-)}, \label{E:weak_time_3_minus_Galerkin}\\
&\left(\g_t^3 (\Jbark\, q^m_t),\phi^+\right)_{L^2} +  \left(\g_t^3(\labark^i_j (\abark^k_j q^m,_k), \phi^+,_i\right)_{L^2(\Omega^+)} + \int_\Gamma \g_t^3(\Jbark(\k^{-2}q^m + \beta^+))\phi^+ d\sigma  \nonumber\\
& =  \left(\g_t^3(\labark^i_j q^m,_i \delta_t^s(\psibark^j)), \phi^+\right)_{L^2(\Omega^+)} + \int_{\g\Omega}(\g_t^3(\Jbark\,\gamma)) \phi^+ d\sigma\nonumber\\
&\qquad + \left(\g_t^3(\Jbark\,\alpha^+),\phi^+\right)_{L^2(\Omega^+)},\label{E:weak_time_3_plus_Galerkin}
\end{align}
\end{subequations}
for all $\phi^\pm$ in the finite dimensional space generated by $\{\varphi_n^\pm\}_{n\leq m}$. In addition, given the definition of the initial data \eqref{E:Galerkin_initial_data}, we can integrate in time equation \eqref{E:weak_time_l_Galerkin} as many as three times, to obtain that $q^m$ solves an analogous formulation as \eqref{E:weak_time_l} for 
$l=0,1,2,3$
\begin{subequations}\label{E:weak_time_l_Galerkin}
\begin{align}
&\left(\g_t^l (\Jbark\, q^m_t),\phi^-\right)_{L^2} +  \left(\g_t^l(\labark^i_j (\abark^k_j q^m,_k), \phi^-,_i\right)_{L^2(\Omega^-)} + \int_\Gamma \g_t^l(\Jbark(\k^{-2}q^m + \beta^-))\phi^- d\sigma  \nonumber\\
& = \left(\g_t^l(\labark^i_j q^m,_i \delta_t^s(\psibark^j)), \phi^-\right)_{L^2(\Omega^-)} + \left(\g_t^l(\Jbark\,\alpha^-),\phi^-\right)_{L^2(\Omega^-)}, \label{E:weak_time_l_minus_Galerkin}\\
&\left(\g_t^l (\Jbark\, q^m_t),\phi^+\right)_{L^2} +  \left(\g_t^l(\labark^i_j (\abark^k_j q^m,_k), \phi^+,_i\right)_{L^2(\Omega^+)} + \int_\Gamma \g_t^l(\Jbark(\k^{-2}q^m + \beta^+))\phi^+ d\sigma  \nonumber\\
& =  \left(\g_t^l(\labark^i_j q^m,_i \delta_t^s(\psibark^j)), \phi^+\right)_{L^2(\Omega^+)} + \int_{\g\Omega}(\g_t^l(\Jbark\,\gamma)) \phi^+ d\sigma\nonumber\\
&\qquad + \left(\g_t^l(\Jbark\,\alpha^+),\phi^+\right)_{L^2(\Omega^+)},\ \ \forall \phi^\pm \in \langle\{\varphi^\pm_n\}_{n=1}^m\rangle.\label{E:weak_time_l_plus_Galerkin}
\end{align}
\end{subequations}
The next step is to obtain estimates intependent of $m$. Standard parabolic estimates gives us that $q^m \in L^\infty([0,T^m_\k];L^2(\Omega))\cap L^2([0,T^m_\k];H^1(\Omega))$ with $m$-independent bounds, which allows us to extend $q^m(t)$ beyond $T_\k^m$, and up to an $m$-independent time $T_\k$. Indeed, substituting $\phi = q^m$ on equation \eqref{E:weak_time_l_Galerkin} for $l=0$, we have the bound,
\begin{equation}\label{E:first_parabolic_estimate}
\|q^m(t)\|^2_{L^2} + \|\nabla_{\psibark}q^m\|^2_{L^2_tL^2} + \k^{-2}|q^m|^2_{L^2_tL^2}\leq C(M,q_0),\ \forall t\in [0,T_\k],
\end{equation}
where $C(M,q_0)$ is a constant that depends only on $M$ and $q_0$ (see Lemma \ref{L:first_parabolic_estimates} in the apendix for more details). Moreover, given the regularity of $\psibark$, we can improve the bounds so that $q^m\in L^\infty([0,T_\k]; H^1(\Omega))$ and $q_t^m\in L^2([0,T_\k];L^2(\Omega))$ by using as a test function $\phi = q_t^m$ in \eqref{E:weak_time_l_Galerkin} with $l=0$, and following similar estimates as in the proof of \eqref{E:first_parabolic_estimate}.
Consequently, we found that $q^m\in C([0,T_\k],L^2(\Omega))$, and furthermore, using elliptic estimates, we obtain that $q^m\in L^2([0,T_\k];H^2(\Omega))$, and therefore $q^m\in C([0,T_\k];H^1(\Omega))\cap L^2([0,T_\k];H^2(\Omega))  $.\\

Consider now the first time differentiated problem, equation \eqref{E:weak_time_l_Galerkin} for $l=1$. Using the previously found regularity of $q^m$ and $q_t^m$, and repeating the parabolic regularity arguments for $\tilde{q}^m := q_t^m$, we have that $q_t^m\in C([0,T_\k]; H^1(\Omega))\cap L^2([0,T_\k];H^2(\Omega))$, and $q_{tt}^m\in L^2([0,T_\k]; L^2(\Omega))$. These estimates for $q^m_t$, combined again with elliptic estimates for the non-time-differentiated problem \eqref{E:weak_time_l_Galerkin} for $l=0$, gives us that $q^m\in C([0,T_\k];H^3(\Omega))\cap L^2([0,T_\k];H^4(\Omega))$. Iterating this process one more time for the twice-in-time differentiated problem, we obtain that $q^m_{tt}\in C([0,T_\k]; H^1(\Omega))\cap L^2([0,T_\k];H^2(\Omega))$, and $q^m_{ttt}\in L^2([0,T_\k];L^2(\Omega))$. Elliptic estimates on the one-time-differentiated problem \eqref{E:weak_time_l_Galerkin} for $l=1$, gives us estimates for $q^m_t \in C([0,T_\k];H^3(\Omega))\cap L^2([0,T_\k];H^4(\Omega))$, and therefore, using elliptic regularity once again on the non-time-differentiated problem, we obtain estimates for $q^m\in C([0,T_\k];H^5(\Omega))\cap L^2([0,T_\k];H^6(\Omega))$. The final step follows from the triple time-differentiated problem in the same way, but some terms must be treated carefully, and we address them below. First we will show that the triple time-differentiated approximation $q^m_{ttt}{}^\pm$ to \eqref{E:weak_time_l} satisfies the following inequality:
\begin{align}
\frac{1}{2}\|q_{ttt}^m{}^\pm\|^2_{L^\infty_t L^2} + \|q_{ttt}^m{}^\pm\|^2_{L^2_tH^1} + \k^{-2}|q_{ttt}^m{}^\pm|^2_{L^2_tL^2} \lesssim C_\k M_0 + C_\k \sqrt{t}(\|q_{ttt}^m{}^\pm\|^2_{L^\infty L^2} + \|q_{ttt}^m{}^\pm\|^2_{L^2H^1}),\label{E:q_tttm}
\end{align}
where $M_0 = M_0(q_0^\pm,h_0)$ is a function of the initial data, and $C_\k$ is a constant that depends badly on the smoothing parameter $\k$.
Indeed, the weak form of the triple time differentiated problem can be written as,
\begin{align*}
&\langle \Jbark\, q^m_{tttt},\phi\rangle + \left(\Jbark\,\abark^i_j q^m_{ttt},_i,\abark^k_j\phi,_k\right)_{L^2(\Omega^\pm)} + \int_\Gamma \Jbark\,(\k^{-2}q^m_{ttt} + \g_t^3(v^m\cdot \abark^\top N)\vert_{t=0})\phi d\sigma\\
&= -\left(\Jbark\,v^m\cdot \delta_t^s(\psibark_{ttt}),\phi\right)_{L^2} + \mathcal{Q}^m_3(\phi),
\end{align*}
where in the right-hand side we write the highest order remainder $I:= -\left(\Jbark\,v^m\cdot \delta_t^s(\psibark_{ttt}),\phi\right)_{L^2}$ by itself, and the lower order terms collected in,
\begin{align*}
\mathcal{Q}^m_3(\phi) &= \left([\g_t^3, \labark^i_j\abark^k_j]q^m,_k,\, \phi,_i\right)_{L^2} - \left([\g_t^3,\,\Jbark\, v^m{}]\cdot\psibark_t,\phi\right)_{L^2} - \int_\Gamma [\g_t^3,\, \Jbark](\k^{-2}q^m + \beta^m)\phi\\
&\qquad + \left(\g_t^3(\Jbark\,\alpha^m),\phi\right)_{L^2}, 
\end{align*} 
where $[a,b]c = a(bc) - b(ac)$ denotes the commutator.
We will prove that, choosing $\phi = q^m_{ttt}$, we have the bound
\begin{equation}\label{E:error_control}
\int_0^t (I + \mathcal{Q}_3(q^m_{ttt}))ds \lesssim_M M + C_\k\sqrt{t}(\|q^m_{ttt}\|^2_{L^2H^1} + \|q^m_{ttt}\|^2_{L^\infty L^2}), 
\end{equation}
where $C_\k$ is a constant that depends badly on $\k$. 
The integral $I$ contains as a factor the term $\delta_t^s\psibark_{ttt}$, which, if we were to take the limit as $s\rightarrow 0$ right away, it would depend on $\bar{h}^\k_{tttt}$, which has too many time derivatives on $\bar{h}^\k$ (here lies the necessity to include the discrete operator $\delta_t^s$ into the weak formulation). Nonetheless, this problem is easy to overcome. Recall that, since we know that $q^m_{ttt}$ is in $L^2_tL^2$ and $q^m_{tt}$ is in $L^2_tH^2$, we can use the strong form of the twice-in-time differentiated heat equation to obtain, 
\begin{align*}
I_2:=-\int_0^t \left(\Jbark\,v^m\cdot\delta_t^s\psibark_{ttt}, q^m_{ttt}\right)_{L^2} ds &= -\int_0^t \left(\Jbark\,v^m\cdot \delta_t^s\psibark_{ttt}, -v^m\cdot\psibark_{ttt} + \Delta_\psibark q^m_{tt}\right)_{L^2} + \mathcal{Q}_4,
\end{align*} 
where $\mathcal{Q}_4$ corresponds to the error terms that follow from the $L^2$-inner product of $v^m\cdot\delta_t^s\psibark_{ttt}$ with the remainder terms from the twice-in-time differentiated heat equation. Estimates for these terms follow from integrating by parts in time to remove a derivative from $\delta_t^s\psibark_{ttt}$, and then using standard Cauchy-Schwarz inequality. Therefore, we focus on the higher-order terms, which can be rewritten as,
\begin{align*}
I_3 &:= \int_0^t \left(\Jbark\,v^m\cdot \delta_t^s\psibark_{ttt}, -v^m\cdot\psibark_{ttt} + \Delta_\psibark q^m_{tt}\right)_{L^2}\\
& = \frac{1}{2} \int_0^t \delta_t^s \left[\|\sqrt{\scriptstyle\Jbark} (v^m\cdot\psibark_{ttt})\|^2_0\right]ds - \int_0^t \left(\delta_t^s(\Jbark\,v^m)\cdot \psibark_{ttt}, v^m\cdot\psibark_{ttt}\right)_{L^2}ds\\
&\qquad - \int_0^t \left( \Jbark\,v^m\cdot\psibark_{ttt}, \delta_t^s(\Delta_{\psibark}q^m_{tt})\right)_{L^2}ds - \int_0^t \left( \delta_t^s(\Jbark v^m)\cdot\psibark_{ttt}, \Delta_{\psibark}q^m_{tt}\right)_{L^2}ds\\
&\qquad +\int_0^t \delta_t^s \left[\left( \Jbark\, v^m\cdot\psibark_{ttt}, \Delta_{\psibark}q^m_{tt}\right)_{L^2}\right]ds.    
\end{align*}

Most of these terms are easily bounded as \eqref{E:error_control} by using Cauchy-Schwarz inequality, but the fourth and sixth terms require an intermediate step. It is necessary to first integrate by parts in space to remove a derivative from $q^m_{ttt}$ or $q^m_{tt}$ respectively.

\begin{align*}
I_4 &:= - \int_0^t \left( \Jbark\, v^m\cdot\psibark_{ttt}, \delta_t^s(\Delta_{\psibark}q^m_{tt})\right)_{L^2}ds\\
&= -\int_0^t \int_\Gamma \Jbark\,\nabla q^m\cdot\abark^\top N \bar{h}^\k_{ttt} \delta_t^s(\abark^k_j\abark^i_j q^m_{tt},_i ) N^k d\sigma ds\\
&\qquad + \int_0^t \int_{\Omega^\pm} (\Jbark\,v^m\cdot \psibark_{ttt}\abark^i_j,_i \abark^k_j + (\Jbark\,v^m\cdot\psibark_{ttt}),_i \abark^i_j\abark^k_j)\delta_t^s q^m_{tt},_k + l.o.t.\\
&\leq   \k^{-2}\int_0^t \int_\Gamma \Jbark\,(v^m\cdot N) \bar{h}^\k_{ttt} \delta_t^s q^m_{tt}d\sigma ds + C_\k\sqrt{t}|\bar{h}^\k_{ttt}|_{L^\infty_t L^2}\|\delta_t^sq^m_{tt}\|_{L^2_t H^1} + l.o.t.\\
&\leq C_M \k^{-2}\sqrt{t}|\bar{h}^\k_{ttt}|_{L^\infty_t L^2} (|q^m_{ttt}|_{L^2_tL^2(\Gamma)} + \|q^m_{ttt}\|_{L^2 H^1})\sum\limits_{l\leq 2} \|\g_t^l q^m\|_{L^2_tH^{6-2l}}\\
&\leq C_\k\sqrt{t} (\|q^m_{ttt}\|^2_{L^\infty_t L^2}+\|q^m_{ttt}\|^2_{L^2_t H^1}),
\end{align*}    
where in the last inequality we used the previously found bounds for the terms $\|\g_t^l q^m\|_{L^2_t H^{6-2l}}$. Now, for the sixth term, we integrate by parts one of the derivatives on $q^m_{tt}$,
\begin{align*}
I_5 &:=  \int_0^t \delta_t^s \left[\left( \Jbark\, v^m\cdot\psibark_{ttt}, \Delta_{\psibark}q^m_{tt}\right)_{L^2}\right]ds\\
&= \int_0^t \delta_t^s \left[-\int_\Omega (\Jbark\,v^m\cdot\psibark_{ttt}\abark^i_j),_i \abark^k_j q^m_{tt},_k + \int_\Gamma \Jbark\,v^m\cdot\psibark_{ttt}\abark^i_j N^i \abark^k_j q^m_{tt},_k\right]ds\\
&\leq C\k^{-1}\|q^m_{tt}\|_{L^\infty H^1} |\bar{h}^\k_{ttt}|_{L^\infty L^2(\Gamma)} + \int_0^t\delta_t^s\left[\int_\Gamma \Jbark v^m\cdot N \bar{h}^\k_{ttt} (v^m_{tt}{}^j A^i_jN^i) d\sigma\right] ds + \text{l.o.t.}\\
&\leq C_\k M_0 + \int_0^t\delta_t^s\left[\int_\Gamma \Jbark v^m\cdot N \bar{h}^\k_{ttt} \k^{-2}q^m_{tt} d\sigma\right] ds + \text{l.o.t.}\\
&\leq C_\k M_0 + C\k^{-2}|\bar{h}^\k_{ttt}|_{L^\infty L^2}\|q^m_{tt}\|_{L^\infty H^1} + \text{l.o.t.}\\
&\leq C_\k M_0,
\end{align*}
where in line three we pulled a coefficient $\k^{-1}$ to absorb a derivative from the norm of $\bar{h}^\k_{ttt}$, and in line four we used the strong form of the boundary condition \eqref{E:ALEregularized_dirichlet} differentiated twice in time. Combining all the estimates together, and using the bounds for $\Jbark$ analogous to the ones in Lemma \ref{gJbounds}, we obtain therefore,
\begin{align*}
\frac{1}{2}\|q_{ttt}^m\|^2_{L^\infty L^2} + \|\nabla_{\psibark}q^m_{ttt}\|^2_{L^2L^2} + \k^{-2}|q^m_{ttt}|^2_{L^2L^2} \lesssim C_\k M_0 + C_\k \sqrt{t}(\|q^m_{ttt}\|^2_{L^\infty L^2} + \|q^m_{ttt}\|^2_{L^2H^1}).
\end{align*}
With the bounds for the matrices $\abark$ being analogous as those of Lemma \ref{A_bounds}, combined with a modified Poincar\'e inequality detailed in equation (4.6) of \cite{mHsS2015}, we obtain the desired inequality.

Now, taking $T_\k$ small enough on equation \eqref{E:q_tttm} gives us $m$-independent bounds for $q_{ttt}^m\in L^\infty([0,T_\k],L^2(\Omega^\pm))\cap L^2([0,T_\k],H^1(\Omega^\pm))$, and moreover, from the weak formulation \eqref{E:weak_time_l_Galerkin}, we can obtain as well that $q^m_{tttt}\in L^2([0,T],H^1(\Omega^\pm)^\ast)$ with bounds independent of $m$. As a consequence, $q^m_{ttt}\in C([0,T_\k];L^2(\Omega^\pm))\cap L^2([0,T_\k];H^1(\Omega^\pm))$, and therefore we can use elliptic regularity in succession on the time differentiated problems to obtain the desired $m$-independent estimates for $\g_t^l q^m \in C([0,T_\k];H^{6-2l}(\Omega))\cap L^2([0,T_\k];H^{7-2l}(\Omega))$ for $l=0,1,2,3$.\\

Passing to the limit as $m\rightarrow \infty$, we obtain a weak solution $\qs$ to \eqref{E:weak_time_l} for $l=3$, that, from lower semi-continuity, satisfies that $\g_t^l\qs\in C([0,T_\k];H^{6-2l}(\Omega))\cap L^2([0,T_\k];H^{7-2l}(\Omega))$, and $\qs_{tttt}\in L^2([0,T_\k];H^1(\Omega)^*)$. 
Consider now the the time integral of equation \eqref{E:weak_time_l} for $l=3$. Given the compatibility conditions, we will recover that $\qs$ satisfy \eqref{E:weak_time_l} for $l=2$, and for this case, the term containing $\delta_t^s \psibark_{tt}$ converges strongly to $\psibark_{ttt}$ in $L^2(\Omega)$ as $s\rightarrow 0$. The estimates that we obtained were independent of the parameter $s$, therefore we can pass to the limit as $s\rightarrow 0$, to obtain a weak solution $q^\pm$, such that,
\begin{align}
&\langle\g_t^2 (\Jbark\, q_t),\phi^-\rangle +  \left(\g_t^2(\labark^i_j (\abark^k_j q,_k), \phi^-,_i\right)_{L^2(\Omega^-)} + \int_\Gamma \g_t^2(\Jbark(\k^{-2}q^- + \beta^-))\phi^- d\sigma  \nonumber\\
& = \left(\g_t^2(\labark^i_j q,_i \psibark^j_t), \phi^-\right)_{L^2(\Omega^-)} + \left(\g_t^2(\Jbark\,\alpha^-),\phi^-\right)_{L^2(\Omega^-)},\\
&\langle\g_t^2 (\Jbark\, q_t),\phi^+\rangle +  \left(\g_t^2(\labark^i_j (\abark^k_j q,_k), \phi^+,_i\right)_{L^2(\Omega^+)} + \int_\Gamma \g_t^2(\Jbark(\k^{-2}q^+ + \beta^+))\phi^+ d\sigma  \nonumber\\
& =  \left(\g_t^2(\labark^i_j q,_i \psibark^j_t), \phi^+\right)_{L^2(\Omega^+)} + \int_{\g\Omega}(\g_t^2(\Jbark\,\gamma)) \phi^+ d\sigma\nonumber\\
&\qquad + \left(\g_t^2(\Jbark\,\alpha^+),\phi^+\right)_{L^2(\Omega^+)},\nonumber
\end{align}
holds for all $\phi^\pm\in H^1(\Omega^\pm)$, and that it satisfies that $\g_t^lq^\pm\in C([0,T_\k];H^{6-2l}(\Omega^\pm))\cap L^2([0,T_\k];H^{7-2l}(\Omega^\pm))$. This finishes the proof of existence of weak solutions to the linear problem \eqref{E:ALEregularized_linear}.

\subsubsection{Step 3: Contraction mapping theorem.}
We now define an operator $\Phi_\k$ on $X_M$. Given $\bar{h}\in X_M$ we set
\begin{equation}\label{E:ALE_ht}
\Phi_\k(\bar{h}) = h := h_0 + \int_0^t [v\cdot\bar{n}^\k]^+_- ds,
\end{equation} 
where $v^\pm$ is the solution to the linearized problem \eqref{E:ALEregularized_linear}, and $\bar{n}^\k:= (N-\tau\frac{\t \bar{h}^\k}{(1+H\bar{h}^\k)})$. 
Notice that $\Phi_\k$ maps $X_M$ to itself as proven in Lemma \ref{L:xmtoxm} in the Appendix. To prove that $\Phi_\kappa$ is a contraction, we assume $\bar{h}^1,\ \bar{h}^2$ are given, and consider $\Phi_\kappa(\bar{h}^1) = h^1,\ \Phi_\kappa(\bar{h}^2) = h^2$
with the associated temperature gradients $v_1, v_2.$ The difference $\delta h_t = h_t^1-h_t^2$ is given by,
\begin{equation}
\delta h_t = [\delta v\cdot \bar{n}^\kappa_1]^+_- + [v_2\cdot\delta \bar{n}^\kappa]^+_-,
\end{equation}
where $\bar{n}_\alpha^\kappa = N - \dfrac{\t \bar{h}^\k_\alpha}{(1+\bar{H}^\k\bar{h}^\k_\alpha)}\tau$, for $\alpha=1,2,$
$\delta v = v_1-v_2,$ and $\delta\bar n^\kappa = \bar n^\kappa_1-\bar n^\kappa_2.$
Taking two time derivatives we obtain that, 
\begin{equation*}
\delta h_{ttt} = [\delta v_{tt}\cdot\bar{n}^\k_1]^+_- + [v^2\cdot\bar{n}^\k_{tt}]^+_- + \mathcal{Y},
\end{equation*}
where we gather the lower order terms in $\mathcal{Y}$,
\begin{equation*}
\mathcal{Y} := [v_2{}_{tt}\cdot\delta\bar{n}^\k + \delta v\cdot\bar{n}^\k_1{}_{tt} + 2\delta v_t\cdot\bar{n}^\k_1{}_t + 2v_2{}_t\cdot\delta\bar{n}^\k_t]^+_-. 
\end{equation*}

A straightforward bound using Sobolev and trace inequalities gives:
\begin{align}\label{E:aux_delta_httt}
|\delta h_{ttt}|_{H^k} &\leq C\left(\snorm{\delta v^+_{tt}}_{k+0.5}+\snorm{\delta v^-_{tt}}_{k+0.5}+ (\|v_2^+\|_{k+0.5}+\|v_2^-\|_{k+0.5})|\delta \bar{h}^\kappa_{tt}|_{H^{k+1}}\right), \text{ for } k=0,1,
\end{align}  
where $\delta\bar h^\kappa = \bar h^\kappa_1-\bar h^\kappa_2.$
We now obtain the necessary estimates for $\snorm{\delta v^\pm_{tt}}_{L^2 H^{1.5}}$. We will omit the superscript $\pm$ for simplicity of notation. Taking the difference of equations \eqref{E:ALEregularized_linear_v} for $v_1$ and $v_2$, and taking two time derivatives we obtain,

\begin{equation*}
\delta v_{tt} + \g_t^2\left(\abark_1{}^\top \nabla \delta q\right) = \g_t^2\left(-\delta\left(\abark{}^\top\right)  \nabla q_2\right),
\end{equation*}
with $ \delta\left(\ak{}^\top\right) = \ak_1{}^\top - \ak_2{}^\top.$
Therefore, using the bounds for $\abark$ from Lemma \ref{L:delta_A_bound} we arrive at
\begin{align}\label{E:delta_v}
\snorm{\delta v_{tt}}_{H^{1.5}}&\leq \|\abark_1\|_{1.5} \snorm{\delta q_{tt}}_{2.5} + \|\delta \abark\|_{1.5}\|q_2{}_{tt}\|_{2.5} + \|\delta \abark_{tt}\|_{1.5}\|q_2\|_{2.5}\nonumber\\
&\qquad +\|\abark_1{}_{tt}\|_{1.5}\|\delta q\|_{2.5} +l.o.t.\nonumber\\
&\leq C_M(\|\delta q_{tt}\|_{2.5} + \sqrt{t}|\delta\bar{h}_{tt}|_{L^2 H^2}) 
\end{align}

Furthermore, the difference $\delta q_{tt}$ satisfies the following parabolic problem:
\begin{subequations}
\begin{align*}
\delta q^\pm_{ttt} - \abark^i_{1j}{}^\pm(\abark^k_{1j}{}^\pm\delta q_{tt}^\pm,_k),_i &= f^\pm \text{ in } \Omega^\pm,\\
\delta q^\pm_{tt} &= 0 \text{ on } \Gamma,\\
\delta v^+_{tt}\cdot {\bf N}^+ &= 0 \text{ on }\g\Omega,\\
\delta q^\pm_{tt}(0,{\bf x}) &= 0 \text{ on } \Omega^\pm\times\{t=0\},
\end{align*} 
\end{subequations}
where 
\begin{align}\label{E:SOURCEF}
f &= \g_t^2\left(-\delta v\cdot \bar{w}_1{}_\k - v_2\cdot\delta\psibark_t + \delta\left(\abark^i_j{}\right)(\abark^k_{1j}{} q_2,_k),_i  
+ \abark^i_{2j}{}(\delta\left(\abark^k_j{}\right) q_2,_k),_i\right)\nonumber\\
& + \g_t^2 \abark_1{}^i_j(\abark_1{}^k_j\delta q,_k),_i + \abark_1{}^i_j(\g_t^2\abark_1{}^k_j\delta q,_k),_i + 2\g_t\abark_1{}^i_j(\g_t\abark_1{}^k_j\delta q,k),_i\nonumber\\
& + 2\g_t\abark_1{}^i_j(\abark_1{}^k_j\delta q_t,_k),_i + 2\abark_1{}^i_j(\g_t\abark_1{}^k_j\delta q_t,_k),_i
\end{align} 
Standard parabolic regularity provides:
\begin{equation}\label{E:parabolicestimate}
\|\delta q_{tt}^\pm\|_{L^\infty H^{2}} + \|\delta q_{tt}^\pm\|_{L^2 H^{2.5}}\leq C_\k\snorm{f^\pm}_{L^2 H^{0.5}},
\end{equation}
where estimates for the source term $f^\pm$ given in~Lemma \ref{L:sourceF} in the appendix gives us,
\begin{equation*}
\|\delta q^\pm_{tt}\|_{L^\infty H^{2}} + \|\delta q^\pm_{tt}\|_{L^2_t H^{2.5}}\leq C_M\k^{-1}\sqrt{T_\k}\ \mathcal{S}(\delta q,\delta\bar{h}^\k)^{1/2},
\end{equation*}
where $\mathcal{S}(\delta q,\delta\bar{h}^\k)$ is the high-order norm defined in \eqref{S-local} evaluated in $\delta q$ and $\delta\bar{h}^\k$.
Repeating this argument for the parabolic problems associated to $\delta q_{t}$ and $\delta q$, and adding all the inequalities together we obtain that,
\begin{equation}
\newe^+(\delta q,\delta\bar{h}^\k) + \newe^-(\delta q,\delta\bar{h}^\k) + \int_0^t (\newd^+(\delta q,\delta\bar{h}^\k) + \newd^-(\delta q,\delta\bar{h}^\k))ds \lesssim_M \k^{-1} \sqrt{T_\k} \mathcal{S}(\delta q, \delta\bar{h}^\k),\label{E:norm_balance_contraction}
\end{equation}
where we recall the definitions of $\newe^\pm$ and $\newd^\pm$ from Section \ref{S:norms} as the higher-order norms of $\g_t^l\delta q$. A small enough time $T_\k$ allows us to absorb the terms in the right side of equation \eqref{E:norm_balance_contraction} with the same norms of $\g_t^l\delta q$ as the in the left side, leaving only the boundary norms,
\begin{align*}
&\newe^+(\delta q,\delta\bar{h}^\k) + \newe^-(\delta q,\delta\bar{h}^\k) + \int_0^t (\newd^+(\delta q,\delta\bar{h}^\k) + \newd^-(\delta q,\delta\bar{h}^\k))ds\\
& \lesssim_M \k^{-1} \sqrt{T_\k}(\newe^\ppp_{\loc}(\delta\bar{h}^\k) + \int_0^t \newd^\ppp_{\loc}(\delta\bar{h}^\k) ds).
\end{align*} 
Therefore, using this together with~\eqref{E:delta_v} and \eqref{E:aux_delta_httt}, we obtain

\begin{align*}
|\delta h_{ttt}|^2_{L^\infty L^2(\Gamma)} + |\delta h_{ttt}|^2_{L^2H^1(\Gamma)} \lesssim_M \k^{-1} \left(\newe^\ppp_{\loc}(\delta\bar{h}^\k) + \int_0^t \newd^\ppp_{\loc}(\delta\bar{h}^\k) ds\right).
\end{align*}

An analogous estimate for $\g_t^l \delta h$ in $L^\infty H^{6-2l}\cap L^2 H^{6.5-2l}$ for $l=0,1,2$ respectively, allows us to conclude that

\begin{equation*}
\newe^\ppp_{\loc}(\delta h) + \int_0^t \newd^\ppp_{\loc}(\delta h)ds \lesssim_M\k^{-1}\sqrt{T_\k}(\newe^\ppp_{\loc}(\delta\bar{h}^\k) + \int_0^t\newd^\ppp_{\loc}(\delta\bar{h}^\k)ds).
\end{equation*} 
We see that $\Phi_\kappa$ is a contraction for $T_\kappa$ sufficiently small and the theorem follows from the contraction mapping theorem.

\end{proof}

\subsection{Definition of the energy functionals}\label{s:energydef}

The key ingredient to the proof of the main theorems is the introduction of the higher-order \textbf{weighted} energy $\mathscr{E}_\k(t)$, which will be shown it controls the norm $\mathcal{S}$ evaluated on the solutions of the regularized problem \eqref{E:ALEregularized_linear}, which we define as,
\begin{equation}
\mathcal{S}_\k(t) := \mathcal{S}(q^\pm,h^\k).
\end{equation} 
Note that $(q^\pm,h^\k)$ is the solution to the regularized problem \eqref{E:ALEregularized} obtained in Theorem \ref{T:smooth_solutions}, and therefore $\mathcal{S}_\k(t)$ is bounded for all $t\in[0,T_\k]$.  

\subsection*{The weight functions $W^\pm$}

To define the energy associated with the two-phase Stefan problem, we will introduce {\em weight} functions $W^\pm(t,x)$, that will allow us to successfully include the non-degeneracy condition~\eqref{E:localTaylorsign} in our theory.
Let $W^\pm:\Omega^\pm\to\R$ be a solution to the following Dirichlet problem:
\begin{subequations}\label{E:W}
\begin{align}
\Delta W^\pm & = 0,\ x\in\Omega^\pm,\label{E:W_laplacian}\\
W^\pm & = \frac{e^{(-\lambda_1 + \eta)t}}{\g_Nq^\pm},\ x\in\Gamma\label{E:W_dirichlet},\\
W^+ &\equiv \frac{e^{(-\lambda_1 + \lambda_1^+ + \eta) t}}{|c_1^+|},\ \ x\in \partial\Omega,\label{W_dirichlet_top}
\end{align}
\end{subequations}
where $c_1^+:=( q_0^+,\varphi_1^+)_{L^2}$, defined in \eqref{E:c_1}. Note that $W^\pm>0$ in $\Omega^\pm$ by the maximum principle and the Rayleigh-Taylor assumption~\eqref{E:localTaylorsign}, 
which, by continuity guarantees that $\g_Nq^\pm>0$ at least for short times. 
The long-time behavior of $W^\pm$ is very important for the proof of global stability and it depends on the difference between the first eigenvalues of the Dirichlet-Laplacian in 
the regions $\Omega^+$ and $\Omega^-.$ On the other hand, the short-time behavior of $W^\pm$ is easily characterized in the following lemma. 
\begin{lemma}[Local estimates for $W^\pm$]\label{l:wbounds_local}
Suppose the Taylor sign condition \eqref{E:localTaylorsign} holds for some $\delta >0$, and assume there exists a constant $M>0$ such that $S(t) \leq M$. 
Then there exist positive constants $c_{\delta, M} >0,\ C_{\delta, M}>0$, such that the solution $W^\pm$ to \eqref{E:W} satisfy,
$$c_{\delta,M} \leq W^\pm\leq C_{\delta, M}.$$
\end{lemma}
\begin{proof}
Notice that for short time, the Rayleigh-Taylor condition \eqref{E:localTaylorsign} gives us the following {\em upper bound on $\g_N q^\pm$}
\begin{align*}
 |\g_N q^\pm| \leq C\|q^\pm\|_{2.25} \leq C S(t) \leq C_M.
\end{align*}

Similarly, by the fundamental theorem of calculus
\[
 |\g_N q^\pm(t)| \ge  |\g_N q_0^\pm| - \big|\int_0^t\partial_s(\g_N q^\pm)(s)\,ds\big| \ge |\g_N q_0^\pm|  - C\sqrt t \|q_t\|_2 \ge \delta - C\sqrt t M.
\]
Therefore, for small times we have the lower bound
\[
|\g_N q^\pm| \ge c_{\delta,M}>0.
\]
By the maximum principle conclude that 
\begin{align*}
c_{\delta,M} e^{(-\lambda_1 +\eta)t} \leq \min_{x\in\Gamma}W^\pm(t,x) \le \max_{x\in\Gamma}W^\pm(t) \leq C_{\delta,M} e^{(-\lambda_1 +\eta)t}. 
\end{align*}
Taking $t$ so small that $1/2 \leq e^{(-\lambda_1 +\eta)t} \leq 1$, we obtain the result. 
\end{proof}

\subsection*{The natural energy $\mathscr{E}_\k(t)$}
The following definition of the ``natural" energy is seemingly technical, but as it will become apparent in Section~\ref{ab_energy_lemma} it is precisely the natural higher-order positive definite quantity arising from 
an integration-by-parts argument.
\begin{definition}[Higher-order weighted energy $\E_\k(t)$ and dissipation functional $\D_\k(t)$]\label{D:natural_energy}
Let $q^\pm:\Omega^\pm\to\R$, $h:\Gamma\to\R$, and recall the cut off function $\mu$ from \eqref{eq:mudefinition}. We set: 
\begin{align}\label{E:energy}
\E_\k^\pm(t) &:= \frac{1}{2}\sum\limits_{a+2b\le 5}\left(\|\mu^{1/2}\t^a\partial_t^b v^\pm\|^2_{L^{2,W^\pm}} + \k^2e^{(-\lambda_1+\eta)t} |\sqrt{r_\k^\pm} \t^a\g_t^b v^\pm\cdot n^\k|^2_{L^2(\Gamma)} \right)\nonumber\\
& + \frac{1}{2}\sum\limits_{a+2b\le 6}\left(\|\mu^{1/2}(\t^a\partial_t^b q^\pm + \t^a\partial_t^b\psik^\pm\cdot v^\pm)\|^2_{L^{2,W^\pm}} +e^{(-\lambda_1 + \eta)t}|a_\k\t^a\partial_t^b \Lambda_\k h|^2_{L^2(\Gamma)}\right)\nonumber\\
&\frac{1}{2}\sum\limits_{|a|+2b\le 5}\|(1-\mu)^{1/2}\g^a\partial_t^b v^\pm\|^2_{L^2(\Omega^\pm)} +  \frac{1}{2}\sum\limits_{|a|+2b\le 6}\|(1-\mu)^{1/2}(\g^a\partial_t^b q^\pm + \g^a\partial_t^b\psik^\pm\cdot v^\pm)\|^2_{L^2(\Omega^\pm)},
\end{align}
\begin{align}
\D_\k^\pm(t) &:= \sum\limits_{a+2b\le 6}\left(\|\mu^{1/2}\t^a\partial_t^b v^\pm\|^2_{L^{2,W^\pm}} + \k^2e^{(-\lambda_1+\eta)t}|\sqrt{r_\k^\pm}\t^a\g_t^b v^\pm\cdot n^\k|^2_{L^2(\Gamma)}\right)\nonumber\\
&+ \sum\limits_{a+2b\le 5}\left(\|\mu^{1/2}(\t^a\partial_t^{b+1} q^\pm + \t^a\partial_t^{b+1}\psik^\pm\cdot v^\pm)\|^2_{L^{2,W^\pm}} + e^{(-\lambda_1 +\eta)t}|a_\k\t^a\partial_t^{b+1} \Lambda_\k h|^2_{L^2(\Gamma)}\right)\\ 
&+ \sum\limits_{|a|+2b\le 6}\|(1-\mu)^{1/2}\g^a\partial_t^b v^\pm\|^2_{L^2(\Omega^\pm)} + \sum\limits_{|a|+2b\le 5}\|(1-\mu)^{1/2}(\g^a\partial_t^{b+1} q^\pm + \g^a\partial_t^{b+1}\psik^\pm\cdot v^\pm)\|^2_{L^2(\Omega^\pm)}.
\end{align}
where $J_\k:= \det\nabla\psik$ is the determinant of the Jacobian, $g_\k$ is defined by  $g_\k := (\t h^k)^2+(1+H(x)h^\k)^2$, and the coefficients $r_\k^\pm(t,x) := (\g_N q^\pm)^{-1}J_\k^{-2}g_\k$ and $a_\k(t,x) := J_\k^{-1}(1+Hh^\k)$.
 
\end{definition}

We remind the reader that the horizontal derivatives $\t$ are defined in Section~\ref{SS:TANGENTIAL}.
We introduce the {\em total energy}: 
\begin{equation}\label{e:mathscrE_kappa}
\mathscr{E}_\k(t) := \sup\limits_{0\leq s\leq t}\E_\k^+(s) + \sup\limits_{0\leq s\leq t}\E_\k^-(s) + \int_0^t (\D_\k^+(s) + \D_\k^-(s))ds.
\end{equation}


\begin{remark} 
For the proof of the \emph{local well-posedness theorem}, we will show that the following a-priori energy estimate holds,
\be\label{E:EnergyInequality}
\mathscr{E}_\k(t) \leq \mathscr{E}_\k(0) + C \sqrt{t} P(\mathscr{E}_\k(t)),
\ee
where 
$P(\cdot)$ is some polynomial of degree greater than or equal to one, but that it does not depend on $\k$.
A simple continuity argument then yields Theorem~\ref{T:main}. A more careful energy estimate combined with a maximum principle argument gives us the \emph{global stability result}, 
which is explained in Section \ref{C:global}.
\end{remark}

\subsection{Local-in-time energy control}\label{S:ENERGYEQUIV}
Assuming that the Rayleigh-Taylor condition~\eqref{E:localTaylorsign} holds, we shall prove in this section that the control over the derivatives of $q^\pm$ and $h^\k$ provided by the norm $\mathcal{S}_\k(t)$ is dominated by the energy $\mathscr{E}_\k(t)$ defined by~\eqref{e:mathscrE_kappa}.
\begin{prop}\label{P:local_energyequivalence}
Suppose the Taylor sign condition \eqref{E:localTaylorsign} holds for some $\delta>0$, then the norm $\mathcal{S}_\k(t)$ is equivalent to $\mathscr{E}_\k(t)$ in the sense that, 
\begin{equation}\label{EleqS}
\mathcal{S}_\k(t) \leq P(\mathscr{E}_\k(t)),
\end{equation}
for any $t$ on the interval of definition of $\mathscr{E}_\k(t)$ and $\mathcal{S}_\k(t)$ and for $P$ a universal polynomial as described in \ref{ss:Notation}.
\end{prop}

\begin{proof}

The proof of this Proposition follows exactly as the proof of Proposition 2.4 of \cite{mHsS2011}, but with the weights $W^\pm$. First, the contribution from the boundary terms is easy to bound, since, instead of having the weight $\g_N q$ in our energy, we have the coefficient $e^{(-\lambda_1 +\eta)t}$. Also we have the estimate, 
\begin{equation*}
J^{-1}(1+H h^\k) = 1 + O(|h^\k|) \geq \frac{1}{2},
\end{equation*}
where we used the characterization of $J$ from subsection \ref{gJbounds} of the apendix, the bound for the curvature $H$, and that for short time $h^\k$ is small. Recall that $H$ is the curvature of the smooth reference curve, thus it is \emph{$\sigma$-close} to the original initial interface.
Therefore
\begin{align*}
\newe^\ppp_{\loc}(t) &\leq C \sum\limits_{a+2b\leq 6} |J^{-1}(1+Hh^\k)\t^a\g_t^b \Lambda_\k h|^2_{L^\infty_t L^2(\Gamma)},\\
\newd^\ppp_{\loc}(t) &\leq  C \sum\limits_{a+2b\leq 5} |J^{-1}(1+Hh^\k)\t^a\g_t^{b+1} \Lambda_\k h|^2_{L^2(\Gamma)}.
\end{align*}
Second, note that both $\mathcal{S}_\k(t)$ and $\mathscr{E}_\k(t)$ have terms of the form $\|\mu^{1/2}\t^{a}\g_t^b v\|_0^2$, with the difference that in $\mathscr{E}_\k(t)$, the norm has the weights $W^\pm$. The upper and lower bounds for $W^\pm$ in Lemma \ref{l:wbounds_local}, gives us that these terms satisfy inequality \eqref{EleqS} directly.

Finally, we need to show that we can control all derivatives in the interior by controling only the tangential ones that appear on $\mathscr{E}_\k(t)$, but this process is analogous as the proof of estimate $(a)$ from Proposition 2.4 of \cite{mHsS2011}, multiplying and dividing by the coefficient $W^\pm$ on the corresponding integrals over $\Omega^\pm$. We obtain,
\begin{align*}
\sum\limits_{|a|+2b\leq 6} \|\g^a\g_t^b q^\pm\|_{L^\infty_t L^2(\Omega^\pm)} &\leq P\left(\frac{\mathscr{E}_\k(t)}{\inf W^\pm}\right)\leq P(\mathscr{E}_\k(t)),\\
\sum\limits_{|a|+2b\leq 6} \|\g^a\g_t^b q^\pm\|_{L^2_tH^{0.5}(\Omega^\pm)} &\leq P\left(\frac{\mathscr{E}_\k(t)}{\inf W^\pm}\right)\leq P(\mathscr{E}_\k(t)),
\end{align*}
where the last inequalities follow again from the local estimates \ref{l:wbounds_local} for $W^\pm$. This concludes the proof.
\end{proof}

\subsection{Derivation of the energy identities}\label{S:derivation}

For the various notations used in this section we encourage the reader to consult Subsct.~\ref{ss:Notation} -~\ref{SS:TANGENTIAL}.
\begin{lemma}\label{ab_energy_lemma}
Let $(q^\pm,h^\k)$ be a smooth solution to the two-phase Stefan problem given by Theorem \ref{T:smooth_solutions}, on the time interval $[0,T_\k]$. Then the following energy identity holds:
\begin{align}\label{E:energy_identity1}
\frac{d}{dt}\mathscr{E}_\k(t) = \frac{d}{dt}(\E_\k^+(t)+\E_\k^-(t)) + \D_\k^+(t) + \D_\k^-(t) = \mathcal{R}(t),
\end{align}
where the right-hand term $\mathcal{R}(t)$ is an error given explicitly in Lemma \ref{S:error_terms} of the Appendix.
\end{lemma}

\begin{proof}
Apply the operator $\t^a\partial_t^b$ to \eqref{E:ALEregularized_v}, for $0\leq a+2b\leq 6$, multiply by $\t^a\partial_t^b v W\mu$, and integrate over $\Omega^\pm$ respectively. The inclusion of the factor $W^\pm$ is very important as it allows us to recover the positive definite boundary energy by providing a common factor on both regions to form the difference $h_t = v^+\cdot\tilde{n}^\k - v^-\cdot\tilde{n}^\k$. We obtain the identity:
\begin{equation}\label{E:innerproduct}
( \t^a\partial_t^b v + \t^a\partial_t^b\ {}^\k \hspace{-.06in} A^i_j  q,_i + {}^\k \hspace{-.06in} A^i_j\t^a\g_t^b q,_i,\ \t^a\partial_t^b v W\mu)_{L^2(\Omega^\pm)} = \mathcal{R}^\pm_1
\end{equation}
where $\mathcal{R}^\pm_1$ is the error term that contains the lower order derivatives arising from the application of the product rule to ${}^\k \hspace{-.06in} A^i_j q,_i,$ integrated over the regions $\Omega^\pm$ respectively. The process then follows the same methodology as in the proof of Proposition 3.1 from \cite{mHsS2013}, or the proofs of Lemmas 2.2 and 2.3 from \cite{mHsS2011}, with the added weights $W^\pm$. 
A few new error terms appear while integrating by parts, as some derivatives fall on $W^\pm.$ Additional new terms appear to the fixed exterior boundary $\g\Omega$, but there are no new ideas in the process. 
We describe the effect of this weight on the boundary terms, but omit the rest for brevity.

Analogously to the appendix of \cite{mHsS2013}, equation (A.1), when the derivatives $\t^a\g_t^b$ hit the matrix $A$ we have the following identity, 
\begin{equation*}
\t^a\g_t^b\ {}^\k \hspace{-.06in} A^i_j = -{}^\k \hspace{-.06in} A^i_l\t^a\g_t^b \Psi^l_\k,_m {}^\k \hspace{-.06in} A^m_j - \sum\limits_{1\leq s\leq a\atop 1\leq l\leq b}c_{sl}\t^{a-s}\g_t^{b-l}\ {}^\k \hspace{-.06in} A^i_k\t^s\g_t^l \Psi^k_\k,_m {}^\k \hspace{-.06in} A^m_j.
\end{equation*}
Therefore, the second term of equation \eqref{E:innerproduct} becomes,
%
\begin{align*}
&(\t^a\partial_t^b\ {}^\k \hspace{-.06in} A^T \nabla q,\ \t^a\g_t^b v \mu W)_{L^2(\Omega^\pm)} = - \int_{\Omega^\pm} {}^\k \hspace{-.06in} A \t^a\partial_t^b  
\nabla \psik {}^\k \hspace{-.06in} A \nabla q \t^a\partial_t^b v W\mu - \mathcal{R}_2^\pm
\end{align*}
where $\mathcal{R}^\pm_2$ is the error term containing the lower order derivatives hitting $\nabla\psik$. We will specify briefly a computation over the two regions $\Omega^\pm$ as there is a small difference when integrating by parts over the region $\Omega^+$ since it has an exterior fixed boundary $\g\Omega$. We have, after integrating by parts,
\begin{align*}
- \int_{\Omega^+} {}^\k \hspace{-.06in} A \t^a\partial_t^b 
\nabla\psik {}^\k \hspace{-.06in} A \nabla q \t^a\partial_t^b v W\mu&= -\int_\Gamma {}^\k \hspace{-.06in} A^l_i \t^a\partial_t^b \Psi^i_\k {}^\k \hspace{-.06in} A^k_jq,_l\t^a\partial_t^b v^j W (-N^k) + \int_{\g\Omega}(\t^a\g_t^b\psik \cdot v) (\t^a\g_t^b v\cdot {\bf N}^+) W \\
& + \int_{\Omega^+}\t^a\partial_t^b \Psi^i_\k {}^\k \hspace{-.06in} A^k_j({}^\k \hspace{-.06in} A^l_i q,_l\t^a\partial_t^b v W\mu),_k
\end{align*}
Where the boundary condition \eqref{E:ALEregularized_neumann_top} implies that the term in $\g\Omega$ is an error term, that we call $\mathcal{R}^{a,b}_{\g\Omega_1}$, and, by the definition of $W^+$ on $\Gamma$ \eqref{E:W_dirichlet}, the integral on $\Gamma$ becomes,
\begin{align*}
&-\int_\Gamma {}^\k \hspace{-.06in} A^l_i \t^a\partial_t^b \Psi^i_\k {}^\k \hspace{-.06in} A^k_jq,_l\t^a\partial_t^b v^j W^+ (-N^k) = \int_\Gamma \t^a\partial_t^b \Psi^i_\k {}^\k \hspace{-.06in} A^l_i q,_l \t^a\partial_t^b v^j {}^\k \hspace{-.06in} A^k_j N^k \frac{e^{(-\lambda_1 +  \eta)t}}{\partial_N q^+} \\
&= \int_\Gamma \t^a\partial_t^b \Psi^i_\k {}^\k \hspace{-.06in} A^l_i(( \nabla q^+\cdot \tau) \tau^l + ( \nabla q^+\cdot N) N^l) \t^a\partial_t^b v^j {}^\k \hspace{-.06in} A^k_j N^k \frac{e^{(-\lambda_1 +  \eta)t}}{\partial_N q^+}\\
&= e^{(-\lambda_1 +  \eta)t}\int_\Gamma (\t^a\partial_t^b \psik\cdot {}^\k \hspace{-.06in} A^TN) (\t^a\partial_t^b v^+\cdot {}^\k \hspace{-.06in} A^TN) d\sigma + \mathcal{R}^{a,b}_{\Gamma_1}\\
&= e^{(-\lambda_1 +  \eta)t}\int_\Gamma J_\k^{-2}g_\k (\t^a\partial_t^b \psik\cdot n^\k) (\t^a\partial_t^b v^+\cdot n^\k) d\sigma + \mathcal{R}^{a,b}_{\Gamma_1}{}^+,
\end{align*}
where 
\begin{equation*}
\mathcal{R}^{a,b}_{\Gamma_1}{}^+:= e^{(-\lambda_1+\eta)t}\int_\Gamma \t^a\g_t^b\psik^i {}^\k \hspace{-.06in} A^l_i \tau^l \t(-\k^2 v^+\cdot \ak^\top N + \k^2\beta^+)(\t^a\g_t^b v^+\cdot n^\k)\frac{J_\k^{-1}\sqrt{g_\k}}{\g_N q^+}d\sigma,
\end{equation*}

since the term $ \nabla q^+\cdot N = \g_N q^+$, and $q^+$ along $\Gamma$ is given by \eqref{E:ALEregularized_dirichlet}. The last equalities follows from the geometric identity ${}^\k \hspace{-.06in} A^T N = J_\k^{-1}\sqrt{g_\k} n^\k$, where $n^\k$ is the normal vector to the moving domain. 

 We obtain therefore,
\begin{align*}
& \|\mu^{1/2}\t^a\partial_t^b v^+\|^2_{L^{2,W^+}} + \frac{1}{2}\frac{d}{dt}\|\mu^{1/2}(\t^a\partial_t^b q + \t^a\partial_t^b\psik\cdot v^+)\|^2_{L^{2,W^+}}\\
& + e^{(-\lambda_1 +  \eta)t}\int_\Gamma J_\k^{-2}g_\k(\t^a\partial_t^b \psik\cdot n^\k) (\t^a\partial_t^b v^+\cdot n^\k) d\sigma\\
& + \k^2 e^{(-\lambda_1+\eta)t}\int_\Gamma r^+_\k |\t^a\g_t^b v^+\cdot n^\k|^2d\sigma\\
& = \mathcal{R}^+_\aabb +\mathcal{R}^\aabb_{\Gamma_1}{}^+ + \mathcal{R}^\aabb_{\g\Omega_1}{}^+
\end{align*}
where we have gathered all the residue terms of the interior into $\mathcal{R}^+_\aabb$. An analogous process now with $W^-$ in the region $\Omega^-$ gives, 
\begin{align*}
& \|\mu^{1/2}\t^a\partial_t^b v^- \|^2_{L^{2,W^-}} + \frac{1}{2}\frac{d}{dt}\|\mu^{1/2}(\t^a\partial_t^b q^- + \t^a\partial_t^b\psik\cdot v^-)\|^2_{L^{2,W^-}}\\
& - e^{(-\lambda_1 +  \eta)t}\int_\Gamma J_\k^{-2}g_\k(\t^a\partial_t^b \psik\cdot n^\k) (\t^a\partial_t^b v^-\cdot n^\k) d\sigma\\
& +  \k^2 e^{(-\lambda_1+\eta)t}\int_\Gamma r_\k^- |\t^a\g_t^b v^-\cdot n^\k|^2d\sigma\\
&= \mathcal{R}^-_\aabb + \mathcal{R}^\aabb_{\Gamma_1}{}^-
\end{align*}
Hence, adding together the terms from both regions we obtain,
\begin{align}
& \|\mu^{1/2}\t^a\partial_t^b v^+\|^2_{L^{2,W^+}} + \frac{1}{2}\frac{d}{dt}\|\mu^{1/2}(\t^a\partial_t^b q^+ + \t^a\partial_t^b\Psi^+\cdot v^+)\|^2_{L^{2,W^+}}\nonumber\\
& +  \|\mu^{1/2}\t^a\partial_t^b v^-\|^2_{L^{2,W^-}}+\frac{1}{2}\frac{d}{dt}\|\mu^{1/2}(\t^a\partial_t^b q^- + \t^a\partial_t^b\Psi^-\cdot v^-)\|^2_{L^{2,W^-}}\nonumber\\
&+\k^2 e^{(-\lambda_1+\eta)t} \left(|\sqrt{r_\k^+} (\t^a\g_t^b v^+\cdot n^\k)|^2_{L^2(\Gamma)} + |\sqrt{r_\k^-} (\t^a\g_t^b v^-\cdot n^\k)|^2_{L^2(\Gamma)}\right)\nonumber\\
& + e^{(-\lambda_1 +  \eta)t}\int_\Gamma J_\k^{-2}g_\k\left(\t^a\partial_t^b \psik^+\cdot n^\k\ \t^a\partial_t^b v^+\cdot n^\k - \t^a\partial_t^b \psik^-\cdot n^\k\ \t^a\partial_t^b v^-\cdot n^\k\right) d\sigma\nonumber\\
&  = \mathcal{R}^+_\aabb +\mathcal{R}^-_\aabb + \mathcal{R}^\aabb_{\Gamma_1}{}^+ + \mathcal{R}^\aabb_{\Gamma_1}{}^- + \mathcal{R}_{\g\Omega_1}^\aabb{}^+.\label{E:addingtogether}
\end{align}
On the last boundary term of equation \eqref{E:addingtogether} we can factor the $(\t^a\g_t^b \psik^\pm\cdot n^\k)$-term, since on the boundary $\Gamma$, both $\psik^+$ and $\psik^-$ are the same, obtaining
\begin{align}
I_\Gamma &:= e^{(-\lambda_1 +  \eta)t}\int_\Gamma J_\k^{-2}g_\k(\t^a\partial_t^b \psik\cdot n^\k)\left( \t^a\partial_t^b v^+\cdot n^\k - \t^a\partial_t^b v^-\cdot n^\k\right)d\sigma\nonumber\\
&= e^{(-\lambda_1 +  \eta)t}\int_\Gamma J_\k^{-2}g_\k(\t^a\partial_t^b \psik\cdot n^\k) [\t^a\partial_t^b v\cdot n^\k]^+_-d\sigma\label{E:boundary_ei}
\end{align}
Recall that on the boundary $\Gamma$, $\psik(t,x) = x + h^\k(t,x) N(x)$, therefore,
\begin{equation*}
\t^a\g_t^b \psik = \t^a\g_t^b x + \t^a\g_t^bh^k N + h^\k \t^a\g_t^b N + \sum\limits_{s,l}\t^{a-l}\g_t^{b-s} h^\k \t^l\g_t^s N,
\end{equation*}
and the normal vector,
\begin{equation*}
n^\k = \frac{-\t h^\k\tau + (1+H(x)h^\k)N}{\sqrt{(\t h^k)^2+(1+H(x)h^\k)^2}} = (-\t h^\k\tau + (1+H(x)h^\k)N)\frac{1}{\sqrt{g_\k}}.
\end{equation*}

Then the boundary term \eqref{E:boundary_ei} can be rewritten as 
\begin{equation*}
I_\Gamma = e^{(-\lambda_1 +  \eta)t}\int_\Gamma J_\k^{-2} \t^a\g_t^bh^k (1+Hh^\k)^2 [\t^a\g_t^b v\cdot (N - \frac{\t h^\k}{(1+Hh^\k)}\tau)]^+_- - \mathcal{R}_{\Gamma_2}^\aabb
\end{equation*}
where 
\begin{align*} 
\mathcal{R}_{\Gamma_2}^\aabb &= -e^{(-\lambda_1 +  \eta)t}\int_\Gamma J_\k^{-2}(1+Hh^\k) ((h^\k \t^a\g_t^b N + \t^a\g_t^b x + \sum\limits_{s,l}\t^{a-l}\g_t^{b-s} h^\k \t^l\g_t^s N) \cdot (-\t h^\k\tau\\
&\qquad + (1+Hh^\k)N))[\t^a\g_t^b v\cdot (N-\frac{\t h^\k}{(1+Hh^\k)}\tau)]^+_-.
\end{align*}
Recall from equation \eqref{E:ALE_ht}, 
$$h_t = [v\cdot(N - \frac{\t h^\k}{(1+Hh^\k)}\tau) ]^+_- ,$$
then we have,
\begin{equation*}
I_\Gamma= e^{(-\lambda_1 +  \eta)t}\int_\Gamma a_\k^2 (\t^a\g_t^b h^\k) (\t^a\g_t^b h_t) -\mathcal{R}_{\Gamma_2}^\aabb - \mathcal{R}_{\Gamma_3}^\aabb
\end{equation*}
where
\begin{align*}
\mathcal{R}_{\Gamma_3}^\aabb &:= e^{(-\lambda_1 +  \eta)t}\int_\Gamma a_\k^2\t^a\g_t^{b} h^\k \left([v\cdot \t^a\g_t^b( N - \frac{\t h^\k}{(1+Hh^\k)}\tau]^+_-\right.\\
&\qquad \left. + \sum\limits_{s,l}c_{sl}[\t^{a-l}\g_t^{b-s} v\cdot \t^l\g_t^s (N - \frac{\t h^\k}{(1+Hh^\k)}\tau)]^+_-\right).
\end{align*}
First, observe that in this error there is a higher-order term hidden when the highest order derivatives hit the tangential derivative of $h^\k$, i.e. a term of the form $\t^{a+1}\g_t^b h^\k$, and must be considered carefully 
when we prove the energy estimates. Second, notice that one of the factors has the regularized $h^\k$, but the other factor is not regularized, since it comes from equation \eqref{E:ALE_ht}. Therefore, we must commute the smoothing operator $\Lambda_\k$ from $h^\k$ to $h_t$, to form the quadratic energy term. Recall that this operator commutes with the tangential derivatives and therefore
\begin{align*}
I_\Gamma &= e^{(-\lambda_1 +  \eta)t}\int_\Gamma a_\k^2 (\t^a\g_t^b \Lambda_\k\Lambda_k h) (\t^a\g_t^b h_t) - \mathcal{R}_{\Gamma_2}^\aabb - \mathcal{R}_{\Gamma_3}^\aabb\\
&= e^{(-\lambda_1 +  \eta)t}\int_\Gamma a_\k^2 (\t^a\g_t^b \Lambda_k h)(\t^a\g_t^b \Lambda_k h_t) - \mathcal{R}_{comm}^\aabb -  \mathcal{R}_{\Gamma_2}^\aabb - \mathcal{R}_{\Gamma_3}^\aabb,
\end{align*} 
where $\mathcal{R}_{comm}^\aabb$ is a commutation error given by,
\begin{align*}
\mathcal{R}_{comm}^\aabb &= -e^{(-\lambda_1 +  \eta)t}\int_\Gamma \t^a\g_t^b(\Lambda_\k h)[\Lambda_\k, a_\k^2\t^a\g_t^b]h_t.
\end{align*}

Pulling out a time derivative and grouping the error terms, we recover the positive definite energy term,
$$I_\Gamma = \frac{1}{2}\frac{d}{dt}\left(e^{(-\lambda_1 +  \eta)t}\int_\Gamma a_\k^2 (\t^a\partial_t^b \Lambda_\k h)^2\right) - \mathcal{R}_{\Gamma_4}^\aabb$$
where $\mathcal{R}_{\Gamma_4}^\aabb$ is given by,
\begin{align*}
\mathcal{R}^\aabb_{\Gamma_4} &= \frac{1}{2}\int_\Gamma \g_t\left(e^{(-\lambda_1 +  \eta)t} a_\k^2\right)(\t^a\g_t^b\ \Lambda_\k h)^2 + \mathcal{R}_{comm}^\aabb + \mathcal{R}_{\Gamma_2}^\aabb + \mathcal{R}_{\Gamma_3}^\aabb.
\end{align*}
Collecting all together we have, 
\begin{align*}
& \|\mu^{1/2}\t^a\partial_t^b v^+\|^2_{L^{2,W^+}} + \frac{1}{2}\frac{d}{dt}\|\mu^{1/2}(\t^a\partial_t^b q^+ + \t^a\partial_t^b\psik^+\cdot v^+)\|^2_{L^{2,W^+}}\\
& +\|\mu^{1/2}\t^a\partial_t^b v^-\|^2_{L^{2,W^-}} + \frac{1}{2}\frac{d}{dt}\|\mu^{1/2}(\t^a\partial_t^b q^- + \t^a\partial_t^b\psik^-\cdot v^-)\|^2_{L^{2,W^-}}\\
&+\k^2 e^{(-\lambda_1+\eta)t} \left(|r_\k^+ (\t^a\g_t^b v^+\cdot n^\k)|^2_{L^2(\Gamma)} + |r_\k^- (\t^a\g_t^b v^-\cdot n^\k)|^2_{L^2(\Gamma)}\right)\nonumber\\
& + \frac{1}{2}\frac{d}{dt}\left(e^{(-\lambda_1 + \eta)t}\int_\Gamma a_\k^2 (\t^a\partial_t^b \Lambda_\k h)^2 d\sigma\right)\\
& = \mathcal{R}^+_\aabb  + \mathcal{R}^-_\aabb + \mathcal{R}_\Gamma^\aabb + \mathcal{R}^\aabb_{\g\Omega},
\end{align*}
where $\mathcal{R}^\aabb_{\Gamma}$ contains all the error terms in the boundary $\Gamma$ and $\mathcal{R}^\aabb_{\g\Omega}$ the errors in $\g\Omega$. An analogous analysis can be done to obtain energy identities of the second type by considering the differential operator $\t^a\g_t^{b+1}$ to equation \eqref{E:ALEregularized_linear_v}, and multiplying by $\t^a\g_t^{b} v^\pm W^\pm\mu$. For the interior derivatives we consider the differential operator $\g^a\g_t^b$, where $a$ is now a multi index, and 
$\g^a = \g_{x_1}^{a_1}\g_{x_2}^{a_2}$ is a combination of derivatives in all cartesian directions. The result follows by summing over the corresponding values of $a,b$. See \cite{mHsS2013} for more details. 
\end{proof}
\begin{remark}
In contrast to~\cite{mHsS2013}, when the time derivative is applied to the weight $e^{(-\lambda_1 +  \eta)t}$, we obtain the obvious inequality:
\begin{equation*}
(-\lambda_1 + \eta)e^{(-\lambda_1 + \eta)t}\int_\Gamma a_\k^2(\t^a\g_t^b\ \Lambda_\k h)^2 <0, \ \ \text{ for }\eta<\lambda_1.
\end{equation*}
Therefore, we do not need to prove estimates for this energy-critical term as it is sign-definite with a favorable sign.
In particular, many of the technical complications from~\cite{mHsS2013} are eliminated. 
\end{remark}


\subsection{Energy estimates for the local theory}\label{S:energyestimates}

In this section we prove energy estimates for the solutions of the regularized problem \eqref{E:ALEregularized_linear}. The aim is to obtain $\kappa$-independent estimates, and therefore a uniform-in-$\k$ time of existence for our family of regularized solutions. We will accomplish this by using the energy identity \eqref{E:energy_identity1}, and bounding the error terms in $\mathcal{R}(t)$.
As a first step, we prove short time apriori bounds for $ \nabla W^\pm$ and $W_t^\pm$. 
\begin{lemma}\label{L:WSHORTTIME}
Under the assumptions of Theorem \ref{T:main}, the derivatives of the weight functions $W^\pm$ satisfy the following bounds,
\begin{align*}
\| \nabla W^\pm\|_{L^\infty(\Omega^\pm)} + \|W_t^\pm\|_{L^\infty(\Omega^\pm)}&\leq C(1+P(\E_\k(t))).
\end{align*}
\end{lemma} 
\begin{proof}
Without loss we will show only the estimates in $\Omega^+.$ Since $W^+$ satisfies \eqref{E:W}, standard elliptic estimates and Sobolev inequality give
\begin{align*}
\| \nabla W^+\|_{L^\infty} &\leq C\|W^+\|_{2.25} \leq Ce^{(-\lambda_1 + \eta)t} |\frac{1}{\g_N q^+}|_{2} + Ce^{(-\lambda_1 + \lambda_1^+ + \eta)t}\\
&\leq Ce^{(-\lambda_1 + \lambda_1^+ + \eta)t}(1+ |\frac{1}{\g_N q^+}|_{2}).
\end{align*}
On the other hand 
\begin{align*}
|\t^2(\frac{1}{\g_N q^+})|_{L^2} &\leq \frac{|\t^2\g_N q^+|_{L^2}}{\delta^2} + \frac{|\t\g_N q^+|^2_{L^4}}{\delta^3}\\
&\leq \frac{\|q^+\|_{3.5}}{\delta^2} + \frac{\|q^+\|^2_{3}}{\delta^3} \leq CP(\E_\k^+(t)),
\end{align*}
where we used the lower bound for $\g_N q^+$ from \eqref{E:localTaylorsign}. The other components of the $H^2(\Gamma)$ norm follow similarly, therefore, 
\begin{align*}
\| \nabla W^+\|_{L^\infty} &\leq Ce^{(-\lambda_1 + \lambda_1^+ + \eta)t}(1+P(\E_\k^+(t))).
\end{align*}
Taking time short enough so that $e^{(-\lambda_1 + \lambda_1^+ + \eta)t}\leq 2$ gives the desired bound. Next, $W_t^+$ satisfies the following problem, 
\begin{subequations}\label{E:Wt}
\begin{align}
\Delta W_t^+ &= 0 \ \text{ in } \Omega^+\\
W_t^+ &= e^{(-\lambda_1 + \eta)t}\left(\frac{(-\lambda_1 + \eta)}{\partial_N q^+} - \frac{ \partial_N q_t^+}{(\partial_N q^+)^2}\right)\ \text{ on } \Gamma\\
W_t^+ &= (-\lambda_1 + \lambda_1^+ + \eta)e^{(-\lambda_1 + \lambda_1^+ + \eta)t}\ \text{ on } \partial\Omega^+_{\text{fixed}}.
\end{align}
\end{subequations}
On the interface $\Gamma$, 
\begin{align*}
|W_t^+|_{L^\infty(\Gamma)} &\leq Ce^{(-\lambda_1 + \eta)t}\left(\frac{|-\lambda_1 + \eta|}{\delta} + \frac{|\g_N q_t^+|}{\delta^2}\right)\\
&\leq Ce^{(-\lambda_1 + \eta)t}(1+\|q_t^+\|_{2.75}) \leq Ce^{(-\lambda_1 + \eta)t}(1+P(\E_\k^+(t)))
\end{align*}
Therefore, by the maximum principle,
\begin{align*}
\|W_t^+\|_{L^\infty} &\leq Ce^{(-\lambda_1 +\lambda_1^+ +\eta)t}(1+P(\E_\k^+(t))),
\end{align*}
which again, for a sufficiently short time, gives the result.
\end{proof}

\begin{lemma}{Higher regularity for $q$.}\label{L:magical_lemma} We have the following inequality,
\begin{equation}\label{E:magic}
\int_0^t \|q^\pm(s)\|_7^2 + \|\psik^\pm(s)\|_7^2 ds \leq C\E_\k(t).
\end{equation}
\begin{proof}
The proof of this Lemma follows the same argument detailed in the proof of Lemma 2.4 of \cite{mHsS2011}, so we will omit it for economy. 
\end{proof}
\end{lemma}

\begin{prop}\label{P:energyestimate}

For each $\k>0$, the energy function $\E_\k(t)$ is continuous in $[0,T_\k]$, and there exists a constant $C$ and a polynomial $P$, both independent of $\kappa$, such that the following bound holds:
\begin{equation}\label{E:energyestimate}
\mathscr{E}_\kappa(t) \leq C\mathscr{E}_\k(0) + C\sqrt{t}P(\E_\kappa(t)), \text{ for all }\ t\in[0,T_\k] .
\end{equation}

\end{prop}

\begin{proof}

First we show that the map $t \rightarrow \E_\k(t)$ is continuous on $[0,T_\k]$.
The continuity of the terms of the type $L^2([0,t]; H^s)$ follows from the fundamental theorem of calculus, and, for the norms $\|\g_t^l q^\pm(t)\|_{H^{6-2l}}$, continuity follows from the parabolic regularity estimates that we carried out in section \ref{S:higherregularitylinear} for the solution of the regularized equation \eqref{E:ALEregularized}. 

Using the definition of $h_t$ \eqref{E:HKAPPAEVOLUTION}, we can obtain the continuity of $h_t \in C([0,T_\k];H^{4}(\Gamma))$ and $h_{tt} \in C([0,T_\k];H^{2}(\Gamma))$, from the trace estimates for $v^\pm$ and $v_t^\pm$ respectively, combined with the already known continuity of $q$ and $q_t$. Moreover, the continuity of $h_{ttt}\in C([0,T_\k];L^2(\Gamma))$ can be obtained from the three-time-differentiated definition of $h_t$ in \eqref{E:HKAPPAEVOLUTION}, since, for a fixed $\k>0$, we have control of $\k|v_{ttt}\cdot n^\k|_{L^2_t L^2(\Gamma)}$, so the fundamental theorem of calculus gives us the desired continuity. 

Finally, from the higher regularity estimate \eqref{E:magic}, we have the norm $h\in L^2([0,T_\k];H^{6.5}(\Gamma))$. This estimate, combined with the control of $\k|\t^6 v\cdot n^\k|_{L^2_t L^2}$, which implies that $h_t\in L^2([0,T_\k];H^6(\Gamma))$, gives us by interpolation that $h\in C([0,T_\k];H^6(\Gamma))$.\\

Now we will prove the estimate \eqref{E:energyestimate}. The argument consists of carefully bounding the error terms from the energy identity \eqref{E:energy_identity1}. Most of these estimates are done analogously as in the proof of Proposition 2.5 from \cite{mHsS2011}, so we will address first the new error terms that appear as a consequence of having two interacting phases. These terms appear in the last line of the definitions of $\mathcal{R}^\pm_\aabb,\ \widetilde{\mathcal{R}}^\pm_\aabb$ with derivatives of the weight function $W^\pm$. Consider for example on the positive region $\Omega^+$ the errors,  
\begin{align*}
\mathcal{R}_5^\aabb{}^+ &:= \int_0^t \int_{\Omega^+} (\t^a\partial_t^b q^+ + \t^a\partial_t^b\Psi^+\cdot v^+) A^+\t^a\partial_t^b v^+  \nabla W^+ \mu\\
\mathcal{R}_6^\aabb{}^+ &:= \frac{1}{2}\int_0^t \int_{\Omega^+} (\t^a\partial_t^b q^+ + \t^a\partial_t^b\Psi^+\cdot v^+)^2 W_t^+\mu.
\end{align*}
By Lemma~\ref{L:WSHORTTIME}, $\mathcal{R}_5^\aabb{}^+$ can be easily bounded as follows:
\begin{align*}
\big|\mathcal{R}_5^\aabb{}^+\big| &\leq \int_0^t \|\mu^{1/2}(\t^a\partial_t^b q^+ + \t^a\partial_t^b\Psi^+\cdot v^+)\|_{L^{2,W^+}}\|A^+\|_{L^\infty}\|\mu^{1/2}\t^a\partial_t^b v^+\|_{L^{2,W^+}} \left\|\frac{ \nabla W^+}{W^+}\right\|_{L^\infty}\\
&\leq C\int_0^t \E_\k^+(s)^{1/2}\D_\k^+(s)^{1/2} (1+P(\E^+_\kappa(s)))ds \\
&\leq P(\sup\limits_{0\leq s\leq t} \E^+_\k(s))\int_0^t \D^+_\k(s)^{1/2}ds 
\leq \sqrt{t} P(\mathscr{E}_\k(t))
\end{align*}
where we used the bounds for $A^+$ from \ref{A_bounds}, H\"older's inequality, and the definition of the energy $\mathscr{E}_\k(t)$. In the same way we bound $\mathcal{R}_6^\aabb{}^+:$
\begin{align*}
\big|\mathcal{R}_6^\aabb{}^+\big| &\leq \frac{1}{2}\int_0^t \E^+_\kappa(s)\left\|\frac{W^+_t}{W^+}\right\|_{L^\infty}ds
\leq\ t\ P(\sup\limits_{0\leq s\leq t} \E^+_\kappa(s)).
\end{align*}
The remaining error terms are dealt with in the same way as in the proof of Proposition 2.5 of \cite{mHsS2011}, since the only difference, the weight $W^\pm$, can be bounded in $L^\infty$ using~Lemma \ref{l:wbounds_local}.


\subsection*{Boundary estimates} 

We write the boundary error $\mathcal{R}_\Gamma$ as a sum of its integral terms:
\begin{equation*}
\mathcal{R}_\Gamma = \mathcal{R}_{\Gamma_1}^\aabb + \mathcal{R}_{\Gamma_2}^\aabb + \mathcal{R}_{\Gamma_3}^\aabb + \mathcal{R}_{\Gamma_4}^\aabb + \mathcal{R}_{comm}^\aabb
\end{equation*}


\subsubsection*{Estimates for $\mathcal{R}_{\Gamma_4}^\aabb$}\label{SSS:boundary1}

The third term, as we have mentioned previously, is very distinctive, since it has the same order as the energy. We will see now that, because of our choice of $W^\pm$, this is not a problem.
\begin{align*}
\int_0^t \mathcal{R}_{\Gamma_4}^\aabb(s)ds & = \frac{1}{2}\int_0^t\int_\Gamma \g_t\left(e^{(-\lambda_1 + \eta)t} a_\k^2\right)(\t^a\g_t^b\ \Lambda_\k h)^2 \\
 & =  (-\lambda_1 +\eta)\int_0^t \int_\Gamma e^{(-\lambda_1 +\eta)s} a_\k^2(\t^a\g_t^b \Lambda_\k h)^2 \\
& \ \ \  + \int_0^t \int_\Gamma e^{(-\lambda_1 +\eta)s} \g_t(a_\k^2)(\t^a\g_t^b \Lambda_\k h)^2,
\end{align*}
where we recall the meaning of $H$ from Section~\ref{SS:FIXEDDOMAIN}, and $a_\k$ from Definition \ref{D:natural_energy}.
The first term is negative, since $-\lambda_1 +\eta <0$ for $\eta$ small enough, so we can eliminate it from the estimates, and the second can be bounded by,
\begin{align*}
& \int_0^t \int_\Gamma e^{(-\lambda_1 +\eta)s} \g_t(a_\k^2)(\t^a\g_t^b \Lambda_\k h)^2 \leq \int_0^t \E_\k(s) \left|\frac{\g_t(a_\k^2)}{a_\k^2}\right|_{L^\infty}\\
&\leq \int_0^t \E_\k(s) 2\left|\frac{\g_tJ_\k}{J_\k} + \frac{Hh_t^\k}{(1+Hh^\k)}\right|_{L^\infty} \leq t\ P(\mathscr{E}_\k(t)),
\end{align*} 
where we used the estimates for $J_\k$ from \ref{gJbounds}, Sobolev embedding, and the evolution equation~\eqref{E:HKAPPAEVOLUTION} for $h^\kappa_t$.


\subsubsection*{Estimates for $\mathcal{R}_{\Gamma_2}^\aabb$ and $\mathcal{R}_{\Gamma_3}^\aabb$}

In these errors there are \emph{problematic} terms that contain higher-order derivatives of the normal vector to the reference curve. If the reference curve were the initial domain (which is $H^6$), we would have at most 6 tangential derivatives in $L^2$, which is not enough to bound $\t^6 N$, since $N$ contains one derivative of the parametrization. 
Instead, the reference domain is a $C^\infty$ curve, which is $\sigma$-close to the initial domain, described by a height function $h_0\in H^6$, and $\sigma$ is a fixed, but small parameter. 
Our energy estimates will therefore depend on the parameter $\sigma$ as well as the time of existence, but that will not interfere with the proof. 
We write the first error terms as:
\begin{align}
\int_0^t \mathcal{R}_{\Gamma_2}^\aabb(s)ds &= \int_0^t \int_\Gamma e^{(-\lambda_1+\eta)s}J_\k^{-2}(1+Hh^\k) h^\k(\t^a\g_t^b N\cdot \tau)\t h^\k \t^a\g_t^b h_t\nonumber\\
&\quad - \int_0^t \int_\Gamma e^{(-\lambda_1+\eta)s}J_\k^{-2}(1+Hh^\k) h^\k(\t^a\g_t^b N\cdot \tau)\t h^\k [v\cdot \t^a\g_t^b(N-\frac{\t h^\k}{1+Hh^\k}\tau)]^+_-\nonumber\\
&\quad +\text{ l.o.t.} \label{E:gamma_error}
\end{align}
where ``l.o.t." stands for a collection of integral terms of the form, 
\begin{equation}\label{E:lotdefinition}
\int_0^t\int_\Gamma f(x,t)dxdt,\ \text{ or } \int_0^t\int_{\Omega^\pm} g(x,t)dxdt,
\end{equation}
that can be bounded to obtain estimate \eqref{E:energyestimate} using H\"older's inequality in a straightforward way. The first term of \eqref{E:gamma_error}, after taking one copy of  the smoothing operator $\Lambda_\kappa$ from $h^\k=\Lambda_\k\Lambda_\k h$ 
and applying it to $h_t$ gives
\begin{align}
I_{\Gamma_2}&:= \int_0^t \int_\Gamma e^{(-\lambda_1+\eta)s}J_\k^{-2} (1+Hh^\k)(\Lambda_\k h)(\t^a\g_t^b N\cdot \tau)\t h^\k (\t^a\g_t^b \Lambda_\k h_t) + I_{\text{comm}},\label{E:boundary_aux1}
\end{align}
where $I_{\text{comm}}$ is a lower-order commutator error given by,
\begin{equation*}
I_{\text{comm}} = \int_0^t \int_\Gamma \Lambda_\k h \big[\Lambda_\k,\ e^{(-\lambda_1+\eta)s}J_\k^{-2} (1+Hh^\k)(\t^a\g_t^b N\cdot \tau)\t h^\k \t^a\g_t^b\big] h_t. 
\end{equation*}
The first term in \eqref{E:boundary_aux1} has too many tangential derivatives over $\Lambda_\k h_t$ since the index $a\leq 6$, and the natural energy $\mathscr{E}_\k(t)$ can control at most 5 tangential derivatives of $\Lambda_\k h_t$. We integrate by parts to pass one derivative to the other terms which are lower order, 
\begin{align*}
I_{\Gamma_2} &=-  \int_0^t \int_\Gamma \t\left(e^{(-\lambda_1+\eta)s}J_\k^{-2}(1+Hh^\k) (\Lambda_\k h)(\t^a\g_t^b N\cdot \tau)\t h^\k\right) (\t^{a-1}\g_t^b \Lambda_\k h_t) + I_{\text{comm}}\\
&\leq C\sqrt{t}|e^{(-\lambda_1+\eta)t/2}\sqrt{a_\k} \t^{a-1}\g_t^b \Lambda_\k h_t|_{L^2_tL^2} \left|\frac{\t\left(J_\k^{-2}(1+Hh^\k) (\Lambda_\k h)(\t^a\g_t^b N\cdot \tau)\t h^\k\right)}{J_\k^{-1}(1+H h^\k)}\right|_{L^\infty_tL^2}\\
&\quad + \text{l.o.t.}\\
& \leq \sqrt{t}\ \mathscr{E}_\k(t)^{1/2}P(\E_\k(t)).
\end{align*}
For the second integral term of \eqref{E:gamma_error}, we notice that there is a higher-order term when the operator $\t^a\g_t^b$ inside the braket hits the $\t h^\k$ coefficient. Integrating by parts we get,
\begin{align*}
I_{\Gamma_3} &:= \int_0^t \int_\Gamma e^{(-\lambda_1+\eta)s}J_\k^{-2}(1+Hh^\k) h^\k(\t^a\g_t^b N\cdot \tau)\t h^\k [(v\cdot \tau) \frac{\t^{a+1}\g_t^b h^\k}{1+Hh^\k}]^+_- + \text{l.o.t.}\\
& \ = -\int_0^t \int_\Gamma e^{(-\lambda_1+\eta)s}\t\left(J_\k^{-2}(1+Hh^\k) h^\k(\t^a\g_t^b N\cdot \tau)\t h^\k\frac{[(v\cdot \tau)]^+_-}{(1+Hh^\k)}\right)  \t^{a}\g_t^b h^\k + \text{l.o.t}.\\
&\leq C t\ |e^{(-\lambda_1+\eta)t/2}J_\k^{-1}(1+Hh^\k)\t^a\g_t^b h^\k|_{L^\infty_tL^2} P(\E_\k(t))  + \text{l.o.t.}\\
& \leq C\ t\ \mathscr{E}_\k(t)^{1/2}P(\E_\k(t)).
\end{align*}
We now examine the error $\mathcal{R}_{\Gamma_3}^\aabb$. 
Notice that a similar situation occurs when the derivatives $\t^a\g_t^b$ inside the braket hits the term $\t h^\k$, but observe that now we 
cannot just integrate by parts since there is another higher-order factor $\t^{a+1}\g_t^b h^\k$ inside the integral. We extract a full derivative instead as follows: 
\begin{align*}
\int_0^t \mathcal{R}_{\Gamma_3}^\aabb(s)ds &=\int_0^t \int_\Gamma e^{(-\lambda_1+\eta)s}J_\k^{-2}(1+Hh^\k)^2 \t^a\g_t^b h^\k \frac{[v\cdot \tau]^+_-}{(1+Hh^\k)} \t^{a+1}\g_t^b h^\k\\
&=  \int_0^t \int_\Gamma e^{(-\lambda_1+\eta)s}J_\k^{-2}(1+Hh^\k)[v\cdot \tau]^+_- \frac{1}{2}\t(\t^a\g_t^b h^\k)^2\\
&= - \frac{1}{2}\int_0^t \int_\Gamma e^{(-\lambda_1+\eta)s}\t\left(J_\k^{-2}(1+Hh^\k)[v\cdot \tau]^+_-\right)(\t^a\g_t^b h^\k)^2\\
&\leq C\ t\ |e^{(-\lambda_1+\eta)t/2}J_\k^{-1}(1+Hh^\k)\t^a\g_t^b h^\k|^2_{L^\infty_tL^2} \biggr\vert \frac{\t\left(J_\k^{-2}(1+Hh^\k)[v\cdot \tau]^+_-\right)}{J_\k^{-2}(1+Hh^\k)^2}\biggr\vert_{L^\infty_tL^\infty}\\
&\leq\ t\ \mathscr{E}_\k(t)P(\E_\k(t)).
\end{align*}


\subsubsection*{Estimates for the commutation error $\mathcal{R}_{comm}^\aabb$}

\begin{align*}
 \int_0^t \mathcal{R}_{comm}^\aabb(s)ds &=-\int_0^t \int_\Gamma e^{(-\lambda_1+\eta)s} \t^a\g_t^b(\Lambda_\k h)[\Lambda_\k,  a_\k^2\t^a\g_t^b]h_t.
\end{align*}
To bound this term, we will use the commutation estimate described in Lemma 5.1 from \cite{CoSh_Kacov2010},
\begin{equation*}
|\Lambda_\k(f\t g) - f\Lambda_\k\t g|_0 \leq C|f|_{W^{1,\infty}(\Gamma)}|g|_0,
\end{equation*}
where the constant $C$ does not depend on $\k$.
Therefore, 
\begin{align*}
\int_0^t\mathcal{R}_{comm}^\aabb(s)ds &\leq \int_0^t e^{(-\lambda_1+\eta)s}|\t^a\g_t^b(\Lambda_\k h)|_0 |[\Lambda_\k,\ a_\k^2 \t^a\g_t^b] h_t|_0\\
&\leq P(\E_\k(t))\int_0^t e^{(-\lambda_1+\eta)s/2}|[\Lambda_\k,\ a_\k^2 \t^a\g_t^b] h_t|_0
\end{align*}
and 
\begin{align*}
|[\Lambda_\k,\ a_\k^2 \t^a\g_t^b] h_t|_0 &\leq | a_\k^2|_{W^{1,\infty}(\Gamma)} |\t^{a-1}\g_t^b h_t|_0
\end{align*}
Notice that, since there is not a smoothing operator $\Lambda_\k$ on $h_t$, we cannot bound it directly by the dissipation term $\D_\k(t)$. Instead, we use the definition of $h_t$ from \eqref{E:ALE_ht}. We have,
\begin{align*}
|\t^{a-1}\g_t^b h_t|_0 &\leq |[\t^{a-1}\g_t^b v \cdot (N- \frac{\t h^\k}{(1+H h^\k)}\tau)]^+_-|_0 + \left|\frac{[v\cdot \tau]^+_-}{1+Hh^\k}\right|_{L^\infty} |\t^a\g_t^b h^\k|_0 + l.o.t.
\end{align*}
We notice that the second term, multiplied by $e^{(-\lambda_1+\eta)t/2}$, is an energy term that is bounded by the $L^\infty$-norm in time. 
The first term now has the same number of derivatives on $v$ as the dissipation term $\D_\k(t)$, yet it cannot be bounded by it as it is a boundary norm. We have instead the following bounds using the control of the norm $\mathcal{S}_\k(t)$ with $\E_\k(t)$,
\begin{align*}
|[\t^{a-1}\g_t^b v \cdot (N- \frac{\t h^\k}{(1+H h^\k)}\tau)]^+_-|_0 &\leq \left(|[\t^{a-1}\g_t^b A^T  \nabla q]^+_-|_0 + |[A^T \t^{a-1}\g_t^b \nabla q]^+_-|_0\right) |(N- \frac{\t h^\k}{(1+H h^\k)}\tau)|_{L^\infty}\\
& \quad + l.o.t. \\
&\leq P(\E_\k(t))\left(|\t^a\g_t^b h^\k|_0 + \|\g_t^b q\|_{H^{6.5-2b}}\right).
\end{align*}
Therefore, 
\begin{align*}
\int_0^t \mathcal{R}_{comm}^\aabb(s)ds  &\leq P(\E_\k(t)) \int_0^t e^{(-\lambda_1 +\eta)s/2}\left(|\t^a\g_t^b h^\k|_0 + \|\g_t^b q\|_{H^{6.5-2b}}\right) ds + \text{l.o.t.}\\
& \leq P(\E_\k(t))(t\E_\k(t)^{1/2}+\sqrt{t}P(\E_\k(t)))
 \leq \sqrt{t}\ P(\E_\k(t)).
\end{align*}
Estimates for $\mathcal{R}^\aabb_{\Gamma_1}$ are straightforward, since it has a factor $\k^2$, and therefore we can bound the higher-order factors inside of the integral in $L^2(\Gamma)$ and use the extra terms in the energy with the $\k$-coefficients to obtain the desired estimate.   
The error $\widetilde{\mathcal{R}}_{\Gamma}(t)$ can be studied similarly as $\mathcal{R}_\Gamma$, but it is actually simpler, since $\widetilde{\mathcal{R}}_\Gamma$ does not contain the conflicting term $\mathcal{R}_{\Gamma_4}^\aabb$. Thus, simple integration by parts and Sobolev estimates suffice for obtaining the desired bounds. For the outside boundary error terms $\mathcal{R}^+_{\g\Omega}$ and $\widetilde{\mathcal{R}}^+_{\g\Omega}$, we capitalize on the boundary condition \eqref{E:ALEregularized_neumann_top}, and the smoothness of $\g\Omega$, which allows us to remove derivatives from the critical terms, and put them in the smooth normal vector. The estimates for the interior errors $\mathring{\mathcal{R}}^\pm$, follow the same methodology as the terms in $\mathcal{R}^\pm$. We will omit these estimates as there is no new ideas involved in their bounds.

\subsubsection{Estimates for the errors with $\alpha$ and $\gamma$.}\label{SS:alpha_gamma_errors}
\begin{align*}
\mathcal{R}^\aabb_{\alpha} &:= \int_{\Omega^\pm}(\t^a\g_t^b q + \t^a\g_t^b\psik\cdot v)(\t^a\g_t^b\alpha) W\mu, \text{ for } a+2b\leq 6,\\
\widetilde{\mathcal{R}}^\aabb_{\alpha} &:= \int_{\Omega^\pm}(\t^a\g_t^{b+1}q + \t^a\g_t^{b+1}\psik\cdot v)(\t^a\g_t^b\alpha) W\mu, \text{ for } a+2b\leq 5.
\end{align*}
Let us look at the case when $a=6,\ b=0$, 
\begin{align*}
\mathcal{R}^{6,0}_{\alpha} &\leq \int_{\Omega^\pm}(\t^7q + \t^7\psik\cdot v)\t^5\alpha W\mu + \text{l.o.t.}\leq C(\|q\|_7 + \|\psik\|_7)\|\alpha\|_5 + \text{l.o.t.}\\
& \leq \E_\k(t)^{1/2}\|\alpha\|_5 + \text{l.o.t.}
\end{align*}
Now, recall that from the definition of $\alpha$ in \eqref{E:alpha} we get, 
\begin{align*}
\|\alpha\|_5 &\leq \|\alpha_0\|_5 + \int_0^t \|r(s)\|_5 ds \lesssim \|q_0\|_6 + \sqrt{T_\k} \|r\|_{L^2_tH^5} \lesssim \|q_0\|_6 + C\sqrt{T_\k}. 
\end{align*} 
Therefore the bound follows directly. Similarly for the other terms $a,b$, and for the error terms involving $\gamma$, 
\begin{align*}
\mathcal{R}^\aabb_{\g\Omega} &= -\int_{\g\Omega}(\t^a\g_t^b q^+ + \t^a\g_t^b\psik\cdot v)\t^a\g_t^b \gamma\ W + \text{ l.o.t.},\\
\widetilde{\mathcal{R}}_{\g\Omega}^\aabb &= \int_{\g\Omega}( \t^a\g_t^{b+1} q^+ + \t^a\g_t^{b+1}\psik\cdot v)\t^a\g_t^b \gamma\ W + \text{l.o.t.},
\end{align*}
since, after integrating by parts a tangential derivative, control for $|\t^{a-1}\g_t^b \gamma|_{L^2(\Gamma)}$ follows from the definition of $\gamma$ in \eqref{E:gamma}. Consider for example $a=6,\ b=0$, then,
\begin{align*}
|\t^5 \gamma|_{L^\infty_tL^2} &\lesssim \|\tilde{\mathcal{G}}^0\|_{L^\infty_t H^{5.5}} + \sqrt t \|\tilde{\mathcal{G}}^1\|_{L^2_t H^{5.5}}ds + t^{3/2} \|\tilde{\mathcal{G}}^2\|_{L^2_t H^{5.5}} \\
& \leq C(1+\sqrt{T_\k}),
\end{align*} 
where we have used the bound~\eqref{E:GBOUND}.
Therefore, for $T_\k$ small enough,
\begin{align*}
\int_0^t \mathcal{R}^{6,0}_{\g\Omega}ds &\lesssim\ t\ \E_\k(t)^{1/2} (1+T_\k) \lesssim\ t\ \E_\k(t)^{1/2}.
\end{align*}
The other terms of $\mathcal{R}^\aabb_{\g\Omega}$ and $\widetilde{\mathcal{R}}^\aabb_{\g\Omega}$ follow directly using the same ideas, so we will omit them for brevity.  

\end{proof}

\subsection{Proof of  Theorem \ref{T:main}} \label{S:PROOF1}

\subsubsection*{Existence of solutions to the two-phase Stefan problem}

The energy estimates of Proposition \ref{P:energyestimate}, and a continuation argument like the one described in Section 9 of \cite{MR2208319}, gives us that there exists a $\kappa$-independent time $T>0$ such that the following bound holds,
\begin{align}\label{E:energy_estimate_final}
\mathscr{E}_\k(t) \leq C\mathscr{E}_\k(0) \ \text{ for all } t\in [0,T],
\end{align}
where $C$ is independent of $\k$. The energy control of Lemma \ref{P:local_energyequivalence}, and the definition of the initial data $\Qko$, gives us the bound,
\begin{equation*}
\mathcal{S}_\kappa(t)\leq P(\mathscr{E}_\k(t)) \leq P(\E_\k(0)) \leq C\mathcal{S}(0) \ \text{ for all } t\in [0,T],
\end{equation*}
where we note that the polynomial $P$ is also independent of $\kappa$. 
We conclude then that for all $\k>0$, the solutions to the $\k$-problem exist at least until a $\k$-independent time $T>0$, and they are bounded by,
\begin{equation}\label{E:uniform_bound}
\mathcal{S}_\k(t) \leq C\mathcal{S}(0)\ \ \forall\ t\in[0,T].
\end{equation}   

Consider now $\k=\frac{1}{n}$ and let $n\rightarrow \infty$. Let us define the reflexive Hilbert space, 
\begin{align*}
X(T) &:= \left\{(\tilde{q}^\pm,\tilde{h})|\ \g_t^l \tilde{q}^\pm \in L^2([0,T];H^{6.5-2l}(\Omega^\pm)) \text{ for } l=0,1,2,3,\right.\\
&\qquad \qquad\left. \g_t^s \tilde{h}_t \in L^2([0,T];H^{5-2l}(\Gamma)) \text{ for } s=0,1,2\right\},
\end{align*}
then the uniform bound \eqref{E:uniform_bound} implies that there exists a subsequence that converges weakly in $X(T)$ to a limit which we call $(q^\pm, h)$, and that $\mathcal{S}(t)=\mathcal{S}(q^\pm,h)$ also satisfies the estimate,
\begin{equation}\label{E:energy_estimate_final}
\mathcal{S}(t) \leq C\mathcal{S}(0)\ \ \forall\ t\in[0,T].
\end{equation}
Since the space of functions such that $\mathcal{S}(T)$ is bounded imbeds compactly into 
\begin{align*}
C^{1,2}&:= \left\{(q^\pm, h):\ q\in C^1([0,T];C^0(\Omega^\pm))\cap C([0,T];C^2(\Omega^\pm)),\right.\\
&\qquad \left. h\in C^1([0,T];C^0(\Gamma))\cap C([0,T];C^2(\Gamma))\right\},
\end{align*}
we have that the subsequence in fact converges strongly to a solution $(q^\pm, h)\in C^{1,2}$ of \eqref{E:ALE}.
This finishes the existence part of Theorem \ref{T:main}. 

\subsubsection*{Continuity of $\mathcal{S}(t)$} An application of the fundamental theorem of calculus and the bound \eqref{E:energy_estimate_final} gives us continuity of the lower-order norms. Namely,
\begin{equation}\label{E:lowerordercontinuity}
\g_t^l q^\pm \in C([0,T]; H^{5-2l}(\Omega^\pm)),\ \g_t^l h\in C([0,T]; H^{5-2l}(\Gamma)), \text{ for } l=0,..,2. 
\end{equation}
Moreover, since we have the higher regularity estimate \eqref{E:magic}, then $q^\pm\in L^2([0,T];H^{7}(\Omega^\pm))$ and $\g_t^l q_t^\pm\in L^2([0,T];H^{5-2l}(\Omega^\pm))$, so we can use interpolation estimates to obtain that in fact $q^\pm\in C([0,T];H^{6}(\Omega^\pm))$, $q_t^\pm\in C([0,T];H^{4}(\Omega^\pm))$, and $q_{tt}^\pm\in C([0,T];H^2(\Omega^\pm))$. 
The same estimate \eqref{E:magic}, gives us that $h\in L^2([0,T];H^{6.5}(\Gamma))$, and from the definition of $h_t$ in \eqref{E:ALE_ht}, we can actually recover a higher regularity than $L^2([0,T],H^{5}(\Gamma))$. Indeed, notice that $$|h_t|^2_{L^2_tH^{5.5}} \leq \int_0^t|[v\cdot\tilde{n}]^+_-|^2_{5.5}ds,$$ where the right-hand side is easily bounded using the known higher order estimates \eqref{E:magic}. Therefore we have instead that $h_t\in L^2([0,T];H^{5.5}(\Gamma))$, and consequently, via interpolation, we have that $h\in C([0,T];H^6(\Gamma))$, $h_t\in C([0,T],H^4(\Gamma))$, and $h_{tt}\in C([0,T],H^2(\Gamma))$.
It is left to prove then that the norms with the higher-order time derivatives are continuous functions of time, i.e. $q_{ttt}^\pm\in C([0,T]; L^2(\Omega^\pm))$, and $h_{ttt} \in C([0,T];L^2(\Gamma))$.\\

The continuity of $h_{ttt}(t)$ in $L^2(\Gamma)$ follows from the definition of $h_t$ in equation \eqref{E:ALEneumann} time differentiated two times, 
\begin{align*}
h_{ttt} &= [v_{tt}\cdot\tilde{n} + 2v_{t}\cdot\tilde{n}_t +v\cdot\tilde{n}_{tt}]^+_-.
\end{align*}
The continuity of $q_{tt}^\pm\in C([0,T],H^2(\Omega^\pm))$, and the continuity of $\g_t^l h$ in the lower-order norms, gives us then that $v_{tt}^\pm\in C([0,T],H^1(\Omega^\pm))$ and thus the continuity of $h_{ttt}$ follows. Finally, from the triple time differentiated problem \eqref{E:ALEheat}, we have that $q_{ttt}^\pm$ satisfies,
\begin{equation*}
q_{ttt}^\pm = \g_t^2(\Delta_{\Psi^\pm} q^\pm - \Psi_{t}^\pm\cdot v^\pm),
\end{equation*}
where we see that all the terms of the right-hand side are continuous in $L^2(\Omega)$, including the higher-order term $\Psi_{ttt}\cdot v$, since $h_{ttt}\in C([0,T],L^2(\Gamma))$. This concludes the proof of the continuity in time of $\mathcal{S}(t)$. 

\subsubsection*{Uniqueness}
The uniqueness can be derived from the energy estimates in a straightforward way. A brief sketch of the argument is presented in \cite{mHsS2011}, so we omit the analogous proof in this article.

\section{Global well-posedness}\label{C:global}

In Sections~\ref{S:bootstrap} -~\ref{S:improvedenergy}, we shall collect all the necessary ingredients for the proof of Theorem~\ref{T:main_global}, which is presented in Section~\ref{S:finalproof}.
We shall consider small initial data satisfying the hypothesis~\eqref{initialdatawithF}, which implies the bound
\begin{equation}\label{E:small_data}
\E^\pm(q_0^\pm, h_0) \leq \frac{C}{|c_1^\pm|}\frac{\epsilon_0^2}{F(K)},
\end{equation}
where the denominator $|c_1^\pm|$ comes from~\eqref{E:W_dirichlet} and~\eqref{E:globalTaylorsign}.

\subsection{Bootstrap assumptions}\label{S:bootstrap}

As guaranteed by the local well-posedness Theorem \ref{T:main}, 
we assume that the solution $(q^\pm, h)$ to the two-phase Stefan problem \eqref{E:ALE} exists on a time-interval $[0,T]$, for some $T>0.$
For $\epsilon_0<\epsilon \ll 1$ to be specified later, we make the following bootstrap assumptions:

\begin{subequations}\label{E:bootstrap}
\begin{align}
\frac{\newe^\pm(t)}{|c_1^\pm|} + \newe^\ppp(t) + \int_0^t (\frac{\newd^\pm(s)}{|c_1^\pm|} + \newd^\ppp(s))ds &\leq \frac{\epsilon^2}{|c_1^\pm|},\label{E:energy_bootstrap}\\
\sup\limits_{0\leq s\leq T}E^\pm_\beta(s) &\leq \tilde{C}E^\pm_\beta(0),\label{E:lowernorm_bootstrap}\\
\mathcal{X^\pm}(t):= \inf\limits_{x\in\Gamma}\partial_N q^\pm &\geq C|c^\pm_1| e^{-(\lambda^\pm_1+\eta/2)t},\label{E:lower_bound}
\end{align}
\end{subequations}
where we recall the definitions from Section~\ref{S:norms}.

\begin{remark} 
Note that the bounds for the boundary norms $\newe^\ppp,\ \newd^\ppp$ satisfy~\eqref{E:energy_bootstrap} in both the $+$ and the $-$ case, and therefore with
$c_1 = \max\{ |c_1^+|,|c_1^-|\}$ we have the bound:
\begin{equation*}
\newe^\ppp(t) + \int_0^t \newd^\ppp(s)ds \leq \frac{\epsilon^2}{c_1}.
\end{equation*} 
\end{remark}

As in \cite{mHsS2013}, we will prove first that if $\mathcal{T}$ is the maximal time for which the solution exists and satisfies the bootstrap assumptions, we actually improve upon
the smallness assumptions~\eqref{E:energy_bootstrap},~\eqref{E:lowernorm_bootstrap} and the lower bounds~\eqref{E:lower_bound}. 
A standard continuity argument then leads to the proof of global existence.

The following technical lemma will be fundamental for our analysis and it is a direct consequence of the bootstrap assumptions.  Intuitively, the lower order norms of the temperature have strong enough decay to counter the decay of the weight $e^{(-\lambda_1+\eta)t}$ of the boundary norms. This will be used when proving the energy estimates of Lemma \ref{L:higher_energy}, as we will be able to bound products of boundary and interior terms by multiplying and dividing by the weight $e^{(-\lambda_1+\eta)t}$, while still maintaining control of the decay.

\begin{lemma}\label{L:exponential_balance}
If the bootstrap assumptions \eqref{E:bootstrap} hold, we have that
\begin{equation}\label{E:exponential_balance}
\frac{E_\beta(t)^\pm e^{-\beta^\pm t}}{e^{(-\lambda_1 +\eta)t}} \leq C\frac{\epsilon^2_0}{F(K)} e^{-\gamma^\pm t},
\end{equation}
with 
$$
\gamma^\pm = 2\lambda_1^\pm -\lambda_1 >0.
$$ 
\end{lemma}

\begin{proof}
Using the bootstrap assumptions~\eqref{E:lowernorm_bootstrap} and \eqref{E:lower_bound} we arrive at  
\begin{align*}
\frac{E_\beta^\pm(t)e^{-\beta^\pm t}}{e^{(-\lambda_1 +\eta)t}} &\leq \frac{\tilde{C}E_\beta^\pm(0)e^{-\beta^\pm t}}{e^{(-\lambda_1 +\eta)t}} \leq CE_\beta^\pm(0)e^{-\gamma^\pm t}
\end{align*}
Since by definition $E_\beta^\pm(0) = \sum\limits_{b=0}^2 \|\partial_t^b q^\pm(0)\|_{H^{4-2b}}^2$, it follows from the compatibility assumptions (\ref{compatibility1}, \ref{compatibility2})
that
$
E_\beta^\pm(0) \leq C\|q_0^\pm\|_4^2.
$
Therefore
\begin{align*}
\frac{E_\beta^+(t)e^{-\beta^+ t}}{e^{(-\lambda_1 +\eta)t}} &\leq C\|q_0^\pm\|^2_4 e^{-\gamma^\pm t} \leq C \frac{\epsilon_0^2 }{F(K)}e^{-\gamma^\pm t}.
\end{align*}
where the last inequality follows from the smallness of the initial data \eqref{initialdatawithF}.
\end{proof}

\begin{remark}\label{R:otherboundforEbeta}
Using a higher order Hardy inequality as in Lemma 2.1 of \cite{mHsS2013}, we obtain as well the following bound for the initial lower order energy: 
\begin{equation}\label{E:otherboundforEbeta}
E_\beta^\pm(0) \leq C(K^\pm)^4|c_1^\pm|^2,
\end{equation}
which, by inspecting the proof of the Lemma leads to the following alternative bound :
\be\label{E:ALTERNATIVE}
\frac{E_\beta^\pm(t) e^{-\beta^\pm t}}{e^{(-\lambda_1 +\eta)t}} \leq C(K^\pm)^4|c_1^\pm|^2 e^{-\gamma^\pm t}.
\ee
\end{remark}

\subsection{Global estimates for $W^\pm$ and energy equivalence}\label{S:WGLOBAL}
\begin{lemma}[Global estimates for $W^\pm$]\label{l:wbounds}
Let the bootstrap assumptions \eqref{E:bootstrap} hold. Then $W^\pm$ satisfy the following bounds,
\begin{equation}\label{E:w+}
\frac{Ce^{\sigma^\pm t}}{(K^\pm)^2|c_1^\pm|}\leq \min_{x\in\Omega^\pm}W^\pm(t,x) \le \max_{x\in\Omega^\pm}W^\pm(t,x)   \leq \frac{Ce^{(\sigma^\pm+\eta)t}}{|c_1^\pm|},
\end{equation}
where
\be\label{E:SIGMADEFINITION}
\sigma^\pm : = \lambda_1^\pm-\lambda_1+\frac\eta 2>0.
\ee
\end{lemma}

\begin{proof}
We use the bootstrap assumptions \eqref{E:bootstrap} to obtain the following bounds for $\g_N q^\pm$:
\begin{equation*}
C|c_1^\pm|e^{(-\lambda_1^\pm -\eta/2)t} \leq |\g_N q^\pm| \leq \|q^\pm\|_{2.25} \leq C E_\beta^\pm(t)^{1/2} e^{-\beta^\pm t/2} \leq C (K^\pm)^2|c_1^\pm| e^{(-\lambda_1^\pm + \eta/2)t},
\end{equation*}
 where we used the Sobolev embedding $H^{1.25}(\Omega^\pm)\hookrightarrow L^\infty(\Omega^\pm)$ in the second inequality. Now
the proof follows analogously to the proof of Lemma \ref{l:wbounds_local} but using the above bounds instead. 
\end{proof}

With these estimates, we have a new equivalence relation between the natural energy $\mathscr{E}(t)$ and the norm $S(t)$. 

\begin{prop} \label{P:global_energyequiv}
With the bootstrap assumptions \eqref{E:bootstrap} holding, and with $\epsilon >0$ sufficiently small, there exists a constant $C$, and a $\gamma >0$ such that

\begin{align}\label{E:energyequivsummary}
&\sup\limits_{0\leq s\leq t}\frac{e^{\sigma^\pm s}\newe^\pm(s)}{(K^\pm)^2|c_1^\pm|} + \sup\limits_{0\leq s\leq t}\newe^\ppp(s) + \int_0^t \left(\frac{e^{\sigma^\pm s}\newd^\pm(s)}{(K^\pm)^2|c_1^\pm|} + \newd^\ppp(s)\right)ds\nonumber\\
& \leq C\sup\limits_{0\leq s\leq t} \E^\pm(s) + C\int_0^t  \epsilon e^{-\gamma^\pm s}\E^\pm(s)ds + C\int_0^t \D^\pm(s) ds\nonumber\\
&\leq C\mathscr{E}(t).
\end{align}

We will use this relationship in Section \ref{S:finalproof} to prove the final Theorem.

\end{prop}

\begin{proof} The proof of this Proposition follows the same steps as the proofs of Lemma 2.9, Lemma 2.10 and Corollary 2.11 of \cite{mHsS2013}, with the caveat of adding the weights $W^\pm$. We will show the highest order estimate for ilustration purposes, assuming that we have shown suitable estimates up to $\|q_t\|_{4.5}$, using the time differentiated problems. We will omit the upper indices $\pm$ since both regions follow the same argument. Consider equation \eqref{E:ALEheat}, 
\begin{equation*}
\Delta q = q_t + \Psi_t\cdot v + ((\delta^k_l - A^k_j A^l_j) q,_l),_k,
\end{equation*}
therefore, 
\begin{align*}
\|\t^6 \nabla q\|_0 &\leq \|q_t\|_5 + (\|\t^6 \nabla\Psi\|_0 + \|\t^5\Psi_t\|_0 + \|\t^5 v\|_0)\epsilon^{1/2}e^{-\beta t/2} + C\epsilon^{1/2}\E(t)^{1/2}e^{-\beta t/2}\\
\|\t^5 \nabla q\|_0 &\leq \|q_t\|_4 + (\|\t^5 \nabla\Psi\|_0 + \|\t^4\Psi_t\|_0 + \|\t^4 v\|_0)\epsilon^{1/2} e^{-\beta t/2} + C\epsilon^{1/2}\E(t)^{1/2}e^{-\beta t/2}.
\end{align*}
Thus, using the estimates for $\Psi$ and $\Psi_t$ in terms of their boundary value, interpolating, and using that control of the $\psi$-\emph{divergence} of $ \nabla q$ gives us control of the normal derivatives. We obtain
the bound:  
\begin{equation*}
\|q\|_{6.5} \leq \|q_t\|_{4.5} + (|h|_{6} + |h_t|_4 + \|\t^5 v\|_0)\epsilon^{1/2}e^{-\beta t/2} + C\epsilon^{1/2}\E(t)^{1/2}e^{-\beta t/2}.
\end{equation*}
Recall that the boundary norms in the natural energy have the weight $e^{(-\lambda_1 +\eta)t}$, and the interior norm of $\t^5 v$ has the weights $W.$ Therefore, 
\begin{align*}
\|q\|_{6.5} &\leq \|q_t\|_{4.5} + C\left(\E(t)^{1/2}e^{(\lambda_1 - \eta)t/2} + \frac{\E(t)^{1/2}}{(\inf W)^{1/2}}\right)\epsilon^{1/2}e^{-\beta t/2} + C\epsilon^{1/2}\E(t)^{1/2}e^{-\beta t/2}\\
&\leq \|q_t\|_{4.5} + C\epsilon^{1/2}\E(t)^{1/2}e^{-\gamma^\pm t} + \frac{\E(t)^{1/2}}{(\inf W)^{1/2}}\epsilon^{1/2}e^{-\beta t/2} 
\end{align*} 
with $\gamma^\pm = \lambda_1^\pm - \lambda_1/2.$
The result then follows using the estimates for $\|q_t\|_{4.5}$. 
\end{proof}

\begin{remark}
Notice that, just as in the one-phase problem, the exponential growth introduced by bounding the norms of $h$ with the natural energy is counterbalanced by the decay of the lower order norms. 
\end{remark}

\subsection{A priori bounds on $h$}\label{S:APRIORIH}

\begin{lemma}[Suboptimal decay bound for $h_t$]\label{L:suboptimal_decay_ht}
Under the bootstrap assumptions \eqref{E:bootstrap}, the following decay bound holds:
\begin{equation}\label{E:decay_h_t}
|h_t|_{2.5} \leq C\epsilon (e^{-\gamma^+ t/2} + e^{-\gamma^- t/2}) \leq C\epsilon e^{-\lambda_1 t/2}.
\end{equation}
\end{lemma}
\begin{proof}
The proof of this Lemma follows exactly as in Lemma 2.4 of \cite{mHsS2013}, but instead of multiplying and dividing by $\sqrt{\mathcal{X}}$ to obtain a boundary energy term, we multiply by $e^{(-\lambda_1 + \eta)t}$. Since the boundary condition for $h_t$ depends on the jump of the temperature gradients from both regions, we obtain the sum of the two exponential decays. The last inequality follows since we defined $\lambda_1 = \min\{\lambda_1^+,\lambda_1^-\}$. Notice that the weights $W^\pm$ do not show up in this proof, as this estimate only uses the lower order energy $E_\beta(t)$, which has no weights.
\end{proof}

\begin{remark}
A more precise statement can be achieved by following the bootstrap regularity arguments of Lemma 2.6 of~\cite{mHsS2013}, wherein we keep 
track of constant $c_1$ in our estimates. In particular we have
\begin{equation*}
|h_t|_{2.5} \le Cc_1e^{-\lambda_1 t},
\end{equation*}
which will be useful in Lemma~\ref{L:smallness_h}.
\end{remark}

\begin{lemma}[Smallness of the height function]\label{L:smallness_h}
Suppose that the bootstrap assumptions \eqref{E:bootstrap} hold. Then, for $\epsilon>0$ taken sufficiently small,

\begin{equation}
\sup\limits_{0\leq s\leq t} |h(s)|_{4.5} \leq C\sqrt{\epsilon},
\end{equation} 
while for the lower order norms, 
\begin{equation}
\sup\limits_{0\leq s\leq t} |h(s)|_{2.5} \leq Cc_1\ \text{ and }\ \sup\limits_{0\leq s\leq t} |h(s)|_4 \leq C\epsilon^{1/2}c_1^{1/4}.
\end{equation}
\end{lemma}

\begin{proof}
The proof follows the same argument as in the proof of Lemma 2.6 of \cite{mHsS2013}, but with again the same modification as before, when instead of multiplying by $\sqrt{\mathcal{X}}$ we use $e^{(-\lambda_1 +\eta)t}$. We will show only the higher order estimate for brevity. Recall the interpolation estimate,
$$
|h|^2_{4.5} \leq C\int_0^t |h|_6|h_t|_3ds \leq C\int_0^t |h|_6|h_t|^{1/5}_5|h_t|^{4/5}_{2.5}
$$
Then, using an improved bound for $|h_t|_{2.5}$, we have,
\begin{align*}
|h|^2_{4.5} &\leq C\int_0^t |h|_6|h_t|^{4/5}_{2.5}|h_t|^{1/5}_5 \leq Cc_1^{4/5}\int_0^t |h|_6 e^{-4\lambda_1 t/5}|h_t|^{1/5}_5\,ds\\
&\leq Cc_1^{4/5}\int_0^t e^{(-\lambda_1+\eta)s/2}|h|_6 |h_t|_5^{1/5} e^{(-\frac{3\lambda_1}{10} - \frac{1}{2}\eta)s} \,ds
\end{align*}
Let $\bar{\gamma} := \frac{3\lambda_1}{10} + \frac{1}{2}\eta >0.$ Then,
\begin{align*}
|h|^2_{4.5} &\leq C\epsilon c_1^{3/10}\int_0^t e^{-\bar{\gamma} s}|h_t|^{1/5}_5 ds\\
&\leq C\epsilon c_1^{3/10}\left(\int_0^t (e^{-\bar{\gamma}s/2})^{10/9}ds\right)^{9/10}\left(\int_0^t (e^{-\bar{\gamma}s/2}|h_t|^{1/5}_5)^{10}\right)^{1/10}\\
&\leq C\epsilon c_1^{3/10}\left(\int_0^t e^{(-\frac{\lambda_1}{2} + \frac{3}{2}\eta)s}e^{(-\lambda_1+\eta)s}|h_t|^2_5\right)^{1/10}\\
&\leq C\epsilon c_1^{3/10} \left(\int_0^t \newd^\ppp(s)\right)^{1/10}\\
&\leq C\epsilon^{6/5}c_1^{1/5},
\end{align*}
and the claim easily follows. Note that we have used H\"older's inequality, the definition of $\newd^\ppp$ from Section \ref{S:norms} and the bootstrap assumption \eqref{E:energy_bootstrap}.
\end{proof}

\subsection{Lower bounds on $\partial_N q^\pm$}\label{S:lowerbound}
The main objective of this section is to prove a lower bound which improves the bootstrap assumption \eqref{E:lower_bound}. The detailed study of the decay rates of $\partial_N q^\pm$ is of fundamental importance as it also determines the growth and decay-in-time properties of the weights $W^\pm$. We present the following,

\begin{lemma}[Lower bound for $\mathcal{X^\pm}(t)$]\label{lowerboundforX}
Assuming the bootstrap assumptions \eqref{E:bootstrap} with $\epsilon$ small enough, there exists a universal constant $C>0$ such that
\begin{equation}
\mathcal{X^\pm}(t) \geq C|c_1^\pm|e^{-(\lambda_1^\pm + \tilde{\lambda}^\pm(t))t}.\label{E:lowerboundforX}
\end{equation}
Moreover, $\tilde{\lambda}^\pm \geq 0$ satisfies, $\tilde{\lambda}^\pm\leq C\sqrt{\epsilon}$ for some positive constant $C$. In particular, with $\epsilon>0$ sufficiently small so that $C\sqrt{\epsilon} < \eta/4$, we obtain the improvement of the bootstrap bound \eqref{E:lower_bound} given by $$\mathcal{X^\pm}(t) \geq C|c_1^\pm|e^{-(\lambda_1^\pm + \eta/4)t}.$$
\end{lemma}

We will omit the proof as it is detailed in Section 2.6 of \cite{mHsS2015}, where a comparison function is constructed
using the so-called {\em demi-eigenvalues and demi-eigenfunctions} of the maximal Pucci operators detailed in~\cite{Lions}.

\subsection{Improved bounds for the lower order terms of the energy $S(t)$}\label{S:lower_order}

In this Section we prove the improvement of the bootstrap bounds for the terms $E_\beta^\pm,$ responsible
for the decay of the below-top-order energy terms.
\begin{lemma}\label{L:EBETAIMPROVED}
There exists a constant $\tilde{C}$ and $\epsilon >0$ sufficiently small, such that if the bootstrap assumptions \eqref{E:bootstrap} hold with such $\epsilon$ and $\tilde{C}$, then the following improved bound holds:
\begin{equation*}
E_\beta^\pm(t) \leq \frac{\tilde{C}}{2}E_\beta^\pm(0).
\end{equation*}
\end{lemma}

\begin{proof} 
 Notice that the lower order norms $E_\beta^\pm$, do not contain weights $W^\pm$ in their definitions (see Section \ref{S:norms}).
Therefore the proof of the Lemma is analogous to the proof of Lemma 4.1 of \cite{mHsS2013} and we omit the details of the proof for brevity purposes. In~\cite{mHsS2013} the authors used elliptic estimates and a Poincar\'e inequality. 
Note that the Poincar\'e constant is given precisely by the first eigenvalue of the Dirichlet-Laplacian in each region $\Omega^\pm$, which gives us the different decay rates $\lambda_1^\pm$ in each respective region.
The basic mechanism for the decay of the $H^4$-norm of the temperature $q$ is exactly the same as in the standard heat equation. We can not however propagate this nearly-optimal decay rate
to the top-order norms of $q$ as that would require a stronger control on the top-order derivatives of $h,$ than the one dictated by the natural energy $\mathscr E.$
\end{proof}

\subsection{Improved bounds for the energy $\mathscr{E}(t)$}\label{S:improvedenergy}

In this Section we prove the higher-order energy estimates to be used in the proof of Theorem~\ref{T:main_global} in Section \ref{S:finalproof}. 
Most of the energy terms are dealt with using the same techniques as in \cite{mHsS2013}, but all the new terms that arise from the interaction of the two phases via the weights $W^\pm$ are presented in detail. 
This is the price we must pay for eliminating the \emph{critical term} of \cite{mHsS2013}, but we will see it is a low price to pay.
As a starting point we first prove sharp upper bounds on the spacetime derivatives of the weights $W^\pm$ as they will be used crucially in dealing with the above mentioned ``critical" terms.

\begin{lemma}[Global estimates for $ \nabla W^\pm$ and $W_t^\pm$] \label{L:WEIGHTDERIVATIVES}
Under the bootstrap assumptions \eqref{E:bootstrap}, the derivatives of the weight functions $W^\pm$ satisfy the following bounds:
\begin{align*}
\|W_t^\pm\|_{L^\infty(\Omega^\pm)} \leq P(K) \frac{e^{\sigma^\pm t}}{|c_1^\pm|}e^{2\eta t}\ \text{ and }\ \ \ \| \nabla W^\pm\|_{L^\infty(\Omega^\pm)} \leq P(K)\frac{e^{\sigma^\pm t}}{|c_1^\pm|} e^{3\eta t}
\end{align*}
where we recall from~\eqref{E:SIGMADEFINITION}
$
\sigma^\pm = \l_1^\pm-\l_1 + \frac\eta 2.
$
\end{lemma}

\begin{proof}

Recall that $W_t^\pm$ solves the elliptic boundary value problem \eqref{E:Wt}. Therefore 
\begin{align*}
|W_t^\pm|_{L^\infty(\Gamma)} &\leq Ce^{(-\lambda_1 +\eta)t}\left(\frac{1}{\mathcal{X}^\pm(t)} + \frac{|\g_Nq_t^\pm|_{L^\infty}}{\mathcal{X}^\pm(t)^2}\right)\\
&\leq Ce^{(-\lambda_1 +\eta)t}\left(\frac{e^{(\lambda_1^\pm + \eta/2)t}}{|c_1^\pm|} + \frac{(K^\pm)^2 |c_1^\pm|e^{-\beta^\pm t/2}}{|c_1^\pm|^2}\right)\\
&\leq \frac{e^{(-\lambda_1 + \lambda_1^\pm +3\eta/2)t}}{|c_1^\pm|}\left(1 + (K^\pm)^2e^{\eta t}\right)\\
&\leq C(K^\pm)^2\frac{e^{(-\lambda_1 + \lambda_1^\pm +5\eta/2)t}}{|c_1^\pm|}\\
&\leq P(K) \frac{e^{\sigma^\pm t}}{|c_1^\pm|}e^{2\eta t}.
\end{align*} 
And so the result follows from the maximum principle applied to $W_t^\pm$. To obtain the bound on $\| \nabla W^\pm\|_{L^\infty}$, we first use the Sobolev inequality and then the elliptic 
theory to infer that 
\begin{align*}
\| \nabla W^-\|_{L^\infty} &\leq C\| \nabla W^-\|_{1.25} \leq C e^{(-\lambda_1 + \eta)t}\left|\frac{1}{\partial_N q^-}\right|_{1.75},\\
\| \nabla W^+\|_{L^\infty} &\leq C\| \nabla W^+\|_{1.25} \leq C e^{(-\lambda_1 + \eta)t}\left|\frac{1}{\partial_N q^+}\right|_{H^{1.75}(\Gamma)} + C\frac{e^{(-\lambda_1 +\lambda_1^\pm + \eta)t}}{|c_1^+|}.
\end{align*}

To estimate the right-hand sides, let us  first estimate the $L^2(\Gamma)$-norm of two tangential derivatives applied to $\frac{1}{\partial_N q^\pm}:$
\begin{align*} 
&\left|\t^2(\frac{1}{\partial_N q^\pm})\right|_0 \leq C\frac{|\t \g_N q^\pm|^2_{L^4}}{\mathcal{X}^\pm(t)^3} + C\frac{|\t^2 \g_N q^\pm|_0}{\mathcal{X}^\pm(t)^2}\\
&\leq \frac{\|q^\pm\|_3^2}{|c_1^\pm|^3}e^{3(\lambda_1^\pm +\eta/2)t} + \frac{\|q^\pm\|_{3.5}}{|c_1^\pm|^2}e^{(2\lambda_1^\pm +\eta)t}\\
&\leq \frac{(K^\pm)^4}{|c_1^\pm|} e^{(-\beta^\pm + 3\lambda_1^\pm + 3\eta/2)t} + \frac{(K^\pm)^2}{|c_1^\pm|} e^{(-\beta^\pm/2 + 2\lambda_1^\pm + \eta)t}\\
&\leq \frac{e^{(\lambda_1^\pm + 5\eta/2)t}}{|c_1^\pm|}\left((K^\pm)^4 + (K^\pm)^2e^{-\eta t}\right)\\
&\leq P(K)\frac{e^{(\lambda_1^\pm + 5\eta/2)t}}{|c_1^\pm|},
\end{align*}
where we used the continuous embedding $H^{1/2}\hookrightarrow L^4,$ the trace estimates, and the bound~\eqref{E:ALTERNATIVE}. Recall as well the definition~\eqref{E:BETAPMDEFINITION}
of $\beta^\pm = 2\l_1^\pm-\eta.$
The same bound for $|\t\left(\frac 1{\g_Nq^\pm}\right)|_{0}$ follows analogously. As a result, we obtain that 
\begin{equation*}
\| \nabla W^\pm\|_{L^\infty} \leq P(K)\frac{e^{(-\lambda_1 + \lambda_1^\pm + \eta/2)t}}{|c_1^\pm|} e^{3\eta t} = P(K)\frac{e^{\sigma^\pm t}}{|c_1^\pm|}e^{3\eta t},
\end{equation*}
for some universal polynomial $P$. This concludes the proof of the Lemma.
\end{proof}

\begin{remark}
The proof of the lemma depends on the boundary estimates for $W_t^\pm$ and $ \nabla W^\pm$ on $\Gamma$. 
This inevitably leads us to study sharp upper bounds for the reciprocals $\frac 1{\g_Nq^\pm},$ which in turn demands
a very good understanding of the lower bounds on the decay rate of the Neumann derivative $\g_Nq^\pm.$ These
ingredients are provided by Lemma~\ref{lowerboundforX}.
\end{remark}

We can now prove the improvement on the higher-order bootstrap assumption \eqref{E:energy_bootstrap}. Notice that, because of the way in which the norm and the natural energy are related in equation \eqref{E:energyequivsummary}, we need to build the estimates with the appropriate coefficients $|c_1^\pm|,\ K^\pm$, and exponential growth $e^{\sigma^\pm t}$.

\begin{lemma}[Higher-order energy estimates]\label{L:higher_energy}
Suppose that the bootstrap assumptions \eqref{E:bootstrap}  hold with $\epsilon >0$ and $\eta >0$ sufficiently small. Then the following bound holds:
\begin{align}\label{E:energy_estimate_1}
\mathscr{E}(t) &\leq\ C\mathcal{E}(0) + \frac{P(K)\epsilon_0^2}{|c_1^\pm|F(K)}
+ O(\epsilon+\delta)\left[\sup\limits_{0\leq s\leq t}\frac{\newe^+(s)e^{\sigma^+ s}}{(K^+)^2|c_1^+|} + \sup\limits_{0\leq s\leq t}\frac{\newe^-(s)e^{\sigma^- s}}{(K^-)^2|c_1^-|} + \sup\limits_{0\leq s\leq t} \newe^\ppp(s)\right.\nonumber\\
&\left.   +\int_0^t \left(\frac{\newd^+(s)e^{\sigma^+ s}}{(K^+)^2|c_1^+|} + \frac{\newd^-(s)e^{\sigma^-s}}{(K^-)^2|c_1^-|} + \newd^\ppp(s)\right)ds\right], \ \text{ for all } t\in [0,\mathcal{T}]
\end{align}
where $P(K)$ is some universal polynomial, $0\le\delta\ll1$ is sufficiently small, and 
$\sigma^\pm = \lambda_1^\pm - \lambda_1 + \eta/2 >0$ has already been introduced in~\eqref{E:SIGMADEFINITION}.
\end{lemma}

\begin{proof}
Recall the energy identity~\eqref{E:energy_identity1} from Lemma \ref{ab_energy_lemma}, 
and notice that we need only to prove bounds for the error terms $\mathcal{R}$. We will exemplify this by showing the estimates for $\mathcal{R}^\pm_\aabb,$ defined in the statement of Lemma~\ref{ab_energy_lemma}.

\subsection*{Bounds for $\mathcal{R}^\pm_\aabb$}
Let us re write $\mathcal{R}^\pm_\aabb$ as a sum: 
$$\int_0^t \mathcal{R}^\pm_\aabb =\sum\limits_{a,b} \mathcal{R}^\aabb_1{}^\pm + \mathcal{R}^\aabb_2{}^\pm + \mathcal{R}^\aabb_3{}^\pm + \mathcal{R}^\aabb_4{}^\pm + \mathcal{R}^\aabb_5{}^\pm + \mathcal{R}^\aabb_6{}^\pm,$$
where $\mathcal{R}^\aabb_i{}^\pm$ are the integrated terms written in the definition of $\mathcal{R}^\pm_\aabb,$  see Lemma~\ref{ab_energy_lemma}.
The estimates for $\mathcal{R}^\aabb_1{}^\pm, \mathcal{R}^\aabb_2{}^\pm, \mathcal{R}^\aabb_3{}^\pm, \mathcal{R}^\aabb_4{}^\pm$ follow the same strategy as in Section 3.2 of \cite{mHsS2013}. The only addition is the presence of the weights $W^\pm,$
which we  bound in $L^\infty$ yielding an additional exponentially growing term with a rate $\sigma^\pm +\eta$. 
It is therefore left to show that in every such error term, there exists a below-top-order energy term which decays sufficiently fast to counteract the potential exponential growth stemming from $W^\pm$. 
We illustrate this by estimating the term $\mathcal{R}^\aabb_1{}^\pm$:
\begin{align*}
 \Big|\int_{\Omega^\pm} \t^5 A \t \nabla q\t^6 v W \Big| & \leq \|\t^5 A\|_{L^4}\|\t \nabla q\|_{L^4} \|\t^6 v\|_0 \|W\|_{L^\infty}  \\
& \leq C|h|_6\|q\|_{2.5}\|\t^6 v\|_0 \frac{C}{|c_1^\pm|}e^{(\sigma^\pm +\eta)t}\\
&\lesssim \frac{1}{|c_1^\pm|} e^{(-\lambda_1 + \eta)t/2}|h|_6 \frac{E_\beta(t)^{1/2} e^{-\beta t/2}}{e^{(-\lambda_1 +\eta)t/2}} \newd^{1/2}(t) e^{(\sigma^\pm + \eta)t} \\
& \lesssim   K^\pm |c_1^\pm|^{-1/2} 
\frac{\sqrt{\epsilon_0}}{\sqrt{F(K)}}\newe^\ppp(t)^{1/2}\newd(t)^{1/2} e^{(\sigma^\pm + \eta -\gamma^\pm/2)t} \\
&\leq \delta \newe^\ppp(t)  + C_\delta(K^\pm)^2|c_1^\pm|^{-1}\frac{\epsilon_0}{F(K)}\newd^\pm(t)e^{\sigma^\pm t} e^{(\sigma^\pm-\gamma^\pm+2\eta) t} \\
&\leq \delta \newe^\ppp(t) + C_\delta(K^\pm)^4\frac{\epsilon_0}{F(K^\pm)^{1/2}} \frac{ \newd^\pm(t)e^{\sigma^\pm t}}{(K^\pm)^2|c_1^\pm|},
\end{align*}  
where $\sigma^\pm$ is given by~\eqref{E:SIGMADEFINITION}, $\gamma^\pm = 2\l_1^\pm - \l_1,$ and therefore
$\sigma^\pm-\gamma^\pm+\eta = -\frac 12 \l_1 + \frac32 \eta<0$ for $\eta$ sufficiently small. 
Note that we used~\eqref{E:w+} in the second line, the definition of $E_\beta$ in the third line, and the estimate~\eqref{E:ALTERNATIVE}  in the fourth line to bound the ratio 
$\frac{E_\beta(t)^{1/2} e^{-\beta t/2}}{e^{(-\lambda_1 +\eta)t/2}}$ by $C K^\pm |c_1^\pm|^{1/2} \frac{\sqrt{\epsilon_0}}{\sqrt{F(K)}}e^{-\gamma^\pm t/2}.$
In the fifth line we used Young's inequality and in the last line the negativity of $\sigma^\pm-\gamma^\pm+\eta$. Considering $\epsilon_0$ small enough so that $C_\delta(K^\pm)^4\frac{\epsilon_0}{F(K^\pm)^{1/2}}<\epsilon$, gives us the desired inequality.

We can apply an entirely analogous reasoning to bound the terms $\mathcal{R}^\aabb_j{}^\pm,$ $j=2,3,4.$
Integrating-in-time we therefore obtain that  
\begin{align*}
\mathcal{R}^\aabb_1{}^\pm + \mathcal{R}^\aabb_2{}^\pm + \mathcal{R}^\aabb_3{}^\pm + \mathcal{R}^\aabb_4{}^\pm &\leq O(\epsilon+\delta) \left[\sup\limits_{0\leq s\leq t}\left(\frac{\newe^+(s)e^{\sigma^+s}}{(K^+)^2|c_1^+|} + \frac{\newe^-(s)e^{\sigma^-s}}{(K^-)^2|c_1^-|} + \newe^\ppp(s)\right)\right.\\
& + \left.\int_0^t \left(\frac{\newd^+(s)e^{\sigma^+ s}}{(K^+)^2|c_1^+|}+ \frac{\newd^-(s)e^{\sigma^-s}}{(K^-)^2|c_1^-|} + \newd^\ppp(s)\right)ds\right].
\end{align*} 
The estimates for the terms  $\tilde{\mathcal{R}}^\pm,\ \mathcal{R}_\Gamma,\ \mathcal{R}^+_{\ppp}$, and $\mathring{\mathcal{R}}^\pm$, follow the same methodology and we omit the details.  

The remaining integrals $\mathcal{R}^\aabb_5{}^\pm$ and $\mathcal{R}^\aabb_6{}^\pm$ in the definition of $\mathcal{R}^\pm_\aabb$ (see Lemma~\ref{ab_energy_lemma}) 
are {\em new} error terms with respect to~\cite{mHsS2013,mHsS2015} and they involve derivatives of the weights $W^\pm:$ 
\begin{align*}
\mathcal{R}^\aabb_5{}^\pm &:= \int_0^t \int_{\Omega^\pm} (\t^a\partial_t^b q^\pm + \t^a\partial_t^b\Psi^\pm\cdot v^\pm) A\t^a\partial_t^b v^\pm  \nabla W^\pm\\
\mathcal{R}^\aabb_6{}^\pm &:= \frac{1}{2}\int_0^t \int_{\Omega^\pm} (\t^a\partial_t^b q^\pm + \t^a\partial_t^b\Psi^\pm\cdot v^\pm)^2 W_t^\pm.
\end{align*}
To bound $\mathcal{R}^\aabb_5{}^\pm$ and $\mathcal{R}^\aabb_6{}^\pm$ we need upper bounds for $ \nabla W^\pm$ and $W_t^\pm$ in $L^\infty$ provided by Lemma~\ref{L:WEIGHTDERIVATIVES}. We will show that even for these terms, the additional exponential growth is also counterbalanced by the decay of the lower order norms.
In the following we omit the upper index $\pm$ for simplicity, and consider only the hardest case $a=6,\ b=0,$ as the argument remains the same for the other cases.
\begin{align}
\mathcal{R}^{a=6,b=0, \, \pm}_5&\leq \int_0^t \|\t^6 q + \t^6\Psi\cdot v\|_0 \|\t^6 v\|_0 \| \nabla W^\pm\|_{L^\infty}\nonumber\\
&\leq \frac{P(K)}{|c_1^\pm|}\int_0^t (\|q\|_6 + \|\t^6\Psi\|_0\|v\|_{L^\infty} )\|\t^6 v\|_0 e^{\sigma^\pm s}e^{3\eta s}ds\nonumber\\
&\leq \frac{P(K)}{|c_1^\pm|}\int_0^t \|q\|_6\|\t^6 v\|_0 e^{\sigma^\pm s}e^{3\eta s}ds + \frac{P(K)}{|c_1^\pm|}\int_0^t\|\t^6\Psi\|_0\|v\|_{L^\infty}\|\t^6 v\|_0 e^{\sigma^\pm s}e^{3\eta s}ds, \label{E:AUX}
\end{align}
where we used Lemma~\ref{L:WEIGHTDERIVATIVES} in the second line.
The first term of the right-most side is estimated as follows:
\begin{align*}
&\frac{P(K)}{|c_1^\pm|}\int_0^t \|q\|_6\|\t^6 v\|_0 e^{\sigma^\pm s}e^{3\eta s}ds \leq \frac{P(K)}{|c_1^\pm|} \int_0^t\|q\|_{4}^{1/5}\|q\|_{6.5}^{4/5} \newd^\pm(s)^{1/2} e^{\sigma^\pm s}e^{3\eta s}ds\\
&\leq \frac{P(K)}{|c_1^\pm|}\int_0^t E_\beta^\pm(s)^{1/10} e^{-\beta^\pm s/10} \newd^\pm(s)^{9/10}e^{\sigma^\pm s}e^{3\eta s}ds\\
&\leq \frac{C}{|c_1^\pm|}\int_0^t \left(C_\delta P(K) E_\beta^\pm(s)e^{(-\beta^\pm + 30\eta) s} + \frac{\delta}{(K^\pm)^2}\newd^\pm(s)\right)e^{\sigma^\pm s} ds\\
&\leq \frac{C_\delta P(K) E_\beta^\pm(0)}{|c_1^\pm|}\int_0^t e^{-\bar{\gamma}^\pm s} ds + \delta\int_0^t\frac{\newd^\pm(s)e^{\sigma^\pm s}}{(K^\pm)^2|c_1^\pm|} ds\\
&\leq \frac{P(K)\epsilon_0^2}{|c_1^\pm|F(K)} + \delta\int_0^t\frac{\newd^\pm(s)e^{\sigma^\pm s}}{(K^\pm)^2|c_1^\pm|} ds,
\end{align*}
where we note that 
$\bar{\gamma}^\pm = \lambda_1^\pm+\lambda_1 - 63\eta/2 >0$ for $\eta$ small enough. We used an interpolation estimate and the definition of the norm $\newd^\pm$  in the first line, and Young's inequality in the third.
Similarly, the second integral on the right-most side of~\eqref{E:AUX} satisfies:
\begin{align*}
&\frac{P(K)}{|c_1^\pm|}\int_0^t\|\t^6\Psi\|_0\|v\|_{L^\infty}\|\t^6 v\|_0 e^{\sigma^\pm s}e^{3\eta s}ds \leq \frac{P(K)}{|c_1^\pm|}\int_0^t |h|_{5.5}E_\beta^\pm(s)^{1/2}e^{(-\lambda_1^\pm +\frac{\eta}{2})s}\newd(s)^{1/2} e^{\sigma^\pm s}e^{3\eta s}ds\\
&\leq P(K)\int_0^t e^{(-\lambda_1 + \eta)s/2}|h|_{5.5}\newd(s)^{1/2} e^{\sigma^\pm s}e^{(-\lambda_1^\pm + \lambda_1/2 + 3\eta)s}ds\\
&\leq P(K)\int_0^t \newe^\ppp(s)^{1/2}\newd(s)^{1/2} e^{(-\lambda_1/2 + 3\eta)s}ds\\
&\leq \delta \sup\limits_{0\leq s\leq t} \newe^\ppp(s) + C_\delta\frac{P(K)\epsilon_0}{F(K)^{1/2}}\int_0^t \frac{\newd^\pm(s)}{(K^\pm)^2|c_1^\pm|} e^{(-\lambda_1/2 + 3\eta)s}ds.
\end{align*}
Taking $\epsilon_0$ so small that $\frac{P(K)\epsilon_0}{F(K)^{1/2}}\leq \epsilon$, we obtain the desired inequality.
For the error term $\mathcal{R}^\aabb_6{}^\pm$ in the case $a=6,b=0$ we follow the same idea, but instead use the bound on $W_t^\pm$ from Lemma~\ref{L:WEIGHTDERIVATIVES}:
\begin{align}
\mathcal{R}^{a=6,b=0,\, \pm}_6&\leq \frac{1}{2}\int_0^t  \|\t^6 q + \t^6\Psi\cdot v\|_0^2  \|W_t\|_{L^\infty}ds\nonumber\\
&\leq \frac{P(K)}{|c_1^\pm|}\int_0^t (\|q\|_6^2 + \|\t^6\Psi\|_0^2\|v\|_{L^\infty}^2)e^{\sigma^\pm t}e^{2\eta t} ds\nonumber\\
&\leq \frac{P(K)}{|c_1^\pm|}\int_0^t \|q\|_{4}^{2/5}\|q\|_{6.5}^{8/5}e^{\sigma^\pm t}e^{2\eta t}ds + \frac{P(K)}{|c_1^\pm|}\int_0^t |h|^2_{5.5}\|q^+\|_{2.75}^2 e^{\sigma^\pm t}e^{2\eta t}ds\nonumber\\
&\leq  \frac{P(K)}{|c_1^\pm|}\int_0^t E_\beta^\pm(0)^{1/5}e^{(\frac{\sigma^\pm}{5} - \beta^\pm/5 + 2\eta)s}\newd(s)^{4/5} e^{4\sigma^\pm s/5}ds + 
\frac{P(K)\epsilon_0}{F(K)^{1/2}}\int_0^t \newe^\ppp(s)e^{(-\lambda_1^\pm +5\eta/2) s}ds\nonumber\\
&\leq C_\delta\frac{P(K)}{|c_1^\pm|}\int_0^t E_\beta(0) e^{-\tilde{\gamma} s} ds + \delta \int_0^t \frac{\newd^\pm(s)e^{\sigma^\pm s}}{(K^\pm)^2|c_1^\pm|} ds  + \epsilon\sup\limits_{0\leq s\leq t} \newe^\ppp(s)\nonumber\\
&\leq C_\delta\frac{P(K)\epsilon_0^2}{|c_1^\pm|F(K)} + \delta \int_0^t \frac{\newd^\pm(s)}{(K^\pm)^2|c_1^\pm|}e^{\sigma^\pm s}ds + \epsilon\sup\limits_{0\leq s\leq t} \newe^\ppp(s), \label{R2_final},
\end{align}
where $\tilde{\gamma}^\pm = \lambda_1^\pm + \lambda_1 - 19\eta/2>0$  for a sufficiently small $\eta>0$. Note that we used norm-interpolation in the third line and the exponential decay in the fifth line
to infer that $\frac{P(K)\epsilon_0}{F(K)^{1/2}}\int_0^t \newe^\ppp(s)e^{(-\lambda_1^\pm +5\eta/2) s}ds \le \epsilon\sup\limits_{0\leq s\leq t} \newe^\ppp(s).$ 
This completes the proof of the higher-order energy estimate. 
\end{proof}

\subsection{Proof of Theorem~\ref{T:main_global}}\label{S:finalproof}
 
To finish the proof we consider $\mathcal{T}>0$, to be the maximal time at which the solution exists and satisfies the bootstrap assumptions \eqref{E:bootstrap}. 
We will assume by means of a contradiction argument that $\mathcal T$ is finite, and will obtain an improved estimate for the higher energy bootstrap up to $\mathcal{T}$, which, via the local well-posedness Theorem, will gives us a contradiction to the maximality of $\mathcal{T}$.

From the global energy equivalence relation \eqref{E:energyequivsummary}, we have that, 
\begin{align*}
&\sup\limits_{0\leq s\leq t}\frac{\newe^\pm(s)}{(K^\pm)^2|c_1^\pm|}e^{\sigma^\pm s} + \sup\limits_{0\leq s\leq t}\newe^\ppp(s) 
+ \int_0^t \left(\frac{\newd^\pm(s)}{(K^\pm)^2|c_1^\pm|} e^{\sigma^\pm s} + \newd^\ppp(s)\right)ds \leq C\mathscr{E}(t).
\end{align*}
Therefore, using Lemma~\ref{L:higher_energy}, we have that 
\begin{align*}
&\sup\limits_{0\leq s\leq t}\frac{\newe^\pm(s)}{(K^\pm)^2|c_1^\pm|}e^{\sigma^\pm s} + \sup\limits_{0\leq s\leq t}\newe^\ppp(s) + \int_0^t \left(\frac{\newd^\pm(s)}{(K^\pm)^2|c_1^\pm|} e^{\sigma^\pm s} + \newd^\ppp(s)\right)ds \leq C\mathcal{E}(0) + \frac{P(K)\epsilon_0^2}{F(K)|c_1^\pm|}\\ 
& + O(\epsilon+\delta)\left[\sup\limits_{0\leq s\leq t}\frac{\newe^\pm(s)e^{\sigma^\pm s}}{(K^\pm)^2|c_1^\pm|} + \sup\limits_{0\leq s\leq t} \newe^\ppp(s) +\int_0^t (\frac{\newd^\pm(s)e^{\sigma^\pm s}}{(K^\pm)^2|c_1^\pm|} + \newd^\ppp(s))ds\right],
\end{align*} 
Taking $\delta$ and $\epsilon$ small enough, we can absorb the term in the rectangular brackets to obtain
\begin{align*}
&\sup\limits_{0\leq s\leq t}\frac{\newe^\pm(s)}{(K^\pm)^2|c_1^\pm|}e^{\sigma^\pm s} + \sup\limits_{0\leq s\leq t}\newe^\ppp(s) + \int_0^t \left(\frac{\newd^\pm(s)}{(K^\pm)^2|c_1^\pm|} e^{\sigma^\pm s} + \newd^\ppp(s)\right)ds \leq C\mathcal{E}(0) +\frac{P(K)\epsilon_0^2}{F(K)|c_1^\pm|}.
\end{align*}

Using the smallness condition on the initial data \eqref{E:small_data},
\begin{align*}
&\sup\limits_{0\leq s\leq t}\frac{\newe^\pm(s)}{(K^\pm)^2|c_1^\pm|}e^{\sigma^\pm s} + \sup\limits_{0\leq s\leq t}\newe^\ppp(s) + \int_0^t \left(\frac{\newd^\pm(s)}{(K^\pm)^2|c_1^\pm|} e^{\sigma^\pm s} + \newd^\ppp(s)\right)ds \leq \frac{\epsilon_0^2}{|c_1^\pm|F(K)}(1 +P(K)).
\end{align*}
Taking $\epsilon_0$ small enough so that 
\begin{equation*}
\frac{\epsilon_0^2(K^\pm)^2(1+P(K))}{F(K)} \leq \frac{\epsilon^2}{2},
\end{equation*} 
we obtain an improvement on the bootstrap assumption \eqref{E:energy_bootstrap},
\begin{equation}
\sup\limits_{0\leq s\leq t}\frac{\newe^\pm(s)}{|c_1^\pm|}e^{\sigma^\pm s} + \sup\limits_{0\leq s\leq t}\newe^\ppp(s) + \int_0^t \left(\frac{\newd^\pm(s)}{|c_1^\pm|} e^{\sigma^\pm s} + \newd^\ppp(s)\right)ds \leq \frac{\epsilon^2}{2|c_1^\pm|},  \label{E:higherdecay}
\end{equation}
for all $t\in [0,\mathcal{T}]$. Therefore, by the continuity of the energy, we can extend the solution by the local well-posedness Theorem to $[0,\mathcal{T}+T_0)$, for some small $T_0 >0$, 
so that the bootstrap assumptions hold up until $\mathcal{T}+T_0$. This contradicts the maximality of $\mathcal{T}$, and therefore $\mathcal{T} = +\infty$, and the proof is complete.  

\begin{remark}
Notice that not only we obtained an improvement on the energy bootstrap, but from \eqref{E:higherdecay} we also obtained that the higher order norms $\newe^\pm$ and $\newd^\pm$ 
decay at a rate $e^{-\sigma^\pm t}$, with $\sigma^\pm = \lambda_1^\pm - \lambda_1 + \eta/2$. 
Recall that $\lambda_1$ is the minimum of the two eigenvalues, which means intuitively that the temperature in the region with smaller surface area decays faster to an equilibrium.   
Note however that the norm $\newe^\ppp$ measuring the size of the free boundary deviation from the reference domain {\em does not} exhibit any decay, which is
consistent with the idea that the asymptotic equilibrium shape is selected from a continuum of possible nearby steady states.
\end{remark}

\renewcommand{\thetheorem}{A.\arabic{theorem}}
\renewcommand{\theequation}{A.\arabic{equation}}
\appendix
\section{Basic inequalities}
\label{S:apendix_local}

\begin{lemma}[A priori bounds for $A$ in terms of $h$]\label{A_bounds} 
Let $h\in H^3(\Gamma)$ such that $|h|_3$ is small, and consider $\Psi^\pm$ the solution to the elliptic problem \eqref{E:definition_Psi}, and $A^\pm := ( \nabla\Psi^\pm)^{-1}$, then: 
\begin{equation}
\|A^\pm\|_{L^\infty} \leq 1 + C|h|_{2.25} + O(|h|_{2.25}^2).
\end{equation}
\end{lemma}
\begin{proof}
We will omit the super-index $\pm$ throughout the proof for clarity.
First, we bound the difference $ \nabla\Psi - I$, using Sobolev embedding and elliptic estimates: 
\begin{equation*}
\| \nabla\Psi - \text{Id}\|_{L^\infty} \leq \| \nabla\Psi - I\|_{1.75} \leq C|h|_{2.25}
\end{equation*}
Moreover, for $0\leq s\leq 3$, $$\| \nabla^2\Psi\|_s \leq C|h|_{s+1.5}$$
Using this we obtain,
\begin{equation*}
\|A-\text{Id}\|_{L^\infty} \leq \|A(\text{Id}- \nabla\Psi)\|_{L^\infty} \leq C\|A\|_{L^\infty}|h|_{2.25}.
\end{equation*} 
Therefore, 
\begin{align*}
\|A\|_{L^\infty} &\leq \|\text{Id}\|_{L^\infty} + \|A-\text{Id}\|_{L^\infty} \leq 1 +C\|A\|_{L^\infty}|h|_{2.25} \\
\|A\|_{L^\infty} &\leq \frac{1}{1-C|h|_{2.25}} \leq 1 + C|h|_{2.25} + O(|h|_{2.25}^2).
\end{align*}
Moreover we have that for $0\leq s\leq 3$, 
$$\| \nabla A\|_s \leq C|h|_{s+1.5}.$$

\end{proof}

\begin{lemma}\label{L:delta_A_bound}
Considering $\bar{h}_\alpha$ such that $S(\bar{q}_\alpha,\bar{h}_\alpha)\leq M$, $\alpha = 1,2$, we have that there exists a time $T_\k$ small enough so that the following bound for $\delta A = A_1 - A_2$ holds, for $s>1.5$:

\begin{equation}
\snorm{\delta A}_{L^\infty H^s}\leq \epsilon |\delta \bar{h}_t|_{L^2 H^{s+0.5}}
\end{equation}
for any $\epsilon > 0$.

\begin{proof}
\begin{align}
& A_1-A_2 = \int_0^t A_{1t} - A_{2t} = -\int_0^t A_1A_1 \nabla\delta \Psi^1_t + (A_1(A_1-A_2) + (A_1-A_2)A_2) \nabla\Psi^2_t,\nonumber\\
&\snorm{\delta A}_s \leq \int_0^t \left(\snorm{A_1}^2_s\snorm{\delta\Psi_t}_{s+1} + (\snorm{A_1}_s + \snorm{A_2}_s)\snorm{\delta A}_s\snorm{\Psi^2_t}_{s+1}\right),\nonumber\\
&\snorm{\delta A}_{L^\infty H^s} \leq \frac{\sqrt{T_\kappa}\snorm{A_1}^2_{L^\infty H^s}}{(1-\sqrt{T_\kappa}C(M))}|\delta \bar{h}_t|_{L^2 H^{s+0.5}}.
\end{align} 
Taking $T_\kappa$ small enough yields the result. 
\end{proof}

\end{lemma}

\begin{lemma}[Bounds for $J$] \label{gJbounds}
Under the hypothesis of Proposition \ref{P:local_energyequivalence}, there exist a $\delta>0$ such that if $t<\delta$, then the determinants $J^\pm=det( \nabla\Psi^\pm)$ satisfy the bound,
\begin{equation}
\frac{1}{2}\leq J \leq \frac{3}{2}\ \text{ on } \Gamma.
\end{equation}

\end{lemma}

\begin{proof}
For any of the regions $\Omega^\pm$, 
\begin{align*}
J &= det( \nabla\Psi) = \Psi^1,_1\Psi^2,_2 - \Psi^1,_2\Psi^2,_1= 1 + [(\Psi^1,_1 -1)\Psi^2,_2 + (\Psi^2,_2 - 1) - \Psi^1,_2\Psi^2,_1].
\end{align*}
Now, 
\begin{align*}
|(\Psi^1,_1 -1)\Psi^2,_2 + (\Psi^2,_2 - 1) - \Psi^1,_2\Psi^2,_1|_\infty &\leq |\Psi^1,_1 -1|_\infty |\Psi^2,_2|_\infty + |\Psi^2,_2 - 1|_\infty + |\Psi^1,_2|_\infty |\Psi^2,_1|_\infty\\
&\leq \| \nabla\Psi -I\|_{1.25}(\| \nabla\Psi\|_{1.25} +1) + \| \nabla\Psi -I\|^2_{1.25}\\
&\leq C|h|_{1.75} + C|h|^2_{1.75}
\end{align*}

But recall that, $h = h_0 + \int_0^t h_t(s)ds$, therefore, $|h|_s \leq |h_0|_s + \int_0^t |h_t(s)|_s ds\leq |h_0|_s + \sqrt{t}|h_t|_{L^2H^s}$, with $|h_0|_s=O(\sigma)$ which we can make as small as we want. Therefore, taking $t$ small enough, and since we are considering $|h_t|_{L^2H^s} \leq M$, we obtain that $|h|_s$ can be made small for short time, and so,  

$$1-C|h|_{1.75}-C|h|^2_{1.75} \leq J \leq 1 + C|h|_{1.75} + C|h|^2_{1.75}$$
$$\frac{1}{2}\leq J \leq \frac{3}{2}\ \text{ on } \Gamma.$$
\end{proof}

\begin{lemma}{Parabolic estimates for $q^m$.}\label{L:first_parabolic_estimates} Given $q^m$ as in \eqref{E:galerkin_sol}, there exists a small enough time $T_\k$ independent of $m$, such that we have the following estimates,
\begin{equation*}
\|q^m(t)\|^2_{L^2(\Omega)} + \|q^m\|^2_{L^2_tH^1} + \k^{-2}|q^m|^2_{L^2_tL^2} \leq C(M,q_0), \text{ for all } t\in [0,T_\k],
\end{equation*}
where $C(M,q_0)$ is a constant independent of $m$. 
\end{lemma}
\begin{proof}
Indeed, considering $\phi = q^m$ on the weak formulation \eqref{E:weak_time_l} for $l=0$, and integrating in time, we obtain, 
\begin{align}\label{E:parabolic_bounds}
\|{\scriptstyle \Jbark^{1/2}}q^m(t)\|^2_0 + \int_0^t \|{\scriptstyle \Jbark^{1/2}}\nabla_{\psibark}q^m(s)\|^2_0 + \k^{-2}|{\scriptstyle \Jbark^{1/2}} q^m|^2_0 ds \leq \|\Qko^m\|^2_0 + \mathcal{R},
\end{align}
where the error $\mathcal{R}$ is given by, 
\begin{align*}
\mathcal{R} := \int_0^t\left[\int_\Omega \g_t(\Jbark)(q^m)^2 - \int_\Gamma \Jbark \beta^m q^m d\sigma + \int_\Omega\labark^i_j q^m,_i\psibark^j_t q^m + \int_\Omega \Jbark\alpha^m q^m\right]ds.
\end{align*}
Using Sobolev and Cauchy-Schwarz inequality for the terms of $\mathcal{R}$, we have that, 
\begin{align}
\mathcal{R} &\leq \int_0^t \left[\|\Jbark{}_t\|_{L^\infty}\|q^m\|^2_0 + |\beta^m|_0|q^m|_{L^2(\Gamma)}\right.\nonumber\\
&\left.\qquad + \|\Jbark\|_{L^\infty}\|\psibark_t\|_{L^\infty}\|\nabla_{\psibark}q^m\|_0\|q^m\|_0 + \|\Jbark\|_{L^\infty}\|\alpha^m\|_0\|q^m\|_0\right]ds\nonumber\\
&\leq C(M,q_0) + C\sqrt{T^m_\k}(\|q^m\|^2_{L^2_tH^1} + \|q^m\|^2_{L^\infty_tL^2}),\label{E:R_bounds_parabolic}
\end{align} 
where in the last inequality we used Young's inequality, trace estimates, and that by definition, $\alpha^m$ and $\beta^m$ are bounded by a function of the initial data $C(M,q_0)$. Finally, estimates for $\abark$ from Lemma \ref{A_bounds}, gives us that for small enough time (that does not depend on $m$),
\begin{align*}
\|q^m\|_1^2 &\leq \|q^m\|^2_0 + \|\nabla_{\psibark} q^m\|^2_0 + \|A-I\|^2_{L^\infty}\|q^m\|^2_1\\
& \leq  \|q^m\|^2_0 + \|\nabla_{\psibark} q^m\|^2_0 + \epsilon\|q^m\|^2_1,
\end{align*}
therefore $\|q^m\|^2_1 \lesssim \|q^m\|^2_0 + \|\nabla_{\psibark} q^m\|^2_0$, and combining \eqref{E:R_bounds_parabolic} with \eqref{E:parabolic_bounds} along with the estimates for $\Jbark$ from \ref{gJbounds}, we obtain that for a time $T^m_\k\leq 1/4C^2$, 
\begin{equation}\label{E:aux_parabolic_estimate}
\|q^m\|^2_{L^\infty_tL^2} + \|q^m\|^2_{L^2_tH^1} + \k^{-2}|q^m|^2_{L^2_t L^2}\leq C(M,q_0).
\end{equation}
Note that since the previous estimate holds up to any $T_\k\leq 1/4C^2$, if $T_k^m < T_k$, we can extend the solution $q^m(t)$ past $T_\k^m$ all the way to $T_\k$, while still satisfying the bound \eqref{E:aux_parabolic_estimate}.  
\end{proof}

\begin{lemma}\label{L:xmtoxm}
For a small enough time $T_\k>0$, the operator $\Phi_\k$ defined in \eqref{E:ALE_ht} is a well defined function from $X_M^\k$ to itself. 
\end{lemma}
\begin{proof}
From the definition of $\Phi_\k$ \eqref{E:ALE_ht}, the definition of $v^\pm$ \eqref{E:ALEregularized_linear_v}, and the regularization of $\bar{h}^\k$ \eqref{E:bar_h}, we have that $h$ satisfy the initial data of $X_M^\k$.
The parabolic estimates obtained for the solution $q^\pm$ of the linear problem \eqref{E:ALEregularized_linear} in Section \ref{S:higherregularitylinear} are the key ingredients to show that $\Phi_\k(\bar{h})=h\in X_M^\k$. We have for example, by the estimates for $\ak^\pm$ from Lemma \ref{A_bounds},
\begin{align*}
|h_{ttt}|_1 &= |\g_t^2[v\cdot \tilde{n}]^+_-|_1 \leq |\nabla_{\psibark^+}q^+_{tt}\cdot\tilde{n}|_1 + |\ak_{tt}^\top\nabla q^+\cdot \tilde{n}|_1 + |\nabla_{\psibark^+}q^+\cdot \tau \frac{\t \bar{h}^\k_{tt}}{(1+H\bar{h}^\k)}|_1\\
&\qquad + |\nabla_{\psibark^-}q^-_{tt}\cdot\tilde{n}|_1 + |\ak_{tt}^\top\nabla q^-\cdot \tilde{n}|_1 + |\nabla_{\psibark^-}q^-\cdot \tau \frac{\t \bar{h}^\k_{tt}}{(1+H\bar{h}^\k)}|_1 + \text{l.o.t.}\\
&\lesssim \|q^+_{tt}\|_{2.5} +\|q^-_{tt}\|_{2.5} + |\bar{h}^\k_{tt}|_2(\|q^+\|_3 + \|q^-\|_3 )+ \text{l.o.t.} 
\end{align*}
Therefore, interpolating: $\|q_{tt}^\pm\|_{2.5} \leq \|q_{tt}^\pm\|_{2}^{1/2}\|q_{tt}^\pm\|_3^{1/2}$, and using the parabolic estimates for $q^\pm$, we obtain,
\begin{align*}
|h_{ttt}|_{L^2_tH^1(\Gamma)} &\lesssim \sqrt{t}(\|q^+_{tt}\|_{L^\infty H^2}^{1/2}\|q^+_{tt}\|_{L^2H^3}^{1/2} + \|q^-_{tt}\|_{L^\infty H^2}^{1/2}\|q^-_{tt}\|_{L^2H^3}^{1/2})\\ &\qquad + t |\bar{h}_{tt}^\k|_{L^\infty H^2}(\|q^+\|_{L^\infty H^3}+\|q^-\|_{L^\infty H^3}) + \text{l.o.t.}\\
&\lesssim \sqrt{t}C(M_0) + t M_0 M\\
&\leq \sqrt{t}M,  
\end{align*}
where $C(M_0)$ is a constant function of the initial data, and the last inequality follows by choosing $M$ so that $M\geq M_0$ and $T_\k$ small enough so that $\sqrt{t}M_0\leq 1,\ \forall t\leq T_\k$.  
A similar estimate can be done for $h_t\in H^5(\Gamma)$,
\begin{align*}
|h_t|_5 &\leq C(\|v^+\|_{5.5}+\|v^-\|_{5.5})(1+|\bar{h}^\k|_6)\\
&\leq C(\|q^+\|_{6.5} + \|q^-\|_{6.5})(1+|\bar{h}^\k|_6)^2\\
&\leq C(\|q^+\|_{6.5} + \|q^-\|_{6.5})(1+|h^\k_0|_6 + \sqrt{T_\k}|\bar{h}^\k_t|_{L^2H^6})^2\\
&\leq C(\|q^+\|_{6.5} + \|q^-\|_{6.5})(1+|h^\k_0|_6 + \k^{-1}\sqrt{T_\k}|\bar{h}^\k_t|_{L^2H^5})^2,
\end{align*}
where on the third line we used the definition of $\bar{h}^\k$ and Cauchy-Schwarz inequality to get the term $\sqrt{T_\k}|\bar{h}^\k_t|_{L^2H^6}$. The next line follows from absorving one of the derivatives of the $H^6$ norm in exchange for the $\k^{-1}$ coefficient due to the tangential convolution structure of $\bar{h}_t^\k$. Taking the $L^2_t$ norm in time we obtain,
 
\begin{align*}
|h_t|_{L^2H^5} &\leq C(\|q^+\|_{L^2H^{6.5}} + \|q^-\|_{L^2 H^{6.5}})(1+|h^\k_0|_6 + \k^{-1}\sqrt{T_\k}|\bar{h}^\k_t|_{L^2H^5})^2\\
&\leq C(M_0)(1+|h^\k_0|_6 + \k^{-1}\sqrt{T_\k}|\bar{h}^\k_t|_{L^2H^5})^2
\end{align*}
Choosing $M\geq C(M_0)(2+|h_0|_6)^2$ and $T_\k$ small enough so that $\k^{-1}\sqrt{T_\k}M\leq 1$ we obtain the desired inequality. The bounds for the other norms of $h$ in the definition of $X_M^\k$ follow in similar fashion, so we will omit their proof for brevity.

\end{proof}

\begin{remark}
Notice that the time of existence $T_\k$ depends on $M$, which is a function of the initial data. 
\end{remark}

\begin{lemma}\label{L:sourceF} The source function $f$ defined by~\eqref{E:SOURCEF} satisfies the bound:
\begin{equation}\label{E:f_bound}
\snorm{f}_{L^2 H^{0.5}}\leq C_M\k^{-1}\sqrt{T_\k}\ \mathcal{S}(\delta q,\delta\bar{h}^\k)^{1/2},
\end{equation}
where $\mathcal{S}(\delta q, \delta\bar{h}^\k)$ is the norm defined in \eqref{S-local} evaluated in the differences $\delta q$ and $\delta\bar{h}^\k$.
\end{lemma}

\begin{proof}

First, let us identify the higher order terms of $f$,
\begin{align}
f &= \abark_1^i{}_j{}_{tt}(\abark_1{}^k_j\delta q,_k),_i - \delta v_{tt}\cdot \bar{w}_{1\k}- \delta v\cdot\bar{w}_{1\k}{}_{tt}\nonumber\\
&\quad - v_2{}_{tt}\cdot\delta\psibark_t - v_2\cdot \delta\psibark_{ttt} + \delta\ak_{tt}{}^i_j(\abark_1{}^k_j q_2,_k),_i\nonumber\\
&\quad + \delta \abark^i_j(\abark_1{}^k_j q_2{}_{tt},_k),_i + \text{l.o.t.},
\end{align} 
where we group together in ``l.o.t.'' all the lower order terms that can be bounded as in equation \eqref{E:f_bound} via a simple application of Cauchy-Schwarz inequality and where the norms are directly bouded by the norm $\mathcal{S}(\delta q, \delta h)$.
The first term in the definition~\eqref{E:SOURCEF} of $f$ satisfies the bound, 

\begin{align*}
\|\ak_1^i{}_j{}_{tt}(\ak_1{}^k_j\delta q,_k),_i\|_{L^2_tH^{1}} &\leq C_M \|\nabla\psibark^1_{tt} \nabla^2\delta q\|_{L^2_tH^{1}} + \text{l.o.t.}\\
& \leq C_M |\bar{h}^\k_{tt}{}^1|_{L^2_tH^{2}}\|\delta q\|_{L^\infty_t H^{3.5}}\\
&\leq C_M \sqrt{T_\k} |\bar{h}^\k_{tt}{}^1|_{L^\infty_tH^{2}}\|\delta q\|_{L^\infty_t H^{3.5}}\\
&\leq C_M \sqrt{T_\k}\|\delta q\|_{L^\infty H^{3.5}}, 
\end{align*}
where we used Cauchy-Schwarz inequality in the time integration to obtain the coefficient $\sqrt{T_\k}$. The second term can be bounded by,

\begin{align*}
\|\delta v_{tt}\cdot \bar{w}_{1\k}\|_{1} &\leq \|\delta v_{tt}\|_1 \|\psibark^1_t\|_{L^\infty} + \|\delta v_{tt}\|_{L^4}\|\nabla\psibark^1_t\|_{L^4}\\
&\leq C_M(\|\nabla\psibark^1_{tt}\|_1\|\delta q\|_{2.5} + \|\delta q_{tt}\|_2)|\bar{h}^\k_t{}^1|_1 + \text{l.o.t.}\\
&\leq C_M(|\bar{h}^\k_{tt}{}^1|_{1.5}\|\delta q\|_{2.5} + \|\delta q_{tt}\|_2)|\bar{h}^\k_t{}^1|_1 + \text{l.o.t.},
\end{align*}
therefore, 
\begin{equation*}
\|\delta v_{tt}\cdot \bar{w}_{1\k}\|_{L^2_tH^1} \leq C_M\sqrt{T_\k}(\|\delta q\|_{L^\infty_t H^{2.5}} + \|\delta q_{tt}\|_{L^\infty_t H^2}).
\end{equation*}
The third term can be bounded by, 
\begin{align*}
\|\delta v\cdot \psibark^1_{ttt}\|_{L^2_tH^1} &\leq C\|\delta v\|_{L^\infty_tH^{1.5}}|\bar{h}^\k_{ttt}{}^1|_{L^2H^1}\\
&\leq C\k^{-1}\|\delta v\|_{L^\infty_tH^{1.5}}|\bar{h}^\k_{ttt}{}^1|_{L^2_tL^2}\\
&\leq C\k^{-1}\sqrt{T_\k}\|\delta v\|_{L^\infty_tH^{1.5}}|\bar{h}^\k_{ttt}{}^1|_{L^\infty_tL^2}\\
&\leq C_M \k^{-1}\sqrt{T_\k}\|\delta v\|_{L^\infty_tH^{1.5}},
\end{align*}
where we used the smoothing of $\bar{h}^\k_{ttt}{}^1$ in line 2, to lower the Sobolev norm in exchange for the coefficient $\k^{-1}$.
The fourth term can be bounded by, 
\begin{align*}
\|v_2{}_{tt}\cdot \delta\psibark_t\|_{L^2_tH^1} &\leq C\|v_2{}_{tt}\|_{L^2H^1}\|\delta\psibark_t\|_{1.5}\\
&\leq C_M(\|\psibark^2_{tt}\|_{L^2_tH^2}\|q_2\|_{L^\infty_tH^{2.5}}+\|q_2{}_{tt}\|_{L^2_tH^2})|\delta \bar{h}^\k_t|_{L^\infty_tH^1} + \text{l.o.t.}\\
&\leq C_M\sqrt{T_\k}(|\bar{h}^\k_{tt}{}^2|_{L^\infty_tH^{1.5}}\|q_2\|_{L^\infty_tH^{2.5}} + \|q_2{}_{tt}\|_{L^\infty_tH^2})|\delta\bar{h}^\k_t|_{L^\infty_tH^1} + \text{l.o.t.}\\
&\leq C_M\sqrt{T_\k}|\delta\bar{h}^\k_t|_{L^\infty_tH^1}.
\end{align*}

The fifth term,
\begin{align*}
\|v_2\cdot \delta\psibark_{ttt}\|_{L^2_tH^1} &\leq C \|v_2\|_{L^\infty_tH^{1.5}}\|\delta\psibark_{ttt}\|_{L^2_tH^1}\\
&\leq C\|v_2\|_{L^\infty_tH^{1.5}}|\delta \bar{h}^\k_{ttt}|_{L^2_tH^1}\\
&\leq C\k^{-1}\|v_2\|_{L^\infty_tH^{1.5}}|\delta \bar{h}^\k_{ttt}|_{L^2_tL^2}\\
&\leq C_M \k^{-1}\sqrt{T_\k}|\delta\bar{h}^\k_{ttt}|_{L^\infty_tL^2}.
\end{align*}

The sixth term, 
\begin{align*}
\|\delta \ak_{tt}{}^i_j(\ak_1{}^k_j q_2,_k),_i\|_{L^2_tH^1} &\leq \|\delta \ak_{tt} \ak_1 \nabla^2 q_2\|_{L^2_tH^1} + \text{l.o.t.}\\
&\leq C_M\|\delta\ak_{tt}\|_{L^2_tH^1}\|q_2\|_{L^\infty_tH^{3.5}} + \text{l.o.t.}\\
&\leq C_M \sqrt{T_\k}|\delta\bar{h}^\k_{tt}|_{L^\infty_tH^{1.5}}.
\end{align*}

The last and most critical term need to be analyzed in the actual $H^{0.5}$ Sobolev norm, 
\begin{align*}
\|\delta\ak^i_j(\ak_1{}^k_j q_2{}_{tt},_k),_i\|_{L^2_tH^{0.5}} &\leq C_M\|\delta \ak\|_{L^\infty_t H^{1.5}}\|q_2{}_{tt}\|_{L^2_tH^{2.5}} + \text{l.o.t.}\\
&\leq C_M\sqrt{T_\k}|\delta \bar{h}^\k_t|_{L^2_tH^2}\|q_2{}_{tt}\|_{L^2_tH^{2.5}}\\
&\leq C_M\sqrt{T_\k}|\delta \bar{h}^\k_t|_{L^2_tH^2},
\end{align*}
where we used the bounds for $\delta\ak$ from lemma \ref{L:delta_A_bound}, and that $\|q_2{}_{tt}\|_{L^2_tH^{2.5}}\leq \mathcal{S}(q_2)\leq C_M$. The proof then follows from collecting all the terms together along with the straightforward estimates of the lower-order terms.

\end{proof}

\begin{lemma}[Error terms]\label{S:error_terms}
The error terms from Lemma \ref{ab_energy_lemma} are given by,

\begin{equation*}
\mathcal{R}(t) = \mathcal{R}^+(t) + \mathcal{R}^-(t) + \mathcal{R}_\Gamma(t) + \mathcal{R}^+_{\g\Omega}(t) + \mathring{\mathcal{R}}^+(t) + \mathring{\mathcal{R}}^-(t), 
\end{equation*}
where,
\begin{align*}
\mathcal{R}^\pm &= \sum\limits_{a+2b\leq 6} \mathcal{R}^\pm_\aabb + \sum\limits_{a+2b\leq 5}\widetilde{\mathcal{R}}^\pm_\aabb,\\
\mathcal{R}_\Gamma  &= \sum\limits_{a+2b\leq 6} \mathcal{R}_\Gamma^\aabb + \sum\limits_{a+2b\leq 5}\widetilde{\mathcal{R}}_\Gamma^\aabb,\\
\mathcal{R}^+_{\g\Omega} &= \sum\limits_{a+2b\leq 6}\mathcal{R}^\aabb_{\g\Omega} + \sum\limits_{a+2b\leq 5} \widetilde{\mathcal{R}}_{\g\Omega}^\aabb,\\
\mathring{\mathcal{R}}^\pm &= \sum\limits_{a+2b\leq 6} \mathring{\mathcal{R}}^\pm_\aabb + \sum\limits_{a+2b\leq 5}\widetilde{\mathcal{R}}^{\circ \pm}_\aabb
\end{align*}
with,
\begin{align*}
\mathcal{R}^\pm_\aabb &= \sum_{1\le s\le a-1\atop 1\le l\le b-1} \int_{\Omega^\pm}c_{sl}\t^{a-s}\partial_t^{b-l}A^T\left(-\t^s\partial_t^l \nabla q+ \t^s\partial_t^l \nabla\Psi A \nabla q\right) \t^a\partial_t^b v W\mu\\
& + \int_{\Omega^\pm}\t^a\partial_t^b\Psi A \nabla v\t q\partial_t^b v W\mu \\
& - \int_{\Omega^\pm}(\t^a\partial_t^b q + \t^a\partial_t^b\Psi\cdot v )\left( \t^a\partial_t^b A  \nabla v + \sum_{1\le s\le a-1\atop 1\le l\le b-1} c_{sl}\t^{a-s}\partial_t^{b-l}A\t^s\partial_t^l \nabla v\right)W\mu\\
& - \int_{\Omega^\pm}(\t^a\partial_t^b q + \t^a\partial_t^b\Psi\cdot v )\left( -\t^a\partial_t^b\Psi\cdot v_t + \Psi_t\cdot\t^a\partial_t^b v + \sum_{1\le s\le a-1\atop 1\le l\le b-1} c_{sl}\t^{a-s}\partial_t^{b-l}\Psi_t\cdot \t^s\partial_t^l v\right)W\mu\\
&\quad + \int_{\Omega^\pm}(\t^a\g_t^b q + \t^a\g_t^b\psik\cdot v)(\t^a\g_t^b\alpha) W\mu\\
&+ \frac{1}{2} \int_{\Omega^\pm}(\t^a\partial_t^b q + \t^a\partial_t^b\Psi\cdot v)^2 W_t + \int_{\Omega^\pm}(\t^a\partial_t^b q + \t^a\partial_t^b\Psi\cdot v)A\t^a\partial_t^b v  \nabla (W\mu),
\end{align*}
\begin{align*}
\widetilde{\mathcal{R}}^\pm_\aabb &=  - \sum_{1\le s\le a-1\atop 1\le l\le b-1} c_{sl} \int_{\Omega^\pm}\t^{a-s}\g_t^{b+1-l}A^i_j \t^s\g_t^l q,_i\t^a\g_t^b v^j W\mu\\
&+ \int_{\Omega^\pm}(\Psi_t^k,_l \t^a\g_t^b (A^i_kA^l_j) + \sum_{1\le s\le a-1\atop 1\le l\le b-1} c_{sl} \t^{a-s}\g_t^{b-l}\Psi_t^k,_l\t^s\g_t^l(A^i_kA^l_j))q,_i\t^a\g_t^b v^j W\mu\\
&+ \int_{\Omega^\pm}\t^a\g_t^{b+1}\Psi^k v^k,_l A^l_j\t^a\g_t^bv^j W\mu\\
&- \int_{\Omega^\pm}(\t^a\g_t^{b+1} q + \t^a\g_t^{b+1}\Psi\cdot v)\left( \Psi_t\cdot \t^a\g_t^{b}v + \t^a\g_t^{b}A^i_j v^j,_i + \sum_{1\le s\le a-1\atop 1\le l\le b-1} c_{sl} \left( \t^{a-s}\g_t^{b-l}\Psi_t\cdot \t^s\g_t^l v \right.\right.\\
&\left.\left.  +\t^{a-s}\g_t^{b-l} A^i_j \t^s\g_t^l v^j,_i\right)\right)W\mu\\
& + \int_{\Omega^\pm}(\t^a\g_t^{b+1}q + \t^a\g_t^{b+1}\psik\cdot v)(\t^a\g_t^b \alpha)W\mu\\
&+\frac{1}{2}\int_{\Omega^\pm}(\t^a\g_t^b v)^2 W_t\mu + \int_{\Omega^\pm}(\t^a\g_t^{b+1} q + \t^a\g_t^{b+1}\Psi\cdot v)(A^i_j \t^a\g_t^{b} v^j)(W\mu),_i,
\end{align*}
\begin{align*}
\mathcal{R}_\Gamma^\aabb &= \frac{1}{2}\int_\Gamma \g_t\left(e^{(-\lambda_1 +\eta)t} a_\k^2\right)(\t^a\g_t^b\ \Lambda_\k h)^2\\
& -e^{(-\lambda_1 +\eta)t}\int_\Gamma J_\k^{-2} (1+Hh^\k)((h^\k \t^a\g_t^b N + \t^a\g_t^b x + \sum_{1\le s\le a-1\atop 1\le l\le b-1}\t^{a-s}\g_t^{b-l} h^\k \t^s\g_t^l N) \cdot (-\t h^\k\tau\\
& + (1+Hh^\k)N))[\t^a\g_t^b v\cdot \tilde{n}^\k]^+_-\\
& + e^{(-\lambda_1 +\eta)t}\int_\Gamma a_\k^2\t^a\g_t^{b} h^\k \left([v\cdot \t^a\g_t^b\tilde{n}^\k]^+_- + \sum_{1\le s\le a-1\atop 1\le l\le b-1}c_{sl}[\t^{a-s}\g_t^{b-l} v\cdot \t^s\g_t^l \tilde{n}^\k]^+_-\right)\\
& -e^{(-\lambda_1 +\eta)t}\int_\Gamma \t^a\g_t^b(\Lambda_\k h)[\Lambda_\k, a_\k^2\t^a\g_t^b]h_t\\
& - \k^2 e^{(-\lambda_1 +\eta)t} \int_\Gamma \left(r^+_\k \t^a\g_t^b \beta^+ (\t^a\g_t^b v^+\cdot n^\k) + r^-_\k \t^a\g_t^b \beta^- (\t^a\g_t^b v^-\cdot n^\k)\right)\\
& - \k^2 e^{(-\lambda_1+\eta)t}\sum\limits_{1\leq s\leq a\atop 1\leq l\leq b} c_{sl} \int_\Gamma \left((\g_N q^+)^{-1}\t^{a-s}\g_t^{b-l} v^+ \cdot \t^s\g_t^l (A^\top N)\t^a\g_t^b v^+\cdot A^\top N \right.\\
&\qquad \left. + (\g_N q^-)^{-1}\t^{a-s}\g_t^{b-l} v^- \cdot \t^s\g_t^l (A^\top N)\t^a\g_t^b v^-\cdot A^\top N \right)\\
&  -\k^2e^{(-\lambda_1+\eta)t}\int_\Gamma \left(\t^a\g_t^b\psik^i {}^\k \hspace{-.06in} A^l_i \tau^l \t( v^+\cdot {}^\k \hspace{-.06in} A^\top N)(\t^a\g_t^b v^+\cdot n^\k)J_\k^{-1}\sqrt{g_\k}(\g_N q^+)^{-1} \right.\\
&\qquad \qquad \left.+ \t^a\g_t^b\psik^i {}^\k \hspace{-.06in} A^l_i \tau^l \t( v^-\cdot {}^\k \hspace{-.06in} A^\top N)(\t^a\g_t^b v^-\cdot n^\k)J_\k^{-1}\sqrt{g_\k}(\g_N q^-)^{-1}\right),
\end{align*}
\begin{align*}
\widetilde{\mathcal{R}}_\Gamma^\aabb &= -e^{(-\lambda_1 + \eta)t}\int_\Gamma J_\k^{-2}(1+Hh^\k)(h_t^\k\t^a\g_t^b N + \sum_{1\le s\le a-1\atop 1\le l\le b-1} c_{sl}\t^{a-s}\g_t^{b-l}h_t^\k\t^s\g_t^l N)\cdot(-\t h^\k\tau\\
& + (1+H h^\k)N)[\t^a\g_t^{b+1}v\cdot n]^+_-\\
& + e^{(-\lambda_1 + \eta)t}\int_\Gamma a_\k^2\t^a\g_t^{b+1} h^\k\left([v\cdot\t^a\g_t^b\tilde{n}^\k]^+_- + \sum_{1\le s\le a-1\atop 1\le l\le b-1} c_{sl}[\t^{a-s}\g_t^{b-s}v\cdot \t^l\g_t^s\tilde{n}^\k]^+_-\right)\\
& + e^{(-\lambda_1 + \eta)t}\int_\Gamma\t^a\g_t^{b+1}\Lambda_\k h [a_\k^2\t^a, \Lambda_\k]\g_t^{b+1}h\\
& + \k^2e^{(-\lambda_1 + \eta)t}\int_\Gamma \left(\t^a\g_t^{b+1} \psik {}^\k \hspace{-.06in} A^i_k \tau^i \t(v^+\cdot {}^\k \hspace{-.06in} A^\top N)\t^a\g_t^{b} v^+\cdot {}^\k \hspace{-.06in} A^\top N (\g_Nq^+)^{-1} \right.\\
&\qquad\qquad  \left.+ \t^a\g_t^{b+1} \psik {}^\k \hspace{-.06in} A^i_k \tau^i \t(v^-\cdot {}^\k \hspace{-.06in} A^\top N)\t^a\g_t^{b} v^-\cdot {}^\k \hspace{-.06in} A^\top N (\g_Nq^-)^{-1} \right)\\
& - \k^2 e^{(-\lambda_1+\eta)t} \sum\limits_{s,l} c_{sl} \int_\Gamma \left((\g_t(\t^{a-l}\g_t^{b-s} v^+\cdot \t^l\g_t^s (A^\top N)) - \t^a\g_t^{b+1}\beta^+)(\t^a\g_t^b v^+\cdot A^\top N) (\g_N q^+)^{-1} \right.\\
&\qquad\qquad  - \left. (\g_t(\t^{a-l}\g_t^{b-s} v^-\cdot \t^l\g_t^s (A^\top N)) - \t^a\g_t^{b+1}\beta^-)(\t^a\g_t^b v^-\cdot A^\top N) (\g_N q^-)^{-1} \right)\\
& + \frac{\k^2}{2}\int_\Gamma \left((\t^a\g_t^b v^+\cdot A^\top N)^2\g_t(e^{(-\lambda_1 +\eta)t}(\g_N q^+)^{-1}) + (\t^a\g_t^b v^-\cdot A^\top N)^2\g_t(e^{(-\lambda_1 +\eta)t}(\g_N q^-)^{-1})\right)
\end{align*}
\begin{align*}
\mathcal{R}_{\g\Omega}^\aabb &= -\int_{\g\Omega}  (\t^a\partial_t^b q^+ + \t^a\g_t^b\psik\cdot v) (\t^a\g_t^b \gamma - \sum\limits_{s,l}c_{sl}\t^{a-l}\g_t^{b-s} v\cdot \t^l\g_t^s {\bf N}^+) W,\\
\widetilde{\mathcal{R}}_{\g\Omega}^\aabb &= \int_{\g\Omega}(\t^a\g_t^{b+1} q^+ + \t^a\g_t^{b+1}\Psi^+\cdot v^+)(\t^a\g_t^b \gamma - \sum\limits_{s,l}c_{sl}\t^{a-l}\g_t^{b-s} v\cdot \t^l\g_t^s {\bf N}^+)W,\\
\mathring{\mathcal{R}}^\pm_\aabb &= \sum_{0\leq s^i\leq a^i\atop{1\le |s|\le |a|-1\atop 1\le l\le b-1}}c_{sl}\int_{\Omega^\pm}(1-\mu)\left(-\g^{a-s}\g_t^{b-l}A^j_i\g^s\g_t^l q,_j  + \g^{a-s}\g_t^{b-l}A^j_k \g^s\g_t^l \Psi^k,_l A^l_i q,_j \right)\g^a\g_t^b v W\\
&-\int_{\Omega^\pm} (1-\mu)\g^a\g_t^b\Psi^k v^k,_l A^l_i\g^a\g_t^b v^i W\\
&+ \int_{\Omega^\pm} (\g^a\g_t^b q + \g^a\g_t^b\Psi\cdot v)\left[ ((1-\mu)W),_j A^j_i\g^a\g_t^b v^i + (1-\mu)W\left( (\g^a\g_t^b A^j_i v^i),_j\right.\right.\\
&\left.\left. + \sum_{0\leq s^i\leq a^i\atop{1\le |s|\le |a|-1\atop 1\le l\le b-1}} c_{sl} (\g^{a-s}\g_t^{b-l}A^j_i \g^s\g_t^l v^i),_j\right)\right]\\
&+ \int_{\Omega^\pm}(1-\mu)(\g^a\g_t^b q + \g^a\g_t^b\Psi\cdot v)(\g^a\g_t^b\Psi\cdot v_t) W\\
& + \int_{\Omega^\pm} (\t^a\g_t^b q + \t^a\g_t^b\psik\cdot v)(\t^a\g_t^b \alpha) (1-\mu)W
\end{align*}
\begin{align*}
\widetilde{\mathcal{R}}^{\circ\pm}_\aabb &= -\sum_{0\leq s^i\leq a^i\atop{1\le |s|\le |a|-1\atop 1\le l\le b-1}} c_{sl}\int_{\Omega^\pm}(1-\mu)\g^{a-s}\g_t^{b-l}A^i_j\g^s\g_t^l q,_i\g^a\g_t^b v^j W\\
&+ \int_{\Omega^\pm}\left(\Psi_t^k,_l \g^a\g_t^b (A^i_kA^l_j) + \sum_{0\leq s^i\leq a^i\atop{1\le |s|\le |a|-1\atop 1\le l\le b-1}} c_{sl} \g^{a-l}\g_t^{b-s}\Psi_t^k,_l\g^l\g_t^s(A^i_kA^l_j)\right)q,_i\g^a\g_t^b v^j (1-\mu)W\\
&+ \int_{\Omega^\pm}\g^a\g_t^{b+1}\Psi^k v^k,_l A^l_j\g^a\g_t^bv^j (1-\mu)\\
&+ \int_{\Omega^\pm}(\g^a\g_t^{b+1} q + \g^a\g_t^{b+1}\Psi\cdot v)(A^i_j \g^a\g_t^{b} v^j)((1-\mu)W),_i\\
&- \int_{\Omega^\pm}(\g^a\g_t^{b+1} q + \g^a\g_t^{b+1}\Psi\cdot v)\left( \Psi_t\cdot \g^a\g_t^{b}v + \g^a\g_t^{b}A^i_j v^j,_i + \sum_{0\leq s^i\leq a^i\atop{1\le |s|\le |a|-1\atop 1\le l\le b-1}} c_{sl} \left( \g^{a-s}\g_t^{b-l}\Psi_t\cdot \g^s\g_t^l v \right.\right.\\
&\left.\left.  +\g^{a-s}\g_t^{b-l} A^i_j \g^s\g_t^l v^j,_i\right)\right)(1-\mu)W\\
& +\int_{\Omega^\pm}(\g^a\g_t^{b+1} q + \g^a\g_t^{b+1}\Psi\cdot v) (\t^a\g_t^b\alpha)(1-\mu)W,
\end{align*}
where we have omitted the upper indices ${}^\pm$ inside the interior integrals for simplicity of notation, but it is assumed that the functions $q^\pm,\ \Psi^\pm,\ A^\pm,\ v^\pm$, $W^\pm$, 
are being integrated over the corresponding region $\Omega^\pm$. 
\end{lemma}

\section*{Acknowledgements} 
S.S. was supported by the National Science Foundation under grant no. DMS-1301380 and by a Royal
Society Wolfson Merit Award.

\bibliographystyle{plain}
\def\cprime{$'$}

\end{document}